\documentclass[11pt]{amsart}
\pdfoutput=1
\usepackage[utf8]{inputenc}
\usepackage{amsfonts,amssymb}
\usepackage{amsmath,mathabx}
\usepackage{graphicx}
\usepackage{array}
\usepackage{enumitem}
\usepackage{url}
\usepackage{xargs}
\usepackage{etoolbox}

\DeclareFontFamily{U}{mathb}{}
\DeclareFontShape{U}{mathb}{m}{n}{ <5> <6> <7> <8> <9> <10> <12> gen * mathb <11> mathb10}{}
\DeclareSymbolFont{mathb}{U}{mathb}{m}{n}
\DeclareMathSymbol{\boxslash}     {2}{mathb}{"6D}
\DeclareMathSymbol{\smalltriangleup}   {2}{mathb}{"98}

\DeclareFontFamily{U}{FdSymbolA}{}
\DeclareFontShape{U}{FdSymbolA}{m}{n}{
    <-7.1> s * [1.0] FdSymbolA-Regular
    <7.1-> s * [1.0] FdSymbolA-Book
}{}
\DeclareFontShape{U}{FdSymbolA}{b}{n}{
    <-7.1> s * [1.0] FdSymbolA-Regular
    <7.1-> s * [1.0] FdSymbolA-Book
}{}
\DeclareSymbolFont{fdsymbols}{U}{FdSymbolA}{m}{n}
\DeclareMathSymbol{\smalltriangleup}   {2}{fdsymbols}{"4F}
\DeclareMathSymbol{\medtriangleup}   {2}{fdsymbols}{"57}
\DeclareMathSymbol{\largetriangleup}   {2}{fdsymbols}{"5E}
\DeclareMathSymbol{\smallboxslash}   {2}{fdsymbols}{"79}

\usepackage{empheq}
\numberwithin{equation}{section}

\usepackage{amsthm}
\theoremstyle{plain}
\newtheorem{introtheorem}{Theorem}{\bf}{\it}

{\bf}{\it}

\newtheorem{thm}{Theorem}[section]
\newtheorem{lem}[thm]{Lemma}
\newtheorem{cor}[thm]{Corollary}
\newtheorem{prop}[thm]{Proposition}

\theoremstyle{definition}
\newtheorem{defn}[thm]{Definition}
\newtheorem{exmp}[thm]{Example}
\newtheorem{rem}[thm]{Remark}
\newtheorem{convt}[thm]{Convention}

\AfterEndEnvironment{introtheorem}{\noindent\ignorespaces}
\AfterEndEnvironment{thm}{\noindent\ignorespaces}
\AfterEndEnvironment{lem}{\noindent\ignorespaces}
\AfterEndEnvironment{cor}{\noindent\ignorespaces}
\AfterEndEnvironment{prop}{\noindent\ignorespaces}
\AfterEndEnvironment{obs}{\noindent\ignorespaces}
\AfterEndEnvironment{fact}{\noindent\ignorespaces}
\AfterEndEnvironment{defn}{\noindent\ignorespaces}
\AfterEndEnvironment{exmp}{\noindent\ignorespaces}
\AfterEndEnvironment{rem}{\noindent\ignorespaces}
\AfterEndEnvironment{convt}{\noindent\ignorespaces}
\AfterEndEnvironment{proof}{\noindent\ignorespaces}
\makeatletter
\let\@afterindenttrue\@afterindentfalse
\@afterindentfalse
\makeatother

\newcommand{\Z}{\mathbb{Z}}
\newcommand{\R}{\mathbb{R}}
\newcommand{\C}{\mathbb{C}}
\newcommand{\braid}{\textsc{Braid}}
\newcommand{\purebraid}{\textsc{PBraid}}

\newcommand{\sym}{\textsc{Sym}}
\newcommand{\affsym}{\widetilde{\sym}}
\newcommand{\cat}{\textsc{CAT}}
\newcommand{\conf}{\textsc{Conf}}

\newcommand{\uconf}{\textsc{UConf}}
\newcommand{\Prod}{\textsc{Prod}}
\newcommand{\Cube}{\raisebox{\depth}{$\Box$}}
\newcommand{\Orth}{\raisebox{\depth}{$\boxslash$}}

\newcommand{\onto}{\twoheadrightarrow}
\newcommand{\into}{\hookrightarrow}
\newcommand{\diag}{\textsc{Diag}}

\newcommand{\fix}{\textsc{Fix}}
\newcommand{\conv}{\textsc{Conv}}
\newcommand{\bool}{\textsc{Bool}}
\newcommand{\cube}{\textsc{Cube}}

\newcommand{\comp}{\textsc{Cplx}}
\newcommand{\cay}{\operatorname{Cay}}
\newcommand{\NC}{\operatorname{NC}}
\newcommand{\DS}{\operatorname{DS}}
\newcommand{\NP}{\operatorname{NP}}

\newcommand{\flag}{\textsc{Flag}}
\newcommand{\bdry}{\textsc{bdry}}
\newcommand{\cov}{\textsc{cov}}
\newcommand{\splt}{\textsc{split}}
\newcommand{\move}{\textsc{Move}}
\renewcommand{\mod}{\mathrel{\operatorname{mod}}}

\newcommand{\defeq}{\mathrel{\mathop{:}}=}

\newcommand{\abs}[1]{\lvert #1 \rvert}
\newcommand{\norm}[1]{\lVert #1 \rVert}

\newcommand{\rk}{\operatorname{rk}}

\newcommand{\zero}{\mathbf{0}}
\newcommand{\one}{\mathbf{1}}
\newcommand{\card}[1]{\lvert{#1}\rvert}
\newcommand{\oprod}{\mathbin{\smallboxslash}}

\newcommand{\nrm}{\textsc{norm}}
\newcommand{\vx}{\vec{x}}
\newcommand{\vv}{\vec{v}}
\newcommand{\vp}{\vec{p}}
\newcommand{\vz}{\vec{z}}
\newcommand{\D}{\mathbb{D}}
\newcommand{\sph}{\mathbb{S}}
\newcommand{\set}{\textsc{set}}
\newcommand{\perm}{\textsc{perm}}
\newcommand{\proj}{\textsc{proj}}
\newcommand{\Edges}{\textsc{Edges}}

\usepackage{tikz,tikz-cd}
\usetikzlibrary{calc}
\newcommand{\drawpoints}{ 
	\foreach \n in {1,...,4} { \draw (\n) node [littledot] {}; }; }
\newcommand{\drawAngle}[3]{\path #1 -- ++#2 coordinate(temp0) ++#3
  coordinate(temp1) #1 -- ++#3 coordinate(temp2); \draw (temp0)--(temp1)--(temp2);}
\newcommand{\drawAngleD}[3]{\path #1 -- ++#2 coordinate(temp0) ++#3
  coordinate(temp1) #1 -- ++#3 coordinate(temp2); \draw[dashed] (temp0)--(temp1)--(temp2);}
\tikzstyle{BlueLine}=[line width=0.3mm,color=blue,text=black]
\tikzstyle{BluePoly}=[BlueLine,fill=blue!20]
\tikzstyle{RedLine}=[line width=0.3mm,color=red,text=black]
\tikzstyle{RedPoly}=[RedLine,fill=red!20]
\tikzstyle{GreenLine}=[thick,color=black!30!green,text=black]
\tikzstyle{GreenPoly}=[thick,color=green!50!black,fill=green!30,join=bevel]
\tikzstyle{OrangeLine}=[thick,color=orange]
\tikzstyle{GrayLine}=[thick,color=black!50!gray]
\tikzstyle{GrayPoly}=[GrayLine,fill=gray!20]
\tikzstyle{dot}=[shape=circle,draw,color=black,fill=black,inner sep=1.5pt]
\tikzstyle{bigdot}=[dot,inner sep=2pt]
\tikzstyle{littledot}=[dot,inner sep=1pt]
\tikzstyle{disk}=[thick,shape=circle,draw,color=black,fill=yellow!10]
\tikzstyle{plate}=[thick,shape=circle,draw,color=black,fill=yellow!10,
	rounded corners,minimum size=1.2cm]
\tikzstyle{dot}=[shape=circle,draw,color=black,fill=black,inner sep=1.5pt]
\tikzstyle{opendot}=[dot,fill=white, inner sep=1.5pt]
\tikzstyle{disk}=[thick,shape=circle,draw,color=black]

\newcommand{\makepent}{ \foreach \n in {1,...,5} { \coordinate (\n) at (\n*72 :0.5cm); }; }
\newcommand{\drawpentb}{ 
	\foreach \n in {2,4,5} { 
		\draw (\n) node [opendot] {}; 
	};
	\foreach \n in {1,3} { 
		\draw (\n) node [littledot] {}; 
	}; 
}

\numberwithin{equation}{section}

\begin{document}

\title{Boundary Braids}
\date{\today}
\subjclass[2010]{
	Primary 20F36; 
	Secondary 20F65.   
}

\keywords{Braid groups}

\author[M.~Dougherty]{Michael Dougherty} \address{Dept. of
  Mathematics and Statistics, Grinnell College, Grinnell, IA 50112, USA}
\email{doughert2@grinnell.edu}

\author[J.~McCammond]{Jon McCammond} \address{Dept. of Mathematics,
  UC Santa Barbara, Santa Barbara, CA 93106, USA}
\email{jon.mccammond@math.ucsb.edu}

\author[S.~Witzel]{Stefan Witzel} \address{Dept. of Mathematics,
  Bielefeld University, PO Box 100131, 33501 Bielefeld, Germany}
\thanks{The third author was funded through the DFG project WI~4079/2. He also acknowledges the hospitality of UC Santa Barbara.}
\email{switzel@math.uni-bielefeld.de}

\begin{abstract}
The $n$-strand braid group can be defined as the fundamental group of
the configuration space of $n$ unlabeled points in a closed disk based
at a configuration where all $n$ points lie in the boundary of the
disk.  Using this definition, the subset of braids that have a
representative where a specified subset of these points remain
pointwise fixed forms a subgroup isomorphic to a braid group with fewer
strands. In this article, we generalize this phenomenon by introducing
the notion of boundary braids.  A boundary braid is a braid that has a
representative where some specified subset of the points remains in the
boundary cycle of the disk.  Although boundary braids merely form a
subgroupoid rather than a subgroup, they play an interesting geometric
role in the piecewise Euclidean dual braid complex defined by Tom
Brady and the second author.  We prove several theorems in this
setting, including the fact that the subcomplex of the dual braid
complex determined by a specified set of boundary braids metrically
splits as the direct metric product of a Euclidean polyhedron and a
dual braid complex of smaller rank.
\end{abstract} 

\maketitle

\noindent
Braids and braid groups play an important role throughout mathematics, in part because of the multiple ways in which they can be described.
In this article we view the $n$-strand braid group $\braid_n$ as the fundamental group of the unordered configuration space of $n$ distinct points in the closed unit disk $\mathbb{D}$,
based at an initial configuration $P$ where all $n$ points lie in the boundary of 
$\mathbb{D}$. Requiring the points indexed by $B \subseteq \{1,\ldots,n\}$ to 
remain fixed defines a parabolic subgroup $\fix_n(B)$ which is isomorphic to 
$\braid_{n - \card{B}}$.

We introduce an extension of this idea. A \emph{$(B,\cdot)$-boundary braid}
is a braid that has a representative where the points indexed by $B$ remain in the boundary of the disk but need not be fixed (see Definition~\ref{def:boundary-braids}). Our first main result is that the subgroup $\fix_n(B)$ has a canonical complement $\move_n(B,\cdot)$ in the set $\braid_n(B,\cdot)$ of 
$(B,\cdot)$-boundary braids that gives rise to a unique decomposition (see Section~\ref{sec:boundary_subcomplex}):

\begin{introtheorem}\label{inthm:decn}
Let $B \subseteq \{1,\ldots,n\}$ and let $\beta$ be a 
$(B,\cdot)$-boundary braid in $\braid_n$. Then there are unique braids 
$\fix^B(\beta)$ in $\fix_n(B)$ and $\move^B(\beta)$ in $\move_n(B,\cdot)$ such that
\[
\beta = \fix^B(\beta)\move^B(\beta)\text{.}
\]
\end{introtheorem}
We call the elements of $\fix_n(B)$ \emph{fix braids} and the elements of $\move_n(B,\cdot)$  \emph{move braids}.

Associated to the braid group $\braid_n$ is the \emph{dual braid complex},
introduced by Tom Brady \cite{brady01} and denoted $\comp(\braid_n)$. 
It is a contractible simplicial complex on which $\braid_n$ acts freely and cocompactly. Brady and the second author equipped $\comp(\braid_n)$ with a piecewise Euclidean \emph{orthoscheme metric} and conjectured that it is $\cat(0)$ with respect to this metric \cite{bradymccammond10}. They verified the conjecture for $n < 6$ and Haettel, Kielak and Schwer proved it for $n = 6$ \cite{haettelkielakschwer16}. We are interested in boundary braids as an approach to proving the conjecture. The sets of boundary braids, fix braids, and move braids have induced subcomplexes in $\comp(\braid_n)$ with the following metric decomposition.

\begin{introtheorem}\label{inthm:cx_decn}
  Let $B\subseteq \{1,\ldots,n\}$. The complex of $(B,\cdot)$-boundary braids decomposes as a metric direct product
  \[
  \comp(\braid_n(B,\cdot)) \cong \comp(\fix_n(B)) \times \comp(\move_n(B,\cdot))\text{.}
  \]
The complex $\comp(\move_n(B,\cdot))$ is $\R$ times a Euclidean simplex and therefore $\cat(0)$.
In particular, $\comp(\braid_n(B,\cdot))$ is $\cat(0)$ if and only if the smaller dual braid complex 
$\comp(\fix_n(B)) \cong \comp(\braid_{n - \card{B}})$ is $\cat(0)$.
\end{introtheorem}
As part of our proof for Theorem~\ref{inthm:cx_decn},
we introduce a new type of configuration space for directed graphs and
the broader setting of $\Delta$-complexes. We refer to these as
\emph{orthoscheme configuration spaces} and explore their geometry for 
the case of oriented $n$-cycles.
The other key element for proving Theorem~\ref{inthm:cx_decn}
is a combinatorial study of noncrossing partitions associated to boundary
braids.

Because the points indexed by $B$ do not necessarily return to their original 
positions, either pointwise or as a set, boundary braids form a subgroupoid
rather than a subgroup. More precisely, if we refer to a boundary braid
where the points indexed by $B$ move in the boundary to end at points indexed
by $B'$, then we see that a $(B,B')$-boundary braid can be composed with a 
$(B',B'')$-boundary braid to produce a $(B,B'')$-boundary braid. 
The groupoid $\braid_n(\cdot,\cdot)$ of boundary braids has subgroupoids 
$\fix_n(\cdot)$ and $\move_n(\cdot,\cdot)$ consisting of fix braids and 
move braids respectively. We prove the following algebraic result (see Section~\ref{sec:groupoid}).

\begin{introtheorem}
The groupoid of boundary braids decomposes as a semidirect product
\[
\braid_n(\cdot,\cdot) \cong \fix_n(\cdot) \rtimes \move_n(\cdot,\cdot)\text{.}
\]
\end{introtheorem}
The article is organized as follows. The first part, Sections~\ref{sec:braids} through~\ref{sec:parabolic}, develops standard material about braid groups and their dual Garside structure in a way that suits our later applications. The second part, Sections~\ref{sec:ord-simp} through~\ref{sec:columns}, is concerned with complexes of ordered simplices. Specifically we show how to equip them with an orthoscheme metric and that a combinatorial direct product gives rise to a metric direct product. The final part, Sections~\ref{sec:conf-sp} through~\ref{sec:groupoid}, 
contains our work on boundary braids and the proofs of the main theorems.

\part{Braids}\label{part:braids}

Braids can be described in a variety of ways.  In this part
we establish the conventions used throughout the article, review
basic facts about dual simple braids and the dual presentation for the
braid group, and introduce the concept of a boundary braid.

\section{Braid Groups}\label{sec:braids}

In this article, braid groups are viewed as fundamental groups of
certain configuration spaces.

\begin{defn}[Configuration spaces]\label{def:conf-sp}
  Let $X$ be a topological space, let $n$ be a positive integer and
  let $X^n$ denote the product of $n$ copies of $X$ whose elements are
  $n$-tuples $\vx = (x_1,x_2,\ldots,x_n)$ of elements $x_i \in X$.
  Alternatively, the elements of $X^n$ can be thought of as functions
  from $[n]$ to $X$ where $[n]$ is the set $\{1,2,\ldots,n\}$.  The
  \emph{configuration space of $n$ labeled points in $X$} is the
  subspace $\conf_n(X)$ of $X^n$ of $n$-tuples with distinct entries,
  i.e. the subspace of injective functions.  The \emph{thick diagonal
    of $X^n$} is the subspace $\diag_n(X) = \{(x_1,\ldots,x_n) \mid
  x_i = x_j \text{ for some } i\neq j\}$ where this condition fails.
  Thus $\conf_n(X) = X^n - \diag_n(X)$. The symmetric group acts on
  $X^n$ by permuting coordinates and this action restricts to a free
  action on $\conf_n(X)$.  The \emph{configuration space of $n$
    unlabeled points in $X$} is the quotient space $\uconf_n(X) = (X^n
  - \diag_n(X))/\sym_n$.  Since the quotient map sends the $n$-tuple
  $(x_1,\ldots,x_n)$ to the $n$-element set $\{x_1,\ldots,x_n\}$, we
  write $\set\colon \conf_n(X) \to \uconf_n(X)$ for this natural
  quotient map.
\end{defn}
Since the topology of a configuration space only depends on the
topology of the original space, the following lemma is immediate.

\begin{lem}[Homeomorphisms]\label{lem:homeos}
  A homeomorphism $X \to Y$ induces a homeomorphism
  $h\colon \uconf_n(X) \to \uconf_n(Y)$. In particular, for any
  choice of basepoint $*$ in $\uconf_n(X)$, there is an induced
  isomorphism $\pi_1(\uconf_n(X),*) \cong \pi_1(\uconf_n(Y),h(*))$.
\end{lem}
\begin{exmp}[Configuration spaces]\label{ex:conf-sp}
  When $X$ is the unit circle and $n=2$, the space $X^2$ is a torus,
  $\diag_2(X)$ is a $(1,1)$-curve on the torus, its complement
  $\conf_2(X)$ is homeomorphic to the interior of an annulus and the
  quotient $\uconf_2(X)$ is homeomorphic to the interior of a M\"obius
  band.
\end{exmp}
\begin{defn}[Braids in $\C$]\label{def:braid-C}
  Let $\C$ be the complex numbers with its usual topology and let $\vz
  = (z_1,z_2,\ldots,z_n)$ denote a point in $\C^n$.  The thick
  diagonal of $\C^n$ is a union of hyperplanes $H_{ij}$, with $i < j
  \in [n]$, called the \emph{braid arrangement}, where $H_{ij}$ is the
  hyperplane defined by the equation $z_i = z_j$.  The configuration
  space $\conf_n(\C)$ is the complement of the braid arrangement and
  its fundamental group is called the \emph{$n$-strand pure braid
    group}.  The \emph{$n$-strand braid group} is the fundamental
  group of the quotient configuration space $\uconf_n(\C) =
  \conf_n(\C)/\sym_n$ of $n$ unlabeled points.  In symbols
  \[
  \purebraid_n = \pi_1(\conf_n(\C),\vz)\quad\text{and}\quad \braid_n =
  \pi_1(\uconf_n(\C),Z)
  \]
  where $\vz$ is some specified basepoint in
  $\conf_n(\C)$ and $Z=\set(\vz)$ is the corresponding basepoint in
  $\uconf_n(\C)$.
\end{defn}
\begin{rem}[Short exact sequence]\label{rem:ses}
  The quotient map $\set$ is a covering map, so the induced map
  $\set_*\colon \purebraid_n \to \braid_n$ on fundamental groups is
  injective.  In fact, $\conf_n(\C)$ is a regular cover of
  $\uconf_n(\C)$, so the subgroup $\set_*(\purebraid_n) \subset
  \braid_n$ is a normal subgroup and the quotient group
  $\braid_n/\set_*(\purebraid_n)$ is isomorphic to the group $\sym_n$
  of covering transformations. The quotient map sends each braid to
  the permutation it induces on the $n$-element set used as the
  basepoint of $\braid_n$, a map we define more precisely in the next
  section.  We call this map $\perm$.  These maps form a short exact
  sequence
  \begin{equation} \purebraid_n \stackrel{\set_*}{\into} \braid_n
    \stackrel{\perm}{\onto} \sym_n.\label{eq:ses}
  \end{equation}
\end{rem}
\begin{exmp}[$n \leq 2$]\label{ex:n<=2}
  When $n=1$ the spaces $\uconf_1(\C)$, $\conf_1(\C)$ and $\C$ are
  equal and contractible, and all three groups in
  Equation~\ref{eq:ses} are trivial.  When $n=2$ the space
  $\conf_2(\C)$ is $\C^2$ minus a copy of $\C^1$, which retracts first
  to $\C^1 - \C^0$ and then to the circle $\sph^1$ of unit length
  complex numbers.  The quotient space $\uconf_2(\C)$ also deformation
  retracts to $\sph^1$ and the map from $\conf_2(\C)$ to
  $\uconf_2(\C)$ corresponds to the map from $\sph^1$ to itself
  sending $z$ to $z^2$.  In particular $\purebraid_2 \cong \braid_2
  \cong \Z$ and $\set_*$ is the map that multiplies by $2$ with
  quotient $\Z/2\Z \cong \sym_2$.
\end{exmp}
\begin{convt}[$n >2$]\label{conv:n>2}
  For the remainder of the article, we assume that the integer $n$ is
  greater than $2$, unless we explicitly state otherwise.
\end{convt}
Let $\D \subset \C$ be the closed unit disk centered at the origin.
Restricting to configurations of points that remain in $\D$ does not
change the fundamental group of the configuration space.

\begin{prop}[Braids in $\D$]\label{prop:braid-D}
  The configuration space $\uconf_n(\C)$ deformation retracts to the
  subspace $\uconf_n(\D)$, so for any choice of basepoint $Z$ in the
  subspace, $\pi_1(\uconf_n(\D),Z) = \pi_1(\uconf_n(\C),Z) =\braid_n$.
\end{prop}
\begin{proof}
  Let $m(\vz) = \max\{1, \card{z_1}, \ldots, \card{z_n}\}$ for each
  $\vz = (z_1,\ldots,z_n) \in \C^n$ and note that $m$ defines a
  continuous map from $\conf_n(\C)$ to $\R_{\geq 1}$.  The
  straight-line homotopy from the identity map on $\conf_n(\C)$ to the
  map that sends $\vz$ to $\frac{1}{m(\vz)}\vz$ is a deformation 
  retraction from $\conf_n(\C)$ to $\conf_n(\D)$ and since $m(\vz)$ 
  only depends on the entries of $\vz$ and not their order, this 
  deformation retraction descends to one from $\uconf_n(\C)$ 
  to $\uconf_n(\D)$.
\end{proof}
The following result combines Lemma~\ref{lem:homeos} and
Proposition~\ref{prop:braid-D}.

\begin{cor}[Braids in $D$]\label{cor:braid-P}
  A homeomorphism $\D \to D$ induces a homeomorphism
  of configuration spaces
  $h\colon \uconf_n(\D) \to \uconf_n(D)$. In particular, for any
  choice of basepoint $Z$ in $\uconf_n(\D)$, there is an induced
  isomorphism $\pi_1(\uconf_n(\D),Z) \cong \pi_1(\uconf_n(D),h(Z))
  =\braid_n$.
\end{cor}
\begin{rem}[Points in $\partial D$]\label{rem:pt-boundary}
  When $\braid_n$ is viewed as the mapping class group of an $n$-times
  punctured disk, the punctures are not allowed to move into the
  boundary of the disk since doing so would alter the topological type
  of the punctured space.  When $\braid_n$ is viewed as the
  fundamental group of a configuration space of points in a closed
  disk, points are allowed in the boundary and we make extensive use
  of this extra flexibility.
\end{rem}
We have a preferred choice of basepoint and disk for $\braid_n$.

\begin{figure}
  \begin{tikzpicture}[scale=2]
    \begin{scope}[xshift=-1.5cm]
      \foreach \n in {1,2,...,9} {\coordinate (\n) at (\n*40:1cm);}
      \draw[GreenPoly] (1)--(2)--(3)--(4)--(5)--(6)--(7)--(8)--(9)--cycle;
      \draw[->] (-1.3,0)--(1.3,0);
      \draw[->] (0,-1.3)--(0,1.3);
      \foreach \n in {1,2,...,9} {\draw (\n) node
        [dot,thick,fill=green!20!white] {};}
      \draw (1) node[anchor=south west]{$p_1$};
      \draw (2) node[anchor=south]{$p_2$};
      \draw (3) node[anchor=south east]{$p_3$};
      \draw (4) node[anchor=east]{$p_4$};
      \draw (5) node[anchor=east]{$p_5$};
      \draw (6) node[anchor=north east]{$p_6$};
      \draw (7) node[anchor=north]{$p_7$};
      \draw (8) node[anchor=north west]{$p_8$};
      \draw (9) node[anchor=south west]{$p_9$};
    \end{scope}
    \begin{scope}[xshift=1.5cm]
      \foreach \n in {0,1,...,9} {\coordinate (\n) at (\n*40:1cm);}
      \draw[GreenPoly,fill=green!10]
      (1)--(2)--(3)--(4)--(5)--(6)--(7)--(8)--(9)--cycle;
      \draw[GreenPoly] (0)--(1)--(2)--(6)--cycle;
      \draw[GreenPoly] (3) to[bend right=10] (5) to[bend right=10] (3)--cycle;
      \draw[GreenPoly] (7)--(8) to[bend right=20] (7)--cycle;
      \draw[->] (-1.3,0)--(1.3,0);
      \draw[->] (0,-1.3)--(0,1.3);
      
      \foreach \n in {1,2,...,9} {\draw (\n) node
        [dot,thick,fill=green!20!white] {};} 
      \draw (1) node[anchor=south west]{$p_1$};
      \draw (2) node[anchor=south]{$p_2$};
      \draw (3) node[anchor=south east]{$p_3$};
      \draw (4) node[anchor=east]{$p_4$};
      \draw (5) node[anchor=east]{$p_5$};
      \draw (6) node[anchor=north east]{$p_6$};
      \draw (7) node[anchor=north]{$p_7$};
      \draw (8) node[anchor=north west]{$p_8$};
      \draw (9) node[anchor=south west]{$p_9$};
    \end{scope}
    \end{tikzpicture}
    \caption{The figure on the left shows the standard basepoint $P =
      \{p_1,p_2,\ldots,p_9\}$ and the standard disk $D$ for $n=9$.
      The figure on the right shows the standard subdisks $D_A$ for
      $A$ equal to $\{1,2,6,9\}$, $\{3,5\}$ and
      $\{7,8\}$.\label{fig:std-disks}}
\end{figure}
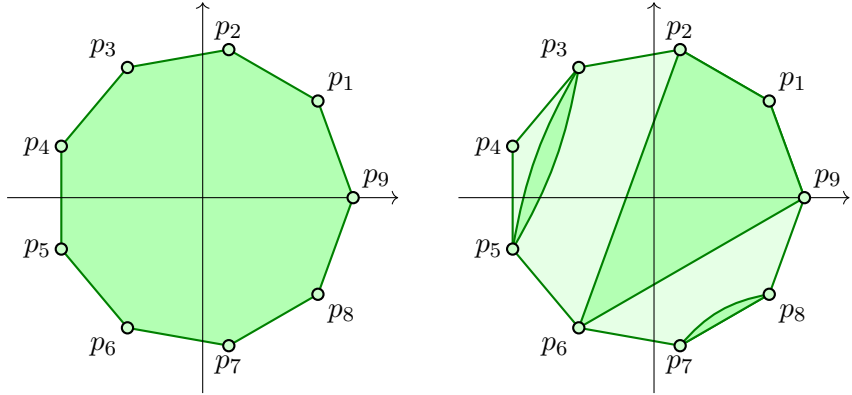
\begin{defn}[Basepoints and disks]\label{def:basepoints}
  Let $\zeta = e^{2\pi i/n} \in \C$ be the standard primitive $n$-th
  root of unity and let $p_i$ be the point $\zeta^i$ for all $i \in
  \Z$.  Since $\zeta^n= 1$, the subscript $i$ should be interpreted as
  an integer representing the equivalence class $i + n\Z \in \Z/n\Z$.
  In particular, we consider $p_{i-n} = p_i = p_{i+n} = p_{i+2n}$
  without further comment.  The \emph{standard basepoint for
    $\purebraid_n$} is the $n$-tuple $\vp = (p_1,p_2,\ldots,p_n)$ and
  the \emph{standard basepoint for $\braid_n$} is the $n$-element set
  $P = \set(\vp) = \{p_1,p_2,\ldots,p_n\}$ of all $n$-th roots of
  unity.  Let $D$ be the convex hull of the points in $P$. Our
  standing assumption of $n>2$ means that $D$ is homeomorphic to the
  disk $\D$.  We call $D$ the \emph{standard disk for $\braid_n$}.
  See Figure~\ref{fig:std-disks}.
\end{defn}
\begin{rem}[Braid groups]\label{rem:braid-def}
  By Corollary~\ref{cor:braid-P}, the braid group $\braid_n$ is
  isomorphic to $\pi_1(\uconf_n(D),P)$, the fundamental group of the
  configuration space of $n$ unlabeled points in the standard disk $D$
  based at the standard basepoint $P$.  In the remainder of the
  article, we use the notation $\braid_n$ to refer to the specific
  group $\pi_1(\uconf_n(D),P)$.  
\end{rem}
%

\section{Individual Braids}\label{sec:individual}

This section establishes our conventions for describing individual
braids and we introduce the concept of a boundary braid.

\begin{defn}[Representatives]\label{def:braid-reps}
  Each braid $\alpha \in \braid_n$ is a basepoint-preserving homotopy
  class of a path $f\colon [0,1] \to \uconf_n(D,P)$ that describes a
  loop based at the standard basepoint $P$.  We write $\alpha = [f]$
  and say that the loop \emph{$f$ represents $\alpha$}.  We use Greek
  letters such as $\alpha$, $\beta$ and $\delta$ for braids and Roman
  letters such as $f$, $g$ and $h$ for their representatives.
\end{defn}
Vertical drawings of braids in $\R^3$ typically have the $t=0$ start
at the top and the $t=1$ end at the bottom.  See
Definition~\ref{def:drawings} for the details.  As a mnemonic, we
use superscripts for information about the start of a braid or a path
and subscripts for information about its end.

\begin{defn}[Strands]\label{def:strands}
  Let $\alpha \in \braid_n$ be a braid with representative $f$.  A
  \emph{strand of $f$} is a path in $D$ that follows what happens to
  one of the vertices in $P$.  There are two natural ways to name
  strands: by where they start and by where they end.  The
  \emph{strand that starts at $p_i$} is the path $f^i\colon [0,1] \to
  D$ defined by the composition $f^i = \proj_i \circ
  \widetilde{f}^{\vp}$, where the map $\widetilde{f}^{\vp}$ is the unique
  lift of the path $f$ through the covering map $\set \colon
  \conf_n(D) \to \uconf_n(D)$ so that the lifted path starts at $\vp$,
  i.e. $\widetilde{f}^{\vp}(0) = \vp$, and $\proj_i:\conf_n(D) \to D$
  is projection onto the $i$-th coordinate.  Similarly the
  \emph{strand that ends at $p_j$} is the path $f_j\colon [0,1] \to D$
  defined by the composition $f_j = \proj_j \circ \widetilde{f}_{\vp}$
  where $\widetilde{f}_{\vp}$ is the unique lift of the path $f$
  through the covering map $\set \colon \conf_n(D) \to \uconf_n(D)$ 
  that ends at $\vp$, i.e. $\widetilde{f}_{\vp}(1) = \vp$.  When the
  strand of $f$ that starts at $p_i$ ends at $p_j$ the path $f^i$ is
  the same as the path $f_j$.  We write $f^i$, $f_j$ or $f^i_j$ for
  this path and we call it the \emph{$(i,\cdot)$-strand}, the
  \emph{$(\cdot,j)$-strand} or the \emph{$(i,j)$-strand of $f$}
  depending on the information specified.
\end{defn}
A braid representative is drawn by superimposing the graphs of its
strands.

\begin{defn}[Drawings]\label{def:drawings}
  Let $\alpha \in \braid_n$ be a braid with representative $f$.  A
  \emph{drawing of $f$} is formed by superimposing the graphs of its
  strands inside the polygonal prism $[0,1] \times D$.  There are two
  distinct conventional embeddings of this prism into $\R^3$.  The
  complex plane containing $D$ is identified with either the first two
  or the last two coordinates of $\R^3$ and the remaining coordinate
  indicates the value $t \in [0,1]$ with the $t$-dependence arranged
  so that the $t=0$ start of $f$ is on the left or at the top and the
  $t=1$ end of $f$ is on the right or at the bottom.  In the
  left-to-right orientation, for each $j \in [n]$, for each $t_0\in
  [0,1]$ and for each point $f_j(t_0) = z_0= x_0 + iy_0 \in D \subset
  \C$ on the $(\cdot,j)$-strand we draw the point $(t_0,x_0,y_0) \in
  \R^3$.  In the top-to-bottom orientation, the same point on the
  $(\cdot,j)$-strand is drawn at $(x_0,y_0,1-t_0) \in \R^3$.
\end{defn}
\begin{defn}[Multiplication]\label{def:multn}
  Let $\alpha_1$ and $\alpha_2$ be braids in $\braid_n$ with
  representatives $f_1$ and $f_2$.  The product $\alpha_1\cdot
  \alpha_2$ is defined to be $[f_1.f_2]$ where $f_1.f_2$ is the
  concatenation of $f_1$ and $f_2$.  In the drawing of $f_1.f_2$ the
  drawing of $f_1$ is on the top or left of the drawing of $f_2$ which
  is on the bottom or right.  See Figure~\ref{fig:rotations}.
\end{defn}
\begin{defn}[Permutations]\label{def:perm}
  A \emph{permutation of the set $[n]$} is a bijective correspondence
  between a left/top copy of $[n]$ and a right/bottom copy of
  $[n]$. Permutations are compactly described in disjoint cycle
  notation.  A cycle such as $(1\ 2\ 3)$, for example, means that $1$
  on the left corresponds to $2$ on the right, $2$ on the left
  corresponds to $3$ on the right, and $3$ on the left corresponds to
  $1$ on the right.  Multiplication of permutations is performed by
  concatenating the correspondences left-to-right or top-to-bottom.
  The permutation $\tau$ of $[n]$ acts on $[n]$ from either the left 
  or the right by following the correspondence: if $i$ on the left corresponds
  to $j$ on the right, then $i\cdot \tau = j$ and $i=\tau\cdot j$.  
\end{defn}
\begin{defn}[Permutation of a braid]\label{def:braid-perm}
  The \emph{permutation} of a braid $\alpha$ is the bijective
  correspondence of $[n]$ under which $i$ on the left corresponds
  to $j$ on the right if $\alpha$ has an $(i,j)$-strand.
  Note that the function
  $\perm(\alpha)$ only depends on the braid $\alpha$ and not on the
  representative $f$.  The direction of the bijection $\perm(\alpha)$
  is defined so that it is compatible with function composition,
  i.e. so that $\perm(\alpha_1 \cdot \alpha_2) = \perm(\alpha_1) \circ
  \perm(\alpha_2)$.
\end{defn}
Information about how a braid permutes its strands can be be used to
distinguish different types of braids.

\begin{defn}[$(i,j)$-braids]
  Let $\alpha \in \braid_n$ be a braid with permutation $g =
  \perm(\alpha)$.  We say that $\alpha$ is an \emph{$(i,j)$-braid} if
  the strand that starts at $p_i$ ends at $p_j$. In other words,
  $\alpha$ is an $(i,j)$-braid if and only if $g(j) = i$.  When
  $\alpha$ is an $(i,j)$-braid and $\beta$ is a $(j,k)$-braid, $\gamma
  = \alpha \cdot \beta$ is a $(i,k)$-braid and the inverse of $\alpha$
  is a $(j,i)$-braid.
\end{defn}
For our applications, we make use of a generating set for the braid
group which is built out of braids we call rotation braids. Rotation braids
are defined using special subsets of our standard basepoint $P$ 
and subspaces of our standard disk $D$.

\begin{defn}[Subsets of $P$]\label{def:subsets}
  For each non-empty $A \subset [n]$ of size $k$, we define $P_A =
  \{p_i \mid i \in A\} \subset P$ to be the subset of points indexed
  by the numbers in $A$.  In this notation our
  original set $P$ is $P_{[n]}$.
  Using this notation we can extend the notion
  of an $(i,j)$-braid.  Let $A$ and $B$ be two subsets of $[n]$ of the
  same size and let $\alpha \in \braid_n$ be a braid with permutation
  $g = \perm(\alpha)$.  We say that $\alpha$ is an
  \emph{$(A,B)$-braid} if every strand that starts in $P_A$, ends in
  $P_B$, i.e. if and only if $g(B) = A$. 
\end{defn}

\begin{defn}[Subdisks of $D$]\label{def:subdisks}
  For all distinct $i,j\in [n]$ let the \emph{edge $e_{ij}$} be the
  straight line segment connecting $p_i$ and $p_j$.  For $k>2$, let
  $D_A$ be $\conv(P_A)$, the convex hull of the points in $P_A$ and
  note that $D_A$ is a $k$-gon homeomorphic to $\D$.  We call this the
  \emph{standard subdisk for $A \subset [n]$}.  In this notation, our
  original disk $D$ is $D_{[n]}$.  For $k=2$ and $A=\{i,j\}$, we
  define $D_A$ so that it is also a topological disk.  Concretely, we
  take two copies of the path along the edge $e = e_{ij}$ from $p_i$
  to $p_j$ and then bend one or both of these copies so that they
  become injective paths from $p_i$ to $p_j$ with disjoint interiors
  which together bound a bigon inside of $D$.  Moreover, when the edge
  $e$ lies in the boundary of $D$ we require that one of the two paths
  does not move so that $e$ itself is part of the boundary of the
  bigon.  See Figure~\ref{fig:std-disks}.  For $k=1$, we define $D_A$
  to be the single point $p_i \in P_A$, but note that this subspace is
  not a subdisk.  The bending of the edges to form the bigons are
  chosen to be slight enough so that for all $A$ and $B \subseteq [n]$
  the standard disks $D_A$ and $D_B$ intersect if and only if the
  convex hulls $\conv(P_A)$ and $\conv(P_B)$ intersect.
\end{defn}
We view the boundaries of these subdisks as directed graphs.

\begin{defn}[Boundary edges]\label{def:boundary-edges}
  When $A$ has more than $1$ element, we view $\partial D_A$, the
  topological boundary of the subdisk $D_A$, as having the structure
  of a directed graph. The vertex set is $P_A$ and for every vertex
  $p_i$ in $P_A$ there is a directed edge that starts $p_i$, proceeds
  along $\partial D_A$ in a counter-clockwise direction with respect
  to the interior of $D_A$, and ends at the next vertex in $P_A$ that
  it encounters.  The edges of the graph $\partial D_A$ are called the
  \emph{boundary edges of $D_A$}.  Note that edges and boundary edges
  are distinct concepts.  An edge is unoriented and necessarily
  straight.  A boundary edge is directed, it belong to the boundary of
  a specific subdisk $D_A$ and the path it describes might curve.
\end{defn}
\begin{figure}
	\begin{tikzpicture}[scale=3]
  \begin{scope}[xshift=3cm,x={(.75cm,-.01cm)},y={(.1cm,.18cm)},z={(0cm,1cm)}]
    \draw[->] (-1.5,0,0)--(1.5,0,0);
    \draw[->] (0,-1.5,0)--(0,.-.5,0);
    \draw[->] (0,0,-.5)--(0,0,0); 
    
    \foreach \i in {0,1,...,6} {\coordinate  (a\i) at (\i*60:1);}
    \foreach \i in {0,1,...,6} {\coordinate  (b\i) at ($(\i*60:1)+(0,0,1)$);}
    \foreach \i in {0,1,...,6} {\coordinate  (c\i) at ($(\i*60:1)+(0,0,2)$);}
    \draw[fill = green!15] (a0)--(a5)--(b5)--(b0)--cycle;
    \draw[fill = green!15] (b2)--(b3)--(c3)--(c4)--(c5)--(c6)--(b6)--(b1)--cycle;
    \draw[fill = green!35] (c1)--(c2)--(c3)--(c4)--(c5)--cycle;
    \draw[fill = green!35,color=green!10] (b1)--(b2)--(b5)--cycle;
    \draw[fill = green!35,color=green!10] (a1)--(a2)--(a5)--(a6)--cycle;
    \draw[fill = blue!20] (b0)--(b1)--(b5)--(c5)--(c0)--cycle;
    \draw[fill = blue!20] (a5)--(b5)--(b4)--(b3)--(a3)--(a2)--cycle;
    \draw[->] (0,.707,0)--(0,1.5,0);

    \draw[GreenLine,very thin]  (a5)--(a0)--(a1)--(a2)--(a3) (b0)--(b1)--(b2)--(b3);
    \draw[blue!20, ultra thick] (c1)--(b5) (c0)--(b1) (b5)--(a2) (b2)--(a3);
    \draw[green!10!white,ultra thick]  (b1)--(a1);
    \draw[GreenPoly,thin,fill=green!60!white] (a2)--(a3)--(a4)--(a5)--cycle;
    \draw[GreenPoly,thin,fill=green!60!white] (b2)--(b3)--(b4)--(b5)--cycle;
    \draw[GreenPoly,thin,fill=green!60!white] (b0)--(b1)--(b5)--cycle;
    \draw[GreenPoly,thin,fill=green!60!white] (c0)--(c1)--(c5)--cycle;
    \draw[GreenLine,very thin] (a2)--(b2) (b1)--(c1);
    \draw[red!90] (c0)--(b1) (c1)--(b5) (c2)--(b2) (b1)--(a1) (b5)--(a2) (b2)--(a3);
    \foreach \x in {a1,a2,b1,b2} {\filldraw[dot,color=black,fill=red!30] (\x) circle (.3mm);}
    \node[anchor=south west] at (a1) {$v_1$};
    \node[anchor=south east] at (a2) {$v_2$};
    \draw[dashed] (-1,0,0)--(1,0,0) (0,-1,0)--(0,1,0);
    \draw[green!10,very thick] (b3)--(b4)--(b5)--(b6)
    (c3)--(c4)--(c5);
    
    \draw[GreenLine,thick] (b3)--(b4)--(b5)--(b6) (c3)--(c4)--(c5)--(c6);
    \draw[blue!20,ultra thick] (a4)--(b4) (c5)--(c6);
    \draw[green!10,thick] (a5)--(c5);
    \draw[GreenLine,thin] (a3)--(b3) (a4)--(b4) (a5)--(c5) (b6)--(c6);
    \draw[GreenLine] (c1)--(c2)--(c3)--(c4)--(c5)--(c6)--cycle;
    \draw[->] (0,0,2)--(0,0,2.5);

    \draw[blue!20, ultra thick] (c5)--(b0) (b3)--(a4) (b4)--(a5);
    \draw[green!10,ultra thick] (b4)--(c4);
    \draw[red!80!black,thick] (c5)--(b0) (c3)--(b3) (c4)--(b4) (b3)--(a4) (b4)--(a5) (b0)--(a0);
    \foreach \x in {a3,a4,a5,a6,b3,b4,b5,b6,c1,c2,c3,c4,c5,c6}{
      \filldraw[dot,color=black,fill=red] (\x) circle (.3mm);}
    \foreach \i in {0,2} {\filldraw (0,0,\i) circle (.12mm);}
    \node[anchor=north east] at (a3) {$v_3$};
    \node[anchor=north east] at (a4) {$v_4$};
    \node[anchor=north west] at (a5) {$v_5$};
    \node[anchor=north west] at (a6) {$v_6=v_0$};
    \node at (-2,0,.5) {$\alpha_2 = \delta_{\{2,3,4,5\}}$};
    \node at (-2,0,1.5) {$\alpha_1 = \delta_{\{1,5,6\}}$};
    \draw[<-] (2,0,.1)--(2,0,1.9);
    \node at (2,0,0) {$t=1$};
    \node at (2,0,2) {$t=0$};
  \end{scope}
\end{tikzpicture}

  \caption{A drawing of $\alpha_1 \cdot \alpha_2$ where $\alpha_1 =
    \delta_{A_1}$ and $\alpha_2 = \delta_{A_2}$ are rotations with
    $A_1 = \{1,5,6\}$ and $A_2 = \{2,3,4,5\}$. \label{fig:rotations}}
\end{figure}
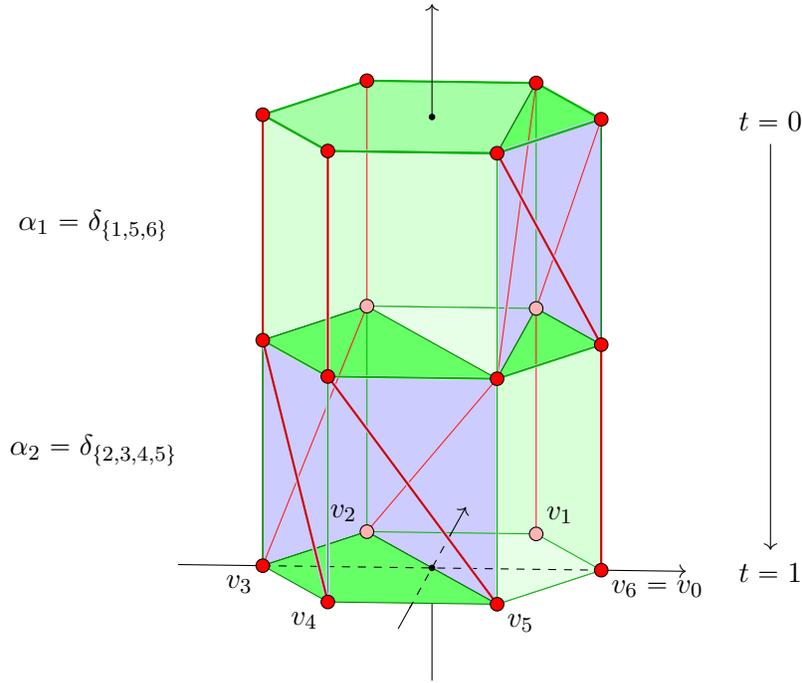
\begin{defn}[Rotation braids]\label{def:rotations}
  For $A \subset [n]$ of size $k = \card{A}>1$ we define an element
  $\delta_A \in \braid_n$ that we call the \emph{rotation
    braid of the vertices in $P_A$}.  It is the braid represented by
  the path in $\uconf_n(D)$ that fixes the vertices in $P - P_A$ and
  where every vertex $p_i \in P_A$ travels in a counter-clockwise
  direction along the oriented edge in the directed graph $\partial
  D_A$ to the next vertex it encounters.  If $f$ is any representative
  of $\delta_A$ satisfying this description, we call $f$ a
  \emph{standard representative} of $\delta_A$.
\end{defn}
When $A = [n]$ we write $\delta$ instead of $\delta_{[n]}$ and when
$A$ has only a single element we let $\delta_A$ denote the identity
element in $\braid_n$.  If $A = \{i,j\}$ and $e=e_{ij}$ is the edge
connecting $p_i$ and $p_j$, then we sometimes write $\delta_e$ to mean
$\delta_A$, the rotation of $p_i$ and $p_j$ around the boundary of the
bigon $D_A$.  Note that if $A = \{i_1,i_2,\ldots,i_k\} \subset [n]$
and $i_1 < i_2 < \cdots < i_k$ is the natural linear order of its
elements, then the bijection $\perm(\delta_A)$ is equal to the
$k$-cycle $(i_1,i_2,\ldots,i_k)$.

\begin{exmp}[Rotation braids]\label{ex:rotations}
  Figure~\ref{fig:rotations} shows a drawing of the product of two
  rotation braids in $\braid_9$.  The top braid $\alpha_1$ is the
  rotation $\delta_{A_1}$ with $A_1 =\{1,5,6\}$ and $\perm(\alpha_1) =
  (1\ 5\ 6)$.  The bottom braid $\alpha_2$ is the rotation braid
  $\delta_{A_2}$ with $A_2 = \{2,3,4,5\}$ and $\perm(\alpha_2) =
  (2\ 3\ 4\ 5)$.  The product braid $\alpha = \alpha_1\cdot \alpha_2$
  is the rotation braid $\delta$ of all $6$ vertices with
  $\perm(\alpha) = \perm(\alpha_1\cdot \alpha_2) = \perm(\alpha_1)
  \circ \perm(\alpha_2) = (1\ 5\ 6)(2\ 3\ 4\ 5) = (1\ 2\ 3\ 4\ 5\ 6)$.
\end{exmp}
In this article, we focus on whether or not particular strands pass
through the interior of the polygonal disk $D$.  Those that remain in
the boundary of $D$ are called boundary strands.

\begin{defn}[Boundary braids]\label{def:boundary-braids}
  Let $f$ be a representative of an $(i,j)$-braid $\alpha \in
  \braid_n$.  If the $(i,j)$-strand of $f$ remains in the boundary
  $\partial D$, then it is a \emph{boundary strand} of $f$ and a
  \emph{boundary parallel strand} of $\alpha$.  The linguistic shift
  from ``boundary'' to ``boundary parallel'' reflects the fact that
  while for many representatives of $\alpha$, the $(i,j)$-strand will
  not remain in the boundary, it will always remain parallel to the
  boundary in a sense that can be made precise.  When $\alpha$ has
  some representative in which its $(i,j)$-strand is a boundary
  strand, $\alpha$ is called a \emph{$(i,j)$-boundary braid}.  More
  generally, suppose $\alpha$ is an $(B,C)$-braid and there is a
  representative $f$ of $\alpha$ so that every strand that starts in
  $P_B$ is a boundary strand of $f$.  We then call $\alpha$ an
  \emph{$(B,C)$-boundary braid}.
\end{defn}
  Note that the definition of an $(B,C)$-boundary braid requires
  a single representative where all of these strands remain in
  $\partial D$.  
  We will see in Section~\ref{sec:boundary_braids} that such a 
  representative exists as soon as there are representatives 
  which keep the $(i,\cdot)$-strand in the boundary
  for each $i \in B$.

\begin{exmp}[Boundary braids]\label{ex:boundary-braids}
  In Figure~\ref{fig:rotations} the rotation braid $\alpha_2 =
  \delta_{A_2} \in \braid_6$ with $A_2 =\{2,3,4,5\}$ is an
  $(A_2,A_2)$-braid but it is not an $(A_2,A_2)$-boundary braid since
  the $(5,2)$-strand passes through the interior of $D$.  It is,
  however, a $(B,C)$-boundary braid with $B = \{1,2,3,4,6\}$ and $C =
  \{1,3,4,5,6\}$ since all five of the corresponding strands, i.e. the
  $(1,1)$, $(2,3)$, $(3,4)$, $(4,5)$ and $(6,6)$ strands, remain in
  the boundary of $D$ in its standard representative.
\end{exmp}
%
\section{Dual Simple Braids}\label{sec:dual-simple}

This section defines dual simple braids and the dual presentation of
the braid group using the rotation braids from the previous section.
We begin with the combinatorics of the noncrossing partition lattice.
Recall that $P = P_{[n]} \subset \C$ denotes the set of $n$-th roots
of unity.

\begin{defn}[Noncrossing partitions]\label{def:nc-part}
  A partition $\pi = \{A_1,\ldots,A_k\}$ of the set $[n]$ is
  \emph{noncrossing} when the convex hulls
  $\conv(P_{A_1}),\ldots,\conv(P_{A_k})$ of the corresponding sets of
  points in $P$ are pairwise disjoint.  A partition is
  \emph{irreducible} if it has exactly one block with more than one
  element.  Since there is an obvious bijection between irreducible
  partitions and subsets of $[n]$ of size at least $2$, we write
  $\pi_A$ to indicate the irreducible partition whose unique
  non-singleton block is $A$.
\end{defn}

\tikzstyle{plate}=[thick,shape=circle,draw,color=black,fill=yellow!10,
	rounded corners,minimum size=1.2cm]
\begin{figure}
  \begin{center}
    \begin{tikzpicture}
      \def\rowA{3.5} \def\rowB{1.5} \def\rowC{-1.5} \def\rowD{-3.5}

      \begin{scope}
	\coordinate (top) at (0,\rowA);
	\foreach \i in {1,...,6} \coordinate (a\i) at (-5.25+\i*1.5,\rowB);
	\foreach \i in {1,...,6} \coordinate (b\i) at (-5.25+\i*1.5,\rowC);
	\coordinate (bottom) at (0,\rowD);
      \end{scope}
      \draw[thick] 
      (top) edge (a1) edge (a2) edge (a3) edge (a4) edge (a5) edge (a6) 
      (a1) edge (b1) edge (b2) edge (b4) 
      (a2) edge (b1) edge (b3) edge (b5)
      (a3) edge (b2) edge (b3)
      (a4) edge (b4) edge (b5) 
      (a5) edge (b2) edge (b5) edge (b6)
      (a6) edge (b3) edge (b4) edge (b6)
      (bottom) edge (b1) edge (b2) edge (b3) edge (b4) edge (b5) edge (b6);
      
      \begin{scope}
	\node[plate] (top) at (0,\rowA) {};
	\foreach \i in {1,...,6} \node[plate] (a\i) at (-5.25+\i*1.5,\rowB) {};
	\foreach \i in {1,...,6} \node[plate] (b\i) at (-5.25+\i*1.5,\rowC) {};
	\node[plate] (bottom) at (0,\rowD) {};
      \end{scope}
      \def\makepoints{\foreach \n in {1,2,3,4} {\coordinate (\n) at (\n*90:0.4cm);};}
      \newcommand{\drawNCP}[1]{\makepoints \filldraw #1 \drawpoints}
      \begin{scope}[BluePoly]
	\begin{scope}[shift={(0,\rowA)}] \drawNCP{(1)--(2)--(3)--(4)--cycle;} \end{scope}
	\begin{scope}[shift={(-3.75,\rowB)}] \drawNCP{(1)--(2)--(3)--cycle;} \end{scope}
	\begin{scope}[shift={(-2.25,\rowB)}] \drawNCP{(1)--(3)--(4)--cycle;} \end{scope}
	\begin{scope}[shift={(-0.75,\rowB)}] \drawNCP{(1)--(4); \draw (2)--(3);} \end{scope}
	\begin{scope}[shift={(0.75,\rowB)}] \drawNCP{(1)--(2); \draw (3)--(4);} \end{scope}
	\begin{scope}[shift={(2.25,\rowB)}] \drawNCP{(2)--(3)--(4)-- cycle;} \end{scope}
	\begin{scope}[shift={(3.75,\rowB)}] \drawNCP{(1)--(2)--(4)--cycle;} \end{scope}
	\begin{scope}[shift={(-3.75,\rowC)}] \drawNCP{(1)--(3);} \end{scope}
	\begin{scope}[shift={(-2.25,\rowC)}] \drawNCP{(2)--(3);} \end{scope}
	\begin{scope}[shift={(-0.75,\rowC)}] \drawNCP{(1)--(4);} \end{scope}
	\begin{scope}[shift={(0.75,\rowC)}] \drawNCP{(1)--(2);} \end{scope}
	\begin{scope}[shift={(2.25,\rowC)}] \drawNCP{(3)--(4);} \end{scope}
	\begin{scope}[shift={(3.75,\rowC)}] \drawNCP{(2)--(4);} \end{scope}
	\begin{scope}[shift={(0,\rowD)}] \makepoints \drawpoints \end{scope}
      \end{scope}
    \end{tikzpicture}
  \end{center}
  \caption{The noncrossing partition lattice $\NC_4$}
  \label{fig:nc-part}
\end{figure}
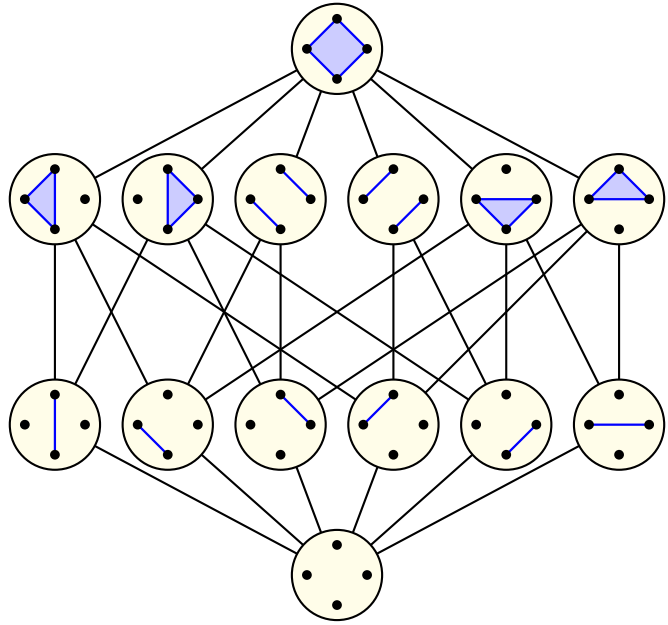

\begin{defn}[Noncrossing partition lattice]\label{def:ncp-lattice}
  Let $\pi$ and $\pi'$ be noncrossing partitions of $[n]$.  If each
  block of $\pi$ is contained in some block of $\pi'$, then $\pi$ is
  called a \emph{refinement} of $\pi'$ and we write $\pi \leq \pi'$.
  The set of all noncrossing partitions of $[n]$ under the refinement
  partial order has well defined meets and joins and is called the
  \emph{lattice of noncrossing partitions $\NC_n$}.
\end{defn}
The Hasse diagram for $\NC_4$ is shown in Figure~\ref{fig:nc-part}.
The noncrossing partition lattice has a maximum partition with only
one block and a minimum partition, also called the \emph{discrete
  partition}, where each block contains a single element.  We write
$\NC_n^*$ for the poset of non-trivial noncrossing partitions,
i.e. $\NC_n$ with the discrete partition removed.

\begin{defn}[Rank function]\label{def:rank}
  The noncrossing partition lattice $\NC_n$ is a graded poset with a
  rank function.  The \emph{rank} of the noncrossing partition $\pi =
  \{A_1,A_2,\ldots,A_k\}$ is $\rk(\pi) = n-k$.  In particular, the
  rank of the discrete partition is $0$, the rank of the maximum
  partition is $n-1$ and the rank of the irreducible partition $\pi_A$
  is $\card{A}-1$.
\end{defn}
For more about noncrossing partitions, see \cite{mccammond06,
  armstrong09, stanley-ec1}.  Using the rotation braids defined in
Definition \ref{def:rotations}, there is a natural map from
noncrossing partitions to braids.

\begin{defn}[Dual simple braids]\label{def:dual-simple-braids}
  Let $\pi = \{A_1,\ldots,A_k\} \in \NC_n$ be a
  noncrossing partition. The \emph{dual simple braid} $\delta_\pi$ is
  defined to be the product of the rotation braids
  $\delta_{A_1}\cdots\delta_{A_k}$ and since the rotation braid of a
  singleton set is the trivial braid, the product only needs to be
  taken over the blocks of size at least $2$.  Moreover, because the
  standard subdisks $D_{A_i}$ are pairwise disjoint, the rotation
  braids $\delta_{A_i}$ pairwise commute and the order in which they
  are multiplied is irrelevant.  Finally note that for each $A \subseteq
  [n]$ of size at least $2$, the irreducible partition $\pi_A$ corresponds 
  to the rotation braid $\delta_A$.  In accordance with notation for noncrossing
  partitions we denote by $\DS_n \defeq \{\delta_\pi \mid \pi \in \NC_n\}$
  the set of dual simple braids and by 
  $\DS_n^* \defeq \{\delta_\pi \mid \pi \in \NC_n^*\}$
  the set of non-trivial dual simple braids.
  We equip both sets with the order coming from $\NC_n$.
\end{defn}
Taking the transitive closure gives a partial order on all of $\braid_n$.
It has the following property:

\begin{prop}\label{prop:dual-simple-order}
	The partial order $\le$ on $\braid_n$ is a left-invariant lattice order. The set $\DS_n$
	is the interval $[1,\delta]$ with respect to this order.
	In particular, if $\sigma,\tau\in\NC_n$ then $\sigma \leq \tau$
	if and only if
	${\delta_\sigma}^{-1}\delta_\tau$ is a dual simple braid if and only
	if $\delta_\tau{\delta_\sigma}^{-1}$ is a dual simple braid.
\end{prop}
\begin{proof}
This can be seen from \cite{birmankolee98} but it is easier to reference from \cite{brady01}. Our dual simple braids are ``(braids corresponding to) allowable elements'' in \cite{brady01}. That the order on $\DS_n$ is left-divisibility follows from \cite[Lemma~3.10]{brady01}. Consequently taking the transitive closure is the same as taking an element to be $\ge 1$ if and only if it is generated by dual simple braids. That this defines a left-invariant partial order follows from \cite[Lemma~5.6]{brady01}. That $\DS_n$ is the interval $[1,\delta]$ follows from the injectivity statement \cite[Theorem~5.7]{brady01}.
\end{proof}
There is a third poset, isomorphic to both $\NC_n$ and $\DS_n$, which
provides another useful perspective on the combinatorics of noncrossing
partitions.

\begin{defn}[Noncrossing Permutations]\label{def:nc-perm}
	As described in Definition~\ref{def:dual-simple-braids}, the
	poset of dual simple braids is obtained via an injection of
	$\NC_n$ into $\braid_n$. The composition of this injection with
	the $\perm$ map is also injective and we refer to the image of the
	composition $\NC_n \hookrightarrow \sym_n$ as the set of \emph{noncrossing
	permutations}, denoted $\NP_n$. Refer to the noncrossing permutation
	corresponding to $\pi\in\NC_n$ as $\sigma_\pi$. With the partial order induced by 
	the noncrossing partition lattice, $\NP_n$ is isomorphic to both
	$\DS_n$ and $\NC_n$.
\end{defn}
To help the reader keep track of the notation adherent to $\NC_n$ and its counterparts within $\braid_n$ and $\sym_n$ we provide a dictionary in Table~\ref{tab:dic}.

\begin{table}
\centering
\begin{tabular}{ccccc}
 & & $\sym_n$ & & $\braid_n$\\
&& \rotatebox{90}{$\subseteq$} && \rotatebox{90}{$\subseteq$}\\
$\NC_n$ &$\cong$& $\NP_n$ &$\cong$& $\DS_n$\\
\clap{\parbox{3cm}{\center(noncrossing\\partitions)}} &\rule{1cm}{0pt}& \clap{\parbox{3cm}{\center(noncrossing\\permutations)}} &\rule{1cm}{0pt}& \clap{\parbox{3cm}{\center(dual simple\\braids)}}
\end{tabular}
\medskip
\caption{The names of noncrossing partitions as subsets of the symmetric group and the braid group.}
\label{tab:dic}
\end{table}
\tikzstyle{disk}=[thick,shape=circle,draw,color=black,fill=yellow!10]
\begin{figure}
  \begin{center}
  \begin{tikzpicture}
    \begin{scope}[shift={(-2.5,0)},BluePoly]
      \node () [disk,minimum size=4.2cm] {};
      \foreach \n in {1,...,9} {
        \coordinate (\n) at (\n*40:1.5cm);
        \node[black] () at (\n*40:1.8cm) {\n};
      }
      \filldraw (1)--(2)--(6)--(9)--cycle;
      \draw (3)--(5) (7)--(8);
      \foreach \n in {1,...,9} {
        \draw (\n) node [dot] {};
      }
    \end{scope}
    \begin{scope}[shift={(2.5,0)},RedPoly]
      \node () [disk,minimum size=4.2cm] {};
      \foreach \n in {1,...,9} {
      	\coordinate (\n) at (\n*40:1.5cm);
      	\node[black] () at (\n*40:1.8cm) {\n};
      }
      \draw[->, shorten >= 6pt] (1)--(2);
      \draw[->, shorten >= 6pt] (2)--(6);
      \draw[->, shorten >= 6pt] (6)--(9);
      \draw[->, shorten >= 6pt] (9)--(1);
      \draw[->, shorten >= 6pt] (3) to [bend right=15] (5);
      \draw[->, shorten >= 6pt] (5) to [bend right=15] (3);
      \draw[->, shorten >= 5pt] (8) to [bend right=40] (7);
      \draw[->, shorten >= 6pt] (7)--(8);
      \foreach \n in {1,...,9} {
        \draw (\n) node [dot] {};
      }
    \end{scope}
  \end{tikzpicture}
  \caption{The noncrossing partition $\pi$ with blocks $A_1 =
    \{1,2,6,9\}$, $A_2 = \{3,5\}$, $A_3 = \{4\}$ and $A_4 = \{7,8\}$
    on the left corresponds to the braid $\delta_\pi = \delta_{A_1}
    \delta_{A_2} \delta_{A_4}$ on the right.\label{fig:nc-perm}}
   \end{center}
\end{figure}
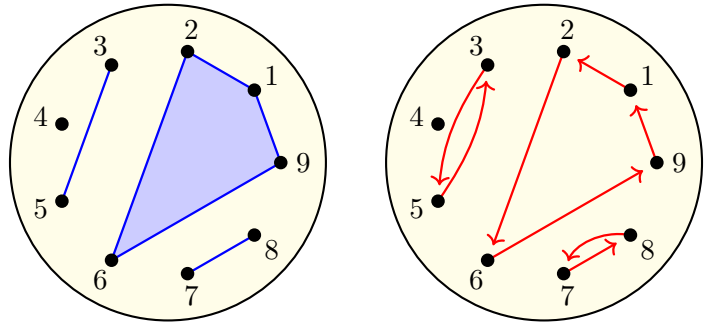
The following proposition records standard facts about factorizations
of dual simple braids into dual simple braids.

\begin{prop}[Relations]\label{prop:relations}
  If $\pi, \pi' \in \NC_n$ are noncrossing partitions
  with $\pi \le \pi'$, then there exist unique $\pi_1, \pi_2 \in
  \NC_n$ such that $\delta_{\pi_1} \delta_{\pi} =
  \delta_{\pi}\delta_{\pi_2} = \delta_{\pi'}$ in $\braid_n$.
  Conversely, if $\pi_1, \pi_2,\pi_3 \in \NC_n$ are
  noncrossing partitions such that $\delta_{\pi_1} \delta_{\pi_2} =
  \delta_{\pi_3}$ in $\braid_n$, then $\pi_1 \le \pi_3$, $\pi_2 \le \pi_3$
  and $\pi_3$ is the join of $\pi_1$ and $\pi_2$ in $\NC_n$.
\end{prop}
\begin{proof}
Follows from Theorem~3.7, Lemma~3.9, and Theorem~4.8 of \cite{brady01}.
\end{proof}
These relations are used to define the dual presentation of the braid
group.

\begin{defn}[Dual presentation]\label{def:dual-presentation}
  Let $S = \{s_\pi \mid \pi \in \NC_n^*\}$ be a set indexed by the
  non-trivial noncrossing partitions and let $R$ be the set of
  relations of the form $s_{\pi_1} s_{\pi_2} = s_{\pi_3}$ where such a
  relation is in $R$ if and only if $\delta_{\pi_1} \delta_{\pi_2} =
  \delta_{\pi_3}$ holds in $\braid_n$.  The finite presentation
  $\langle\ S \mid R\ \rangle$ is called the \emph{dual presentation
    of the $n$-strand braid group}.
\end{defn}
The name reflects the following fact established by Tom Brady in
\cite{brady01}.

\begin{thm}[Dual presentation]\label{thm:dual-presentation}
  The abstract group $G$ defined by the dual presentation of the
  $n$-strand braid group is isomorphic to the $n$-strand braid group.
  Concretely, the function that sends $s_\pi \in S$ to the dual simple
  braid $\delta_\pi \in \braid_n$ extends to a group isomorphism
  between $G$ and $\braid_n$.\qed
\end{thm}
%
\section{Parabolic Subgroups}\label{sec:parabolic}

This section establishes properties of subgroups of $\braid_n$ indexed
by noncrossing partitions of $[n]$.  We begin by showing that two
different configuration spaces have isomorphic fundamental groups.

\begin{lem}[Isomorphic groups]\label{lem:iso-groups}
  For a subset $A \subset [n]$ of size $k$, let $B = [n] - A$ and
  $D^B = D - P_B$.  The natural inclusion map $D_A \into D^B$ extends
  to an inclusion map $h\colon \uconf_k(D_A) \into \uconf_k(D^B)$ and
  the induced map $h_*\colon \pi_1(\uconf_k(D_A),P_A) \to
  \pi_1(\uconf_k(D^B),P_A)$ is an isomorphism.
\end{lem}
\begin{proof}
  When $k=1$ both groups are trivial and there is nothing to prove.
  For each element $[f]$ in $\pi_1(\uconf_k(D^B),P_A)$, the path $f$
  can be homotoped so that it never leaves the subdisk $D_A$.  One
  can, for example, modify $f$ so that the configurations first
  radially shrink towards a point in the interior of $D_A$, followed
  by the original representative $f$ on a rescaled version of $D^B$
  strictly contained in $D_A$, followed by a radial expansion back to
  the starting position.  This shows that $h_*$ is onto.  Suppose
  $[f]$ and $[g]$ are elements in $\pi_1(\uconf_k(D_A),P_A)$ such that
  $f$ and $g$ are homotopic based paths in the bigger space
  $\uconf_k(D^B)$.  A very similar modification that can be done here
  so that the entire homotopy between $f$ and $g$ takes place inside
  the subdisk $D_A$, and this shows that $h_*$ is injective.
\end{proof}
We are interested in the images of these isomorphic groups inside the
$n$-strand braid group.

\begin{defn}[Subgroups]\label{def:A-subgroups}
  Let $A$ be a nonempty subset of $[n]$ of size $k$, let $B = [n]-A$ and let
  $D^B = D - P_B$.  For each such $A$, we define a map from $\braid_k$
  to $\braid_n$ whose image is a subgroup we call $\braid_A$.  When
  $k=1$, $\braid_k$ is trivial, the only possible map is the trivial
  map and $\braid_A$ is the trivial subgroup of $\braid_n$.  For
  $k>1$, the subspace $D_A$ is a disk, by Corollary~\ref{cor:braid-P}
  the group $\pi_1(\uconf_k(D_A),P_A)$ is isomorphic to $\braid_k$,
  and by Lemma~\ref{lem:iso-groups} $\pi_1(\uconf_k(D^B),P_A)$ is also
  isomorphic to $\braid_k$.  Let $g\colon \uconf_k(D^B) \into
  \uconf_n(P)$ be the natural embedding that sends a set $U \in
  \uconf_k(D^B)$ to $g(U) = U \cup P_B \in \uconf_n(P)$ and note that
  $g(P_A) = P$.  The group $\braid_A$ is the subgroup
  $g_*(\pi_1(\uconf_k(D^B),P_A))$.
\end{defn}
Note that for every $A \subset [n]$, the rotation braid $\delta_A$ is
an element of the subgroup $\braid_A$.  We are also interested in the
braids that fix a subset of vertices in $V$.

\begin{defn}[Fixing vertices]\label{def:fixed-strands}
  Let $\alpha$ be a braid in $\braid_n$ represented by $f$.  We say
  that \emph{$f$ fixes the vertex $p_i \in P$} if the strand that
  starts at $p_i$ is a constant path, i.e. $f^i(t) = p_i$ for all
  $t\in [0,1]$. Similarly, \emph{$f$ fixes $P_B \subset P$} if it fixes each
  $p_i\in P_B$ and a braid \emph{$\alpha$ fixes $P_B$} if it has some
  representative $f$ that fixes $P_B$.  Let $\fix_n(B) = \{ \alpha \in
  \braid_n \mid \alpha \textrm{ fixes } P_B\}$.  Since
  the special representatives can be concatenated and inverted while
  remaining special, $\fix_n(B)$ is a subgroup of $\braid_n$.
\end{defn}
The two constructions describe the same set of subgroups.

\begin{lem}[$\fix_n(B) = \braid_A$]\label{lem:fix=para}
  If $A$ and $B$ are nonempty sets that partition $[n]$, then the fixed
  subgroup $\fix_n(B)$ is equal to the subgroup $\braid_A$.
\end{lem}
\begin{proof}
  The map $g$ described in Definition~\ref{def:A-subgroups} shows that
  every braid in $\braid_A$ has a representative that fixes $P_B$.
  Thus $\braid_A \subset \fix_n(B)$.  Conversely, let $\alpha$ be a
  braid in $\fix_n(B)$ and let $f$ be a representative of $\alpha$ that
  fixes $P_B$.  Since the vertices in $P_B$ are always occupied, $f$
  restricted to the strands that start in $P_A$ is a loop in the space
  $\uconf_k(D^B)$.  Thus $\alpha$ is in the subgroup
  $g_*(\pi_1(\uconf_k(D^B),P_A)) = \braid_A$, which means that
  $\fix_n(B) \subset \braid_A$ and the two groups are equal.
\end{proof}
The subgroups of the form $\braid_A = \fix_n(B)$ are used to construct
the \emph{dual parabolic subgroups} of $\braid_n$.

\begin{defn}[Dual parabolic subgroups]\label{def:dual-para}
  Let $\pi = \{A_1,A_2,\ldots,A_k\} \in \NC_n$ be a noncrossing partition.
  We define the subgroup $\braid_\pi \subset
  \braid_n$ to be the internal direct product $\braid_\pi =
  \braid_{A_1} \times \ldots \times \braid_{A_k}$.  These are pairwise
  commuting subgroups that intersect trivially because they are moving
  points around in disjoint standard subdisks $D_{A_i}$.  We call
  $\braid_\pi$ a \emph{dual parabolic subgroup}.  The subgroup
  $\braid_A$ is an \emph{irreducible dual parabolic} because it
  corresponds to the irreducible noncrossing partition $\pi_A$.  And
  when $A = [n]-\{i\}$ we call $\braid_A$ a \emph{maximal irreducible
    dual parabolic}.
\end{defn}
The adjective ``dual'' is used to distinguish them from the standard
parabolic subgroups associated with the standard presentation of
$\braid_n$, but the two collections of subgroups are closely related.
To make the connection between them precise, we pause to discuss the
stardard presentation of $\braid_n$ and the standard parabolic
subgroups derived from this presentation.  We begin by recalling some
of the basic relations satisfied by a pair of rotations indexed by
edges in the disk $D$.

\begin{defn}[Basic relations]\label{def:basic-rels}
  Let $e$ and $e'$ be two edges in $D$.  Since they are straight line
  segments connecting vertices of the convex polygonal disk $D$, $e$
  and $e'$ are either disjoint, share a commmon vertex, or they cross
  at some point in the interior of each edge.  When $e$ and $e'$ are
  disjoint, the rotations $\delta_e$ and $\delta_{e'}$ \emph{commute},
  i.e. $\delta_e \delta_{e'} = \delta_{e'} \delta_e$.  When $e$ and
  $e'$ share a common vertex, $\delta_e$ and $\delta_{e'}$
  \emph{braid}, i.e. $\delta_e\delta_{e'}\delta_e =
  \delta_{e'}\delta_e\delta_{e'}$.  We call these commuting and
  braiding relations the \emph{basic relations} of $\braid_n$.  When
  $e$ and $e'$ cross, no basic relation between $\delta_e$ and
  $\delta_{e'}$ is defined.
\end{defn}
Artin showed that a small set of rotations indexed by edges in $D$ is
sufficient to generate $\braid_n$ and that the basic relations between
them are sufficient to complete a presentation of $\braid_n$.

\begin{defn}[Standard presentation]\label{def:std-pres}
  Consider the abstract group
	\begin{equation}\label{eq:std-pres}
	G = \left< 
      s_1,\ldots,s_{n-1} \left|
      \begin{array}{cl}
        s_i s_j = s_i s_j &\textrm{ if } \card{i-j}>1 \\
        s_i s_j s_i = s_j s_i s_j &\textrm{ if } \card{i-j}=1
      \end{array}
    \right.\right>
	\end{equation}
  This is the \emph{standard presentation of the $n$-strand braid
    group} and $S = \{s_1,s_2,\ldots,s_{n-1}\}$ is its \emph{standard
    generating set}.
\end{defn}
\begin{thm}[{\cite{artin25}}]
The abstract group $G$ defined by the standard presentation of the $n$-strand braid group is isomorphic to the $n$-strand braid group. Concretely, the function that sends $s_i$ to $\delta_e \in \braid_n$ where $e$ is the edge connecting $p_i$ and $p_{i+1}$ extends to an isomorphism between $G$ and $\braid_n$.\qed
\end{thm}
Standard parabolic subgroups are generated by subsets of $S$.

\begin{defn}[Standard parabolic subgroups]\label{def:std-para}
  Let $S = \{s_1,s_2,\ldots,s_{n-1}\}$ be the standard generating set
  for the abstract group $G \cong \braid_n$.  For any subset $S'
  \subset S$, the subgroup $\langle S' \rangle \subset G$ generated by
  $S'$ is called a \emph{standard parabolic subgroup}.  The subsets of
  the form $S_{[i,j]} = \{ s_\ell \mid i \leq \ell < j\}$ generate the
  \emph{irreducible standard parabolic subgroups} of $G$.  These
  subsets correspond to sets of edges forming a connected subgraph in
  the boundary of $D$.
\end{defn}
We record two standard facts about the irreducible standard parabolic
subgroups of the braid groups: (1) they are isomorphic to braid groups
and (2) they are closed under intersection.

\begin{prop}[Isomorphisms]\label{prop:std-isom}
	Let $i,j\in [n]$ with $i<j$. Then the irreducible subgroup generated
	by $S_{[i,j]} = \{ s_\ell \mid i \leq \ell < j\}$ is isomorphic to 
	$\braid_k$, where $k = j-i+1$.
\end{prop}
\begin{proof}
It is immediate from the standard presentation that $\braid_k$
maps onto $S_{[i,j]}$. That this map is injective follows from the
solution of the word problem, \cite[§3]{artin25}.
\end{proof}

\begin{prop}[Intersections]\label{prop:std-inter}  
  For all subsets $S', S'' \subset S$, the intersection $\langle S'
  \rangle \cap \langle S'' \rangle$ is equal to the standard parabolic
  subgroup $\langle S' \cap S'' \rangle$ .  Moreover, when both
  $\langle S' \rangle$ and $\langle S'' \rangle$ are irreducible
  subgroups, so is $\langle S' \cap S'' \rangle$.
\end{prop}
\begin{proof}
Immediate from Proposition~\ref{prop:std-isom}.
\end{proof}
For later use we record the following fact.

\begin{lem}\label{lem:abelian}
The map $\braid_n \to \Z$ that takes $\delta_\pi$ to $\rk(\pi)$ is the abelianization of $\braid_n$.
\end{lem}
\begin{proof}
The map is well-defined by Proposition~\ref{prop:relations} because
if $\delta_{\pi_1}\delta_{\pi_2} = \delta_{\pi_3}$ then the rank
function on $\braid_n$ satisfies $\rk \pi_1 + \rk \pi_2 = \rk \pi_3$.
The fact that it is the full abelianization is immediate from the
standard presentation \eqref{eq:std-pres}.
\end{proof}
With this we end our digression on standard 
presentations and return to dual structure.

\begin{lem}[Maximal dual parabolics]\label{lem:para-1}
  The intersection of two maximal irreducible dual parabolic subgroups
  is an irreducible dual parabolic subgroup.  In particular, for all
  $n>0$ and for all $i,j \in [n]$,
	\[
    \fix_n(\{i,j\}) = \fix_n(\{i\}) \cap \fix_n(\{j\}).
	\]
\end{lem}
\begin{proof}
  For every pair of vertices $p_i$ and $p_j$ one can select a sequence
  $E =(e_1, \ldots,e_{n-1})$ of edges in $D$ so that together, in this
  order, they form an embedded path through all vertices of $D$
  starting at $p_i$ and ending at $p_j$.  Because the rotations
  $\delta_e$ for $e \in E$ satisfy the necessary basic relations
  (Definition~\ref{def:basic-rels}), the function that sends $s_\ell
  \in S$ to the rotation $\delta_{e_\ell}$ extends to group
  homomorphism from $g\colon G \to \braid_n$.  In fact $g$ is a group
  isomorphism since up to homeomophism of $\C$ this is just the usual
  isomorphism between the abstract group $G$ and the braid group
  $\braid_n$.  Under this isomorphism the subgroup $\langle S_{[2,n]}
  \rangle$ is sent to the subgroup $\fix_n(\{i\})$, the subgroup
  $\langle S_{[1,n-1]} \rangle$ is sent to the subgroup
  $\fix_n(\{j\})$, and the subgroup $\langle S_{[2,n]} \rangle$ is
  sent to the subgroup $\fix_n(\{i,j\})$.
  Proposition~\ref{prop:std-inter} completes the proof.
\end{proof}
\begin{lem}[Relative maximal dual parabolics]\label{lem:para-ell}
  The intersection of two irreducible dual parabolic subgroups that
  are both maximal in a third irreducible dual parabolic subgroup is
  again an irreducible dual parabolic subgroup.  In other words, for
  all $n>0$ and for all $\{i\},\{j\},C \subset [n]$,
	\[
    \fix_n(C \cup \{i,j\}) = \fix_n(C \cup \{i\}) \cap \fix_n(C
  \cup \{j\}).
	\]
\end{lem}
\begin{proof}	
  When $C$ is empty, the statement is just Lemma~\ref{lem:para-1} and
  when $C$ is $[n]$ there is nothing to prove.  When $C$ is proper and
  non-empty, all three groups are contained in $\fix_n(C) = \braid_A
  \cong \braid_k$ where $k$ is the size of $A= [n]-C$.  There is a
  homeomorphism from $D_A$ to the regular $k$-gon that sends vertices
  to vertices, so Lemma~\ref{lem:homeos}, shows that the assertion
  now follows by applying Lemma~\ref{lem:para-1} to this $k$-gon.
\end{proof}
\begin{prop}[Arbitrary dual parabolics]\label{prop:para-induct}
  Every proper irreducible dual parabolic subgroup of $\braid_n$ is
  equal to the intersection of the maximal irreducible dual parabolic
  subgroups that contain it and, as a consequence, the collection of
  irreducible dual parabolics is closed under intersection.  In
  other words, for all $n>0$ and for every non-empty $B \subset [n]$,
	\[
    \fix_n(B) = \bigcap_{i \in B} \fix_n(\{i\})
	\]
  and, as a consequence, for all non-empty $C, D \subset B$, 
	\[
    \fix_n(C \cup D) = \fix_n(C) \cap \fix_n(D).
	\]
\end{prop}
\begin{proof}
  When $B$ is a singleton, the result is trivial and when $B$ has size
  $2$ both claims are true by Lemma~\ref{lem:para-1}, so suppose that
  both claims hold for all subsets of size at most $k$ with $k>1$ and
  let $B$ be a subset of size $k+1$.  If $i$ and $j$ are elements in
  $B$, and $C = B-\{i,j\}$, then $\fix_n(B) = \fix_n(C \cup \{i,j\})$
  which is equal to $\fix_n(C \cup \{i\}) \cap \fix_n(C \cup \{j\})$
  by Lemma~\ref{lem:para-ell}.  By applying the second inductive claim
  to the sets $C \cup \{i\}$ and $C \cup \{j\}$ and simplifying
  slightly we can rewrite this as $\fix_n(C) \cap \fix_n(\{i\}) \cap
  \fix_n(\{j\})$.  Applying the first inductive claim to the set $C$
  shows that first claim holds for $B$ and the second claim for $B$
  follows as an immediate consequence. This completes the induction
  and the proof.
\end{proof}
%

\part{Complexes}\label{part:complexes}

In this part, we study complexes built out of ordered simplices, specifically
how they can be equipped with an orthoscheme metric.

\section{Ordered Simplices}\label{sec:ord-simp}

An ordered simplex is a simplex with a fixed linear ordering of its
vertex set. Complexes built out of ordered simplices are often used
as explicit models.
Eilenberg and Steenrod, for example, use ordered simplicial
complexes \cite{eilenbergsteenrod52}. We follow Hatcher in using the
more flexible $\Delta$-complexes \cite{hatcher02}.

\begin{defn}[Ordered simplices]\label{def:ordered-simplex}
  A \emph{$k$-simplex} is the convex hull of $k+1$ points $p_0, p_1,
  \ldots, p_k$ in general position in a sufficiently high-dimensional real
  vector space $E$.  An \emph{ordered $k$-simplex} is a $k$-simplex
  together with a fixed linear ordering of its $k+1$ vertices.
  We write
  $\sigma = [p_0,p_1,\ldots,p_k]$
  for an ordered $k$-simplex $\sigma$ with vertex set
  $\{p_0,p_1,\ldots,p_k\} \subset E$ where the vertices are ordered left-to-right:
  $p_i < p_j$ in the linear order
  if and only if $i < j$ in the natural numbers.
  An isomorphism of ordered simplices is an affine bijection
  $[p_0,\ldots,p_k] \to [p_0',\ldots,p_k']$ that takes $p_i$ to $p_i'$.
\end{defn}
Let $E$ be a vector space containing an ordered $k$-simplex $\sigma$.
To facilitate computations, we establish a standard coordinate system
on the smallest affine subspace of $E$ containing $\sigma$ which both
identifies this subspace with $\R^k$ and also reflects the linear
ordering of its vertices.

\begin{defn}[Standard coordinates]\label{def:std-coord}
  Let $\sigma = [p_0,p_1,\ldots,p_k]$ be an ordered $k$-simplex. We
  take the ambient vector space $E$ to have origin $p_0$ and to be
  spanned by $p_1,\ldots,p_k$. For each $i \in [k]$, let
  $\vv_i = p_i - p_{i-1}$ be the vector from $p_{i-1}$ to $p_i$ so that
  that $\mathcal{B} = (\vv_1,\vv_2,\ldots,\vv_k)$ is an ordered
  basis for $E$.  We call $\mathcal{B}$ the \emph{standard ordered
    basis of $\sigma$}.  In this basis $p_j = \sum_{i=1}^j \vv_i$ and
  \begin{align*}
  x &= (x_1,x_2,\ldots,x_k)_{\mathcal{B}} = x_1 \vv_1 + \ldots + x_k \vv_k\\
  &= (-x_1)p_0
  + (x_1-x_2)p_1 + \cdots + (x_{k-1}-x_k)p_{k-1} + (x_k)p_k\\
  &=(1-x_1)p_0 + (x_1-x_2)p_1 + \cdots + (x_{k-1}-x_k)p_{k-1}
  + (x_k)p_k
  \end{align*}
  since $p_0$ is the origin.  Since the coefficients in the last equation
  are barycentric coordinates on $E$, we see that
  $(x_1,x_2,\ldots,x_k)_{\mathcal{B}}$ is in $\sigma$ if and only if
  $1 \geq x_1 \geq x_2 \geq \cdots \geq x_k \geq 0$.  In particular,
  the facets of $\sigma$ determine the $k+1$ hyperplanes given by the
  equations $x_1 = 1$, $x_i = x_{i+1}$ for $i \in [k-1]$ and $x_k=0$,
  respectively.
  If $\sigma$ and $\sigma'$ are ordered simplices with oriented
  bases $(\vv_1,\ldots,\vv_k)$ and $(\vv_1',\ldots,\vv_k')$, the unique
  isomorphism $\sigma \to \sigma'$ takes $\sum_i \alpha_i \vv_i$ to
  $\sum_i \alpha_i \vv_i'$.
\end{defn}
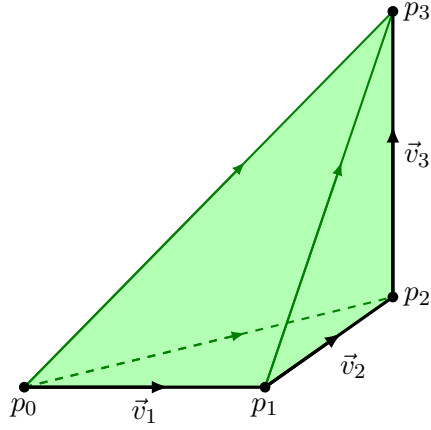
\begin{figure}
  \begin{tikzpicture}[scale=1]
    \begin{scope}[x={(3.2cm,0)},y={(1.7cm,1.2cm)},z={(0cm,3.8cm)}]
      \def\r{.6}
      \begin{scope}[thick,GreenPoly]
        \coordinate (t0) at (0,0,0);
        \coordinate (t1) at (1,0,0);
        \coordinate (t2) at (1,1,0);
        \coordinate (t3) at (1,1,1);
        \coordinate (e1) at (.5,0,0);
        \coordinate (e2) at (1,.5,0);
        \coordinate (e3) at (1,1,.5);
        \filldraw (t0)--(t1)--(t2)--(t3)--cycle;
        \draw[fill=green!50!black,-latex] (t0)--(\r,\r,\r);
        \draw[dashed] (t0)--(t2); 
        \draw[fill=green!50!black,dashed,-latex] (t0)--(\r,\r,0);

        \draw (t1)--(t3); 
        \draw[fill=green!50!black,-latex] (t1)--(1,\r,\r);
      \end{scope}
      \begin{scope}[very thick,color=black]
        \draw (t0)--(t1)--(t2)--(t3);
        \draw[-latex] (t0)--(\r,0,0);
        \draw[-latex] (t1)--(1,\r,0);
        \draw[-latex] (t2)--(1,1,\r);
        \draw (t0) node[anchor=north]{$p_0$};
        \draw (t1) node[anchor=north]{$p_1$};
        \draw (t2) node[anchor=west]{$p_2$};
        \draw (t3) node[anchor=west]{$p_3$};
        \draw (e1) node[anchor=north]{$\vv_1$};
        \draw (e2) node[anchor=north west]{$\vv_2$};
        \draw (e3) node[anchor=west]{$\vv_3$};
        \foreach \n in {t0,t1,t2,t3} {\filldraw (\n) circle (.5mm);}
      \end{scope}
    \end{scope}
  \end{tikzpicture}
  \caption{An ordered $3$-simplex.  In standard coordinates $p_0$ is
    the origin, $p_1 = (1,0,0)$, $p_2 = (1,1,0)$ and $p_3 = (1,1,1)$
    with respect to the ordered basis $\mathcal{B} =
    \{\vv_1,\vv_2,\vv_3\}$.\label{fig:ordered-3-simplex}}
\end{figure}
\begin{exmp}[Standard Coordinates]\label{ex:std-coord}
  Figure~\ref{fig:ordered-3-simplex} shows an ordered $3$-simplex
  $\sigma$. The vector $\vv_1$ is from $p_0$ to $p_1$, the vector $\vv_2$
  is from $p_1$ to $p_2$ and the vector $\vv_3$ is from $p_2$ to $p_3$.
  With respect to the ordered basis $\mathcal{B} = \{\vv_1,\vv_2,\vv_3\}$
  with $p_0$ located at the origin, $p_1 = (1,0,0)$, $p_2 = (1,1,0)$
  and $p_3 = (1,1,1)$.
\end{exmp}
Faces of ordered simplices are ordered by restriction. As such they have standard coordinates which can be described as follows.

\begin{lem}[Facets]\label{rem:subsimplices}
Let $\sigma = [p_0,\ldots,p_k]$ be an ordered simplex with ordered basis $\mathcal{B} = (\vv_1,\ldots,\vv_k)$. The ordered basis $\mathcal{B}'$ of the facet $\tau = [p_0,\ldots,p_{i-1},p_{i+1},\ldots,p_k]$ is
\[
\mathcal{B}' = \left\{\begin{array}{ll}
(\vv_2,\ldots,\vv_k)&\text{if }0 = i\\
(\vv_1,\ldots,\vv_{i-1},\vv_i + \vv_{i+1},\vv_{i+2},\ldots,\vv_k)&\text{if }0 < i < k\\
(\vv_1,\ldots,\vv_{k-1})&\text{if }i = k\text{.}
\end{array}\right.
\]
\par\vspace{-\baselineskip}\qed
\end{lem}
In anticipation of Definition~\ref{def:ortho-cplx} the following
definition is modeled on \cite[Definition~I.7.2]{brihae}.

\begin{defn}[$\Delta$-complex]\label{def:delta-complex}
  Let $(\sigma_\lambda)_{\lambda \in \Lambda}$ be a family of
  ordered simplices with disjoint union
  $X = \bigcup (\sigma_\lambda \times \{\lambda\})$.
  Let $\sim$ be an equivalence relation on $X$ and let
  $K = X/\sim$.
  Let $p \colon X \to K$ be the quotient map and
  $p_\lambda \colon \sigma_\lambda \to K, x \mapsto p(x,\lambda)$
  its restriction to $\sigma_\lambda$.

  We say that $K$ is a \emph{$\Delta$-complex} if:
  \begin{enumerate}
  \item the restriction of $p_\lambda$ to the interior of
  $\sigma_\lambda$ is injective;
  \item for $\lambda \in \Lambda$ and every face  $\tau$
  of $\sigma_\lambda$ there is a $\lambda' \in \Lambda$
  and an isomorphism of ordered simplices
  $h \colon \tau \to \sigma_{\lambda'}$
  such that $p_\lambda|_\tau = p_{\lambda'} \circ h$;\label{item:delt_face}
  \item if $\lambda,\lambda' \in \Lambda$ and interior points
  $x \in \sigma_\lambda$ and $x' \in \sigma_{\lambda'}$ are such
  that $p_\lambda(x) = p_\lambda(x')$ then there is an isomorphism
  of ordered simplices $h \colon \sigma_\lambda \to \sigma_{\lambda'}$
  such that $p_{\lambda'}(h(x)) = p_\lambda(x)$ for
  $x \in \sigma_{\lambda}$.
  \end{enumerate}
  
  We will usually regard $\Delta$-complexes as equipped with
  a structure as above and refer to the simplices $\sigma_\lambda$
  as simplices of $K$. We will also make the identification in
  \eqref{item:delt_face} implicit and regard faces of $\sigma_\lambda$
  as simplices of $K$.
\end{defn}
Turning a simplicial complex into a $\Delta$-complex means to orient the
edges in a consistent way. A setting where a natural orientation exists
is the following.

\begin{prop}[Cayley graphs and $\Delta$-complexes]\label{prop:cayley-delta}
  Let $G$ be a group and let $f\colon G \to (\R,+)$ be a group homomorphism.
  Let $S \subset G$ be a set of generators such that $f(s) > 0$
  for every $s\in S$.
  The right Cayley graph $\Gamma = \cay(G,S)$ is a simplicial graph
  whose flag complex $X = \flag(\Gamma)$ can be turned into
  a $\Delta$-complex.
\end{prop}
\begin{proof}
  The right Cayley graph has no doubled edges because
  $f(g^{-1}) = -f(g)$ for $g \in G$ so that at most one of $g$
  and $g^{-1}$ is in $S$. It has no loops because $f(1) = 0$
  so that $1 \not\in S$.
  
  We define a relation $\le$ on $G$ by declaring that $g \le gs$
  for $s \in S$
  and taking the reflexive transitive closure. The homomorphism $f$
  guarantees that this is a partial order on $G$. Any two adjacent
  vertices are comparable so the restriction to a simplex is a total
  order.
\end{proof}
%

\section{Orthoschemes}\label{sec:orthoschemes}

The goal of this section is to equip certain $\Delta$-complexes
with a piecewise Euclidean metric.

\begin{defn}[Orthoscheme]\label{def:ortho-new}
  Let $E$ be a Euclidean vector space and let $\sigma \subseteq E$
  be a simplex. Then $\sigma$ with the induced metric is called a
  \emph{Euclidean simplex}. If $\sigma$ is an ordered simplex and
  the associated ordered basis is orthogonal then $\sigma$ is
  an \emph{orthoscheme}. If it is an orthonormal basis then
  $\sigma$ is a \emph{standard orthoscheme}.
\end{defn}
\begin{defn}[Orthoscheme complex]\label{def:ortho-cplx}
  An \emph{orthoscheme complex} is a $\Delta$-complex
	where each simplex has been given the metric of an
	orthoscheme in such a way that the isomorphisms in the definition
	of a $\Delta$-complex are isometries.
  It is equipped with the length pseudometric assigning to two points
  the infimal length of a piecewise affine path. 
\end{defn}
\begin{rem}
Orthoscheme complexes are $M_0$-simplicial complexes in the sense
of \cite[I.7.1]{brihae} so we will not discuss the metric subtleties in
detail. Our main interest concerning the metric is the behavior
with respect to products, which is not among the subtleties.
\end{rem}
\begin{lem}[Orthoscheme complex structures and edge norms]\label{lem:edge-norms}
Let $X$ be a $\Delta$-complex. There is a one-to-one correspondence
between orthoscheme complex structures
on $X$ and maps
$\nrm \colon \Edges(X) \to \R_{>0}$ that satisfy
\begin{equation}\label{eq:norm_additive}
\nrm([p_0,p_1]) + \nrm([p_1,p_2]) = \nrm([p_0,p_2])
\end{equation}
for every $2$-simplex $[p_0,p_1,p_2]$ of $X$.
\end{lem}
\begin{proof}
If $X$ is equipped with a orthoscheme complex-structure,
defining $\nrm([p,q]) = \norm{q-p}^2$ gives a map
satisfying condition \eqref{eq:norm_additive} (corresponding
to the right angle in $p_1$).

Conversely, the squares of edge lengths of an orthoscheme
need to satisfy \eqref{eq:norm_additive} and this is the only
requirement for a well-defined assignment. The unique
isomorphism of ordered simplicial complexes between
orthoschemes with same edge lengths is an isometry.
Hence equipping the simplices of a $\Delta$-complex
with a orthoscheme metric satisfying \eqref{eq:norm_additive}
gives rise to an orthoscheme complex.
\end{proof}
We can use this characterization of orthoscheme complexes to extend
Proposition~\ref{prop:cayley-delta}.

\begin{prop}[Cayley graphs and orthoschemes]\label{prop:cayley-orth}
  Let $G$ be a group, $f\colon G \to (\R,+)$ be a group homomorphism,
  and let $S$ be a generating set of $G$ with $f(s) > 0$ for every $s \in S$.
  The right Cayley graph $\Gamma
  = \cay(G,S)$ is a simplicial graph whose flag complex $X =
  \flag(\Gamma)$ can be turned into an orthoscheme complex using $f$.
\end{prop}
\begin{proof}
  By Proposition~\ref{prop:cayley-delta} $X$ is a $\Delta$-complex and
  we claim that
  \[
  \nrm([g,gs]) \defeq f(gs) - f(g) = f(s) > 0
  \]
  satisfies 
  \eqref{eq:norm_additive}.
  
  Indeed
  \begin{align*}
  \nrm([g,gss']) &= f(gss') - f(g)\\
  &= (f(gss') - f(gs)) + (f(gs) - f(g))\\
  &= \nrm([gss',gs]) + \nrm([gs,g])\text{.}
  \end{align*}
  since $f$ is a homomorphism.
\end{proof}

\begin{defn}[Dual braid complex]\label{def:dual-braid-cplx}
  Let $S = \DS_n^*$ be the set of non-trivial dual simple braids in the braid group
  $\braid_n$. By Theorem~\ref{thm:dual-presentation} the set $S
  \subset \braid_n^*$ generates the group and by
  Lemma~\ref{lem:abelian} the abelianization map $f\colon \braid_n
  \to \Z$ sends the non-trivial dual simple braid $\delta_\pi \in S$ to the
  positive integer $f(\delta_\pi) = \rk(\pi)$.  By
  Proposition~\ref{prop:cayley-orth}, the flag complex $X =
  \flag(\Gamma)$ of the simplicial graph $\Gamma = \cay(\braid_n,S)$
  can be turned into an orthoscheme complex using $f$ to compute the
  norm of each edge.  The resulting orthoscheme complex is called the
  \emph{dual braid $n$-complex} and denoted $\comp(\braid_n)$.
  Note that every edge of $\comp(\braid_n)$ is naturally labeled
  by an element of $S = \DS_n^*$ or, equivalently, by a non-trivial
  noncrossing
  partition. More generally, every simplex is naturally labeled by
  a chain of $\NC_n$.
\end{defn}
It is clear from the construction that $\braid_n$ acts freely on $\comp(\braid_n)$
and that $\comp(\braid_n)$ is covered by translates of the full subcomplex supported
on $\DS_n$, which is therefore a fundamental domain. The key feature,
implying that $\braid_n \backslash \comp(\braid_n)$ is a classifying space for
$\braid_n$, is:

\begin{thm}[\cite{brady01}]\label{thm:cpt-quot}
  The complex $\comp(\braid_n)$ is contractible.
\end{thm}
In fact, it is shown in \cite{bradymccammond10} $\comp(\braid_n)$ is
$\cat(0)$ when $n < 6$ and in \cite{haettelkielakschwer16} this was
extended to the case $n=6$.

\section{Products}\label{sec:products}

The main advantage of working with ordered simplices
and $\Delta$-complexes is
that they admit well-behaved products.

\begin{exmp}[Products of simplices]
  The product of two $1$-simplices is a quadrangle.
  It can be subdivided into two triangles in two ways
  but neither of these is distinguished. More generally,
  the product of two (positive-dimensional)
  simplices is not a simplex nor does it have a canonical
  simplicial subdivision.
\end{exmp}
In contrast, we will see that the product of two simplices
whose vertices are totally
ordered admits a canonical subdivision into chains. We start
by looking at finite products of edges first, i.e.\ cubes.

\begin{exmp}[Subdivided cubes]\label{exmp:cubes}
  Let $\R^k$ be a $k$-dimensional real vector space with a fixed
  ordered basis $\mathcal{B} = \{\vv_1,\ldots,\vv_k\}$.  The
  \emph{unit $k$-cube $\cube_k$} in $\R^k$ is the set of vectors where
  each coordinate is in the interval $[0,1]$ and its vertices are the
  points where every coordinate is either $0$ or $1$.  There is a
  natural bijection between the vertex set of $\cube_k$ and the set of
  all subsets of $[k]$: simply send each vertex to the set of indices
  of the coordinates where the value is $1$.  If we partially order
  the subsets of $[k]$ by inclusion (to form the Boolean lattice
  $\bool_k$), this partially orders the vertices of $\cube_k$, 
  and by sending $B\subset [k]$ to the vector $\one_B = \sum_{i\in B} \vec{v}_i$, 
  we obtain a convenient labeling for the vertices.
  At the extremes, we write $\one = \one_{[k]} = (1,1,\ldots,1)$ and 
  $\zero = \one_{\emptyset} = (0,0,\ldots,0)$.

  Let $\mathcal{H}$ be the collection of hyperplanes $H_{ij}$ in
  $\R^k$ defined by the equations $x_i = x_j$ for $i \neq j\in [k]$.
  There is a minimal cellular subdivision of $\cube_k$ for which
  $\cube_k \cap H_{ij}$ is a subcomplex for all $i \neq j \in [k]$ and
  it is a simplicial subdivision.  The subdivision has $k!$
  top-dimensional simplices and the partial order on the vertices of
  $\cube_k$ is a linear order when restricted to each simplex.  In
  particular, this subdivided $k$-cube is a $\Delta$-complex.
\end{exmp}
\begin{rem}\label{rem:prod-euclid}
	When our selected basis $\mathcal{B}$ is orthonormal, $\cube_k$ is a
	regular Euclidean unit cube and the top-dimensional simplices
	are orthoschemes. In other words, the $k!$ simplices in the 
	simplicial structure for $\cube_k$ correspond to the $k!$ ways to take
	$k$ steps from $\zero$ to $\one$ in the coordinate directions.
	A $3$-orthoscheme from the simplicial structure on $\cube_3$ is shown in
	Figure~\ref{fig:orthoscheme}.
\end{rem}
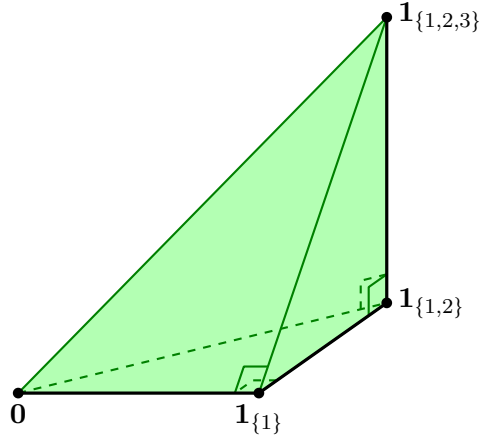
\begin{figure}
  \begin{tikzpicture}[scale=1]
    \begin{scope}[thick,GreenPoly,x={(3.2cm,0)},y={(1.7cm,1.2cm)},z={(0cm,3.8cm)}]
      \filldraw (0,0,0) coordinate(t0)--(1,0,0) coordinate(t1) --
      (1,1,0) coordinate(t2)--(1,1,1) coordinate(t3)--cycle;
      \draw[dashed] (t0)--(t2);
      \drawAngleD{(t1)}{(0,.1414,0)}{(-.1,0)};
      \drawAngleD{(t2)}{(-.0707,-.0707,0)}{(0,0,.1)};
      \drawAngle{(t1)}{(0,.0707,.0707)}{(-.1,0,0)};
      \drawAngle{(t2)}{(0,-.1414,0)}{(0,0,.1)}; \draw (t1)--(t3);
    \end{scope}
    \begin{scope}[very thick,color=black]
      \draw (t0)--(t1)--(t2)--(t3);
      \draw (t0) node[anchor=north]{$\zero$};
      \draw (t1) node[anchor=north]{$\one_{\{1\}}$};
      \draw (t2) node[anchor=west]{$\one_{\{1,2\}}$};
      \draw (t3) node[anchor=west]{$\one_{\{1,2,3\}}$};
      \foreach \n in {t0,t1,t2,t3} {\filldraw (\n) circle (.5mm);}
    \end{scope}
  \end{tikzpicture}
  \caption{A $3$-orthoscheme from $\zero$ to $\one = \one_{\{1,2,3\}}$
    inside $\cube_3$.  The edges of the piecewise geodesic path are
    thicker and darker than the others.\label{fig:orthoscheme}}
\end{figure}
\begin{exmp}[Products of subdivided cubes]
	The product of $\cube_k$ and $\cube_\ell$ is naturally identified 
	the unit cube $\cube_{k+\ell}$.	Using the simplicial subdivision 
	given in Example~\ref{exmp:cubes}, we obtain a canonical 
	$\Delta$-complex structure for the product $\cube_{k+\ell}$. 
	Selecting top-dimensional simplices $\sigma$ in $\cube_k$ and
	$\tau$ in $\cube_\ell$ corresponds to a product of simplices
	$\sigma\times\tau$ in $\cube_{k+\ell}$, which then inherits a
	simplicial subdivision from that of $\cube_{k+\ell}$.
\end{exmp}
Since any ordered simplex can be considered as a top-dimensional
simplex in the subdivision from Example~\ref{exmp:cubes}, we
can use the simplicial structure for $\cube_{k+\ell}$ to describe the
product of two ordered simplices as a $\Delta$-complex.

\begin{exmp}[Product of ordered simplices]\label{exmp:prod-ord-simp}
  Let $\sigma$ and $\tau$ be ordered simplices of dimension $k$ and
  $\ell$, respectively, with $\sigma = [v_0,v_1,\ldots,v_k]$ and $\tau
  = [u_0,u_1,\ldots,u_\ell]$.  In standard coordinates $\sigma
  \subset \R^k$ is the set of points $x = (x_1,x_2,\ldots,x_k) \in
  \R^k$ satisfying the inequalities $1 \geq x_1 \geq x_2 \geq \cdots
  \geq x_k \geq 0$.  Similarly, $\tau \subset \R^\ell$ is the set of
  points $y = (y_1,y_2,\ldots,y_\ell) \in \R^\ell$ satisfying the
  inequalities $1 \geq y_1 \geq y_2 \geq \cdots \geq y_\ell \geq 0$.
  The product $\sigma \times \tau$ is the set of points
  $(x_1,\ldots,x_k,y_1,\ldots,y_\ell) \in \R^k \times \R^\ell$
  satisfying both sets of inequalities.  Let $\mathcal{H}$ be the
  collection of $k\cdot \ell$ hyperplanes $H_{ij}$ defined by the
  equations $x_i = y_j$, with $i \in [k]$ and $j \in [\ell]$.  When we
  minimally subdivide the polytope $\sigma \times \tau$ so that for
  every $i \in [k]$ and every $j \in [\ell]$, $H_{ij} \cap (\sigma
  \times \tau)$ is a subcomplex of the new cell structure, then the
  new cell structure is a simplicial complex which contains
  $\binom{k+\ell}{k}$ simplices of dimension $k+\ell$.  The points in
  the interiors of these top-dimensional simplices correspond to
  points $(x_1,\ldots,x_k,y_1,\ldots,y_\ell)$ where all $k+\ell$
  coordinates are distinct and the simplex containing this point is
  determined by the $xy$-pattern of the coordinates when arranged
  in decreasing linear order. For example, if $k=2$, $\ell=1$ and 
  $x_0 > x_1 > y_0 > x_2 > y_1$ then its pattern is $xxyxy$ and all 
  generic points with this pattern belong to the same top-dimensional 
  simplex.  The natural partial
  order on the vertices of $\sigma \times \tau$ is given by the rule
  $(v_{i_1},u_{j_1}) \leq (v_{i_2},u_{j_2})$ if and only if $i_1 \leq
  i_2$ and $j_1 \leq j_2$.  This restricts to a linear order on each
  simplex in the new simplicial structure, which turns the result into
  an ordered simplicial complex.
\end{exmp}
\begin{defn}[Product of ordered simplices]\label{def:prod-ord-simp}
Let $\sigma$ and $\tau$ be ordered simplices. The decomposition
of $\sigma \times \tau$ described in
Example~\ref{exmp:prod-ord-simp} is the \emph{canonical
decomposition}. We write $\sigma \oprod \tau$ to denote the
$\Delta$-complex that is $\sigma \times \tau$ with the canonical
decomposition.
\end{defn}
The construction described in Definition~\ref{def:prod-ord-simp} is
the natural generalization of partitioning the unit square in the
first quadrant by the diagonal line where the two coordinates are
equal.  It readily generalizes to finite products of ordered
simplices.

\begin{exmp}[Finite products]\label{exmp:finite-products} 
  Let $\sigma_1, \sigma_2, \ldots, \sigma_m$ be ordered simplices of
  dimension $k_1, k_2,\ldots, k_m$, respectively, and view $\sigma_1
  \times \sigma_2 \times \cdots \times \sigma_m$ as a subset of
  $\R^{k_1+\cdots+k_m}$ with coordinates given by concatenating the
  standard ordered bases.  Let $\mathcal{H}$ be the finite collection
  of hyperplanes defined by an equation setting a canonical coordinate
  in one factor equal to a canonical coordinate in different factor.
  The minimal subdivision of the product cell complex $X = \sigma_1
  \times \sigma_2 \times \cdots \times \sigma_m$ so that for every
  hyperplane $H \in \mathcal{H}$, $H \cap X$ is a subcomplex in the
  new cell structure is a simplicial complex with $N$ simplices of
  dimension $k_1+k_2+\cdots+k_m$, where $N$ is the multinomial
  coefficient $\binom{k_1+k_2+\cdots+k_m}{k_1,k_2,\ldots,k_m}$.
  This illustrates that $\oprod$ is associative.
\end{exmp}
The canonical subdivision of products of ordered simplices also
readily extends to the product of $\Delta$-complexes.

\begin{defn}[Products of $\Delta$-complexes]\label{def:prod-delta}
  Let $X$ and $Y$ be $\Delta$-complexes.  The product complex $X
  \times Y$ carries a canonical $\Delta$-complex structure which
  can be described as follows. Let $p_\sigma \colon \sigma \to X$ and
  $p_\tau \colon \tau \to Y$ be simplices of $X$ and $Y$.
  Then every simplex $\rho$ in the canonical subdivision of
  $\sigma \times \tau$ is a simplex of $X \times Y$ via
  $p_\rho = (p_\sigma \times p_\tau)|_{\rho}$. We denote $X \times Y$
  with this $\Delta$-complex structure by $X \oprod Y$.
\end{defn}
Note that when $\sigma$ and $\tau$ are Euclidean simplices, their
product $\sigma \times \tau$ inherits a Euclidean metric from the
metric product of the Euclidean spaces containing them and when they
are ordered Euclidean simplices, the ordered simplicial complex
$\sigma \oprod \tau$ is constructed out of ordered Euclidean
simplices. 

The product construct described in Definition~\ref{def:prod-ord-simp}
is well-behaved when the factors are orthoschemes or standard
orthoschemes.

\begin{lem}[Products of orthoschemes]\label{lem:prod-ortho}
  If $\sigma$ and $\tau$ are orthoschemes, then $\sigma \oprod \tau$
  is an orthoscheme complex isometric to the metric product $\sigma
  \times \tau$.  Moreover, when $\sigma$ and $\tau$ are standard
  orthoschemes, then every top-dimensional simplex in $\sigma \oprod
  \tau$ is a standard orthoscheme.
\end{lem}
\begin{proof}
  Let $\sigma$ and $\tau$ be ordered Euclidean simplices of dimension
  $k$ and $\ell$, respectively and let $\sigma \oprod \tau$ be the
  simplicial decomposition of the Euclidean polytope $\sigma \times
  \tau$ into Euclidean simplices. If $\mathcal{B}_1$ and $\mathcal{B}_2$
  are the standard ordered bases associated to $\sigma$ and $\tau$, then by  
  Example~\ref{exmp:prod-ord-simp}, the $\Delta$-complex 
  $\sigma\oprod\tau$ is a subcomplex of the unit cube $\cube_{k+\ell}$
  in the ordered basis obtained by concatenating 
  $\mathcal{B}_1$ and $\mathcal{B}_2$. When $\mathcal{B}_1$ and
  $\mathcal{B}_2$ are both orthogonal, $\sigma$ and $\tau$ are orthoschemes
  and the concatenated ordered bases produces a metric Euclidean cube
  for which the simplicial subdivision in Example~\ref{exmp:cubes}
  makes $\cube_{k+\ell}$ into an orthoscheme complex. Hence, the subcomplex
  $\sigma \oprod \tau$ is an orthoscheme complex as well. The analogous result
  for standard orthoschemes follows by considering the case 
  when $\mathcal{B}_1$ and $\mathcal{B}_2$ are both orthonormal.
\end{proof}
As a consequence we get:

\begin{prop}[Products of orthoscheme complexes]\label{prop:prod-oc}
  If $X$ and $Y$ are orthoscheme complexes then $X \oprod Y$ is an
  orthoscheme complex isometric to the metric direct product $X \times
  Y$.\qed
\end{prop}
%

\section{Columns}\label{sec:columns}

In this section we describe a particularly useful type of 
orthoscheme complex, initially defined in \cite{bradymccammond10}.

\begin{exmp}[Orthoschemes and $\R^k$]\label{exmp:ortho-Rn}
  Regard $\R$ as an infinite linear graph with vertex set $\Z$ and
  edges from $i$ to $i+1$. Then $\R^k$ is isometric to the $k$-fold
  product $\R \oprod \cdots \oprod \R$. This complex has vertex set
  $\Z^k$ with simplices on vertices $\vec{x} \le \vec{x} + \one_{B_1}
  \le \ldots \le \vec{x} + \one_{B_\ell}$ for $\emptyset \subsetneq
  B_1 \subseteq \ldots \subseteq B_\ell \subseteq [k]$. We call this
  the standard \emph{orthoscheme tiling of $\R^k$}.  It can also be
  viewed as the standard cubing of $\R^k$ in which each $k$-cube has
  been given the simplicial subdivision described in
  Example~\ref{exmp:cubes}.
\end{exmp}
Alternatively, the orthoscheme tiling of $\R^k$ can be viewed as the
cell structure of a simplicial hyperplane arrangement.

\begin{defn}[Types of hyperplanes]\label{def:hyperplanes}
  Consider the hyperplane arrangement consisting of two types of
  hyperplanes.  The \emph{first type} are those defined by the
  equations $x_i = \ell$ for all $i \in [k]$ and all $\ell \in \Z$.
  The \emph{second type} are those defined by the equations $x_i - x_j
  = \ell$ for all $i\neq j \in [k]$ and all $\ell \in \Z$. When both
  types of hyperplanes are used, the resulting hyperplane arrangement
  partitions $\R^k$ into its standard orthoscheme tiling.
\end{defn}
The hyperplanes of the first type define the standard cubing of $\R^k$
and the hyperplanes of the second type are closely related to the
Coxeter complex of the affine symmetric group.

\begin{defn}[Affine symmetric group]\label{def:aff-sym-gp}
  The Euclidean Coxeter group of type $\widetilde A_{k-1}$ is also
  called the \emph{affine symmetric group} $\affsym_k$.  It is
  generated by orthogonal reflection in the hyperplanes of the
  second kind.
\end{defn}
\begin{rem}
Note that the spherical Coxeter group of type $A_{k-1}$, the
symmetric group, is generated by reflections in hyperplanes of
the second type for which $\ell = 0$. Since the \emph{roots}
$e_i - e_j$ are perpendicular to the vector $\one$, both the
symmetric group and the affine symmetric group act on the
$(k-1)$-dimensional space $\one^\perp$.
\end{rem}
\begin{defn}[Coxeter shapes and columns]\label{def:cox-simplex}
  The hyperplane arrangement that consists solely of the hyperplanes
  of the second type restricted
  to any hyperplane $H$ defined by the equation $\langle x, \one
  \rangle = r$ for some $r \in \R$ partitions $H \cong \R^{k-1}$ into
  a reflection tiling by Euclidean simplices whose shape is encoded in
  the extended Dynkin diagram of the type $\widetilde A_{k-1}$. We
  call the isometry type of this Euclidean simplex the \emph{Coxeter
    shape} or \emph{Coxeter simplex} of type $\widetilde A_{k-1}$ and
  when the subscript is clear from context it is often omitted to
  improve clarity.  When this hyperplane arrangement is not restricted
  to a hyperplane orthogonal to the vector $\one$, the closure of a
  connected component of the complementary region is an unbounded
  infinite column that is a metric product $\sigma \times \R$ where
  $\sigma$ is a Coxeter simplex of type $\widetilde A$ and $\R$ is the
  real line.  We call these the \emph{columns of $\R^k$}.
\end{defn}
One consequence of this column structure is that the standard
orthoscheme tiling of $\R^k$ partitions the columns of $\R^k$ into a
sequence of orthoschemes.  We begin with an explicit example.

\begin{figure}
  \begin{tikzpicture}[scale=.53]
    \begin{scope}[x={(3.2cm,0)},y={(2cm,1.2cm)},z={(0cm,3.8cm)}]
      \begin{scope}[GreenPoly,thin]
        \filldraw (0,0,0)--(3,3,3)--(4,4,3)--(1,1,0)--(1,0,0)--cycle;
        \foreach \n in {0,1,2} {\draw[dashed] (\n,\n,\n)-- ++(1,1,0);}
        \foreach \n in {1,2,3} {\draw (\n,\n,\n)-- ++(0,-1,-1);}
        \foreach \n in {1,2,3} {\draw (\n,\n,\n-1)-- ++(1,0,1);}
        \draw (1,0,0)--(4,3,3);
      \end{scope}
      \begin{scope}[very thick,color=black]
        \foreach \n in {1,2,3} {\draw[dashed] (\n,\n,\n-1)-- ++(0,0,1);}
        \foreach \n in {0,1,2,3} {\draw (\n,\n,\n)-- ++(1,0,0);}
        \foreach \n in {0,1,2,3} {\draw (\n+1,\n,\n)-- ++(0,1,0);}
      \end{scope}
      \foreach \n in {1,2,3} {\draw (\n,\n,\n) node[anchor=south
          east]{$(\n,\n,\n)$};}
      \draw (0,0,0) node[anchor=north]{$(0,0,0)$};
      \draw (1,1,0) node[anchor=north west]{$(1,1,0)$};
      \draw (2,2,1) node[anchor=north west]{$(2,2,1)$};
      \draw (2,1,1) ++(.05,0,-.05) node[anchor=north west]{$(2,1,1)$};
      \draw (3,2,2) ++(.05,0,-.05) node[anchor=north west]{$(3,2,2)$};
      \draw (4,3,3) ++(.05,0,-.05) node[anchor=north west]{$(4,3,3)$};
      \draw (3,3,2) node[anchor=north west]{$(3,3,2)$};
      \draw (4,4,3) node[anchor=south]{$(4,4,3)$};
      \draw (1,0,0) node[anchor=north] {$(1,0,0)$};
      \foreach \n in {0,1,2,3} {\filldraw (\n,\n,\n) circle (1mm);}
      \foreach \n in {0,1,2,3} {\filldraw (\n+1,\n,\n) circle (1mm);}
      \foreach \n in {0,1,2,3} {\filldraw (\n+1,\n+1,\n) circle (1mm);}
    \end{scope}
  \end{tikzpicture}
  \caption{A portion of the column in $\R^3$ that contains the
    orthoscheme shown in Figure~\ref{fig:orthoscheme}. The edges of
    the spiral are thicker and darker than the
    others - see Example~\ref{ex:column}.\label{fig:column}}
\end{figure}
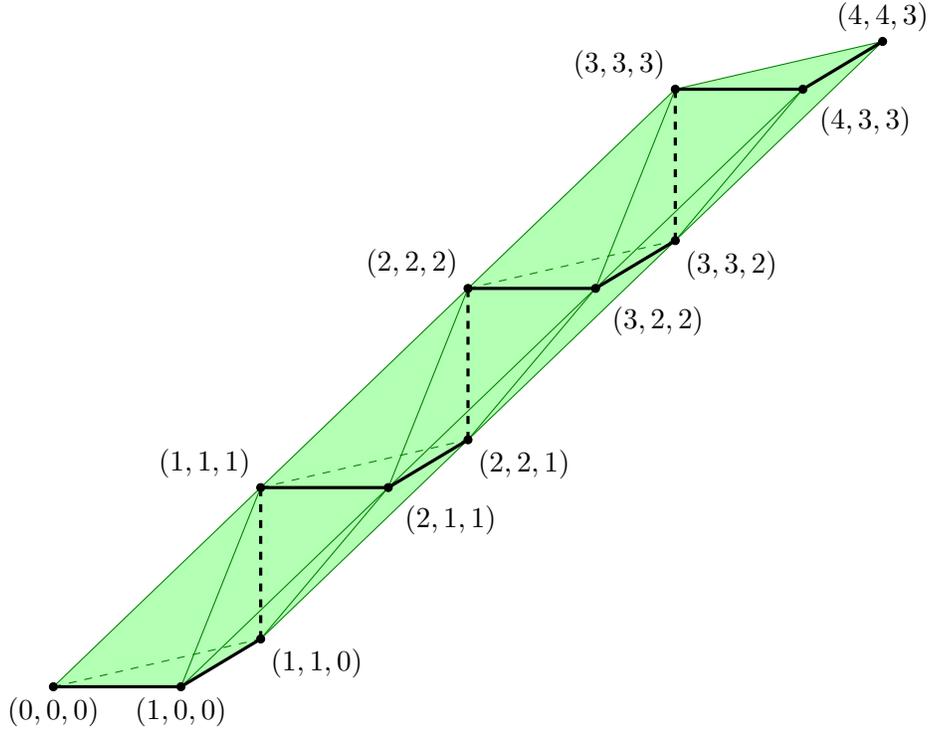
\begin{exmp}[Column in $\R^3$]\label{ex:column}
  Let $\mathcal{C}$ be the unique column of $\R^3$ that contains the
  $3$-simplex shown in Figure~\ref{fig:orthoscheme}.  The column
  $\mathcal{C}$ is defined by the inequalities $x_1 \geq x_2 \geq x_3
  \geq x_1-1$ and its sides are the hyperplanes defined by the
  equations $x_1 - x_2 = 0$, $x_2 - x_3 = 0$ and $x_1-x_3 = 1$.  The
  vertices of $\Z^3$ contained in this column form a sequence
  $\{v_\ell\}_{\ell\in \Z}$ where the order of the sequence is
  determined by the inner product of these points with the special
  vector $\one = (1^3) = (1,1,1)$.  Concretely, the vertex $v_\ell$ is
  the unique point in $\Z^3 \cap \mathcal{C}$ such that $\langle
  v_\ell, \one \rangle = \ell \in \Z$.  The vectors in this case are
  $v_{-1} = (0,0,-1)$, $v_0 = (0,0,0)$, $v_1 = (1,0,0)$, $v_2 =
  (1,1,0)$, $v_3 = (1,1,1)$, $v_4 = (2,1,1)$ and so on.  Successive
  points in this list are connected by unit length edges in coordinate
  directions and this turns the full list into a spiral of edges.
  Traveling up the spiral, the edges cycle through the possible
  directions in a predictable order.  In this case they travel one unit
  step in the positive $x$-direction, $y$-direction, $z$-direction,
  $x$-direction, $y$-direction, $z$-direction and so on.  Any $3$
  consecutive edges in the spiral have a standard $3$-orthoscheme as
  its convex hull and the union of these individual orthoschemes is
  the convex hull of the full spiral, which is also the full column
  $\mathcal{C}$.  See Figure~\ref{fig:column}.  Metrically
  $\mathcal{C}$ is $\sigma \times \R$ where $\sigma$ is an equilateral
  triangle, i.e.\ the Coxeter simplex of type $\widetilde A_2$.
\end{exmp}
Columns in $\R^k$ have many of the same properties.

\begin{defn}[Columns in $\R^k$]\label{def:columns}
  A column $\mathcal{C}$ of $\R^k$ can be defined by
  inequalities of the form 
  \begin{equation}\label{eq:column}
    x_{\pi_1}+a_{\pi_1} \geq x_{\pi_2}+a_{\pi_2} \geq \cdots \geq
    x_{\pi_k} +a_{\pi_k} \geq x_{\pi_1} + a_{\pi_1}-1
  \end{equation} 
  where $(\pi_1,\pi_2,\ldots,\pi_k)$ is a permutation of integers
  $(1,2,\ldots,k)$ and $a = (a_1,a_2,\ldots,a_k)$ is a point in
  $\Z^k$.  The vertices of $\Z^k$ contained in $\mathcal{C}$ form a
  sequence $\{v_\ell\}_{\ell\in \Z}$ where the order of the sequence is
  determined by the inner product of these points with the
  vector $\one = (1,1,\ldots,1)$.  Concretely the vertex $v_\ell$
  is the unique point in $\Z^k \cap \mathcal{C}$ such that $\langle
  v_\ell, \one \rangle = \ell \in \Z$.  Successive points in this list are
  connected by unit length edges in coordinate directions and this
  turns the full list into a spiral of edges.  Traveling up the
  spiral, the edges cycle through the possible directions in a
  predictable order based on the list $(\pi_1,\pi_2,\ldots,\pi_k)$.
  Any $k$ consecutive edges in the spiral have a standard
  $k$-orthoscheme as its convex hull and the union of these individual
  orthoschemes is the convex hull of the full spiral, which is also
  the full column $\mathcal{C}$.  Metrically, $\mathcal{C}$ is
  $\sigma \times \R$ where $\sigma$ is a Coxeter simplex of type
  $\widetilde A_{k-1}$.  Since the full column is a convex subset of
  $\R^k$, it is a $\cat(0)$ space.
\end{defn}
\begin{defn}[Dilated columns]\label{def:dilated-cols}
  If the $-1$ in the final inequality of Equation~\ref{eq:column}
  defining a column in $\R^k$ is replaced by a $-\ell$ for some positive
  integer $\ell$, then the shape described is a \emph{dilated column},
  i.e.\ a dilated version of a single column.  As a metric space, a
  dilated column is a metric direct product of the real line and a
  Coxeter shape of type $\widetilde A$ dilated by a factor of $\ell$
  and is also a $\cat(0)$ space.
  As a cell complex, a dilated column is the union of $\ell^{k-1}$
  ordinary columns of $\R^k$ tiled by orthoschemes.
\end{defn}
Some of these dilated columns are of particular interest.

\begin{defn}[$(k,n)$-dilated columns]\label{def:nk-dil-cols}
  Let $n > k > 0$ be positive integers and let $\mathcal{C}$ be the
  full subcomplex of the orthoscheme tiling of $\R^k$ restricted to
  the vertices of $\Z^k$ that satisfy the strict inequalities
  \[x_1 < x_2 < \cdots < x_k < x_1 + n\text{.}\]
  We call $\mathcal{C}$ the
  \emph{$(k,n)$-dilated column in $\R^k$}.  A point $x \in \Z^k$ is in
  $\mathcal{C}$ if and only if its coordinates are strictly increasing
  in value from left to right and the gap between the first and the
  last coordinate is strictly less than $n$.  To see that the subspace
  $\mathcal{C}$ really is a dilated column of $\R^k$, note that it is
  defined by the weak inequalities 
  \[x_1-1 \leq x_2-2 \leq \cdots \leq x_k-k \leq x_1-(k+1)+n.\] 
  There is a natural bijection between the
  sets of integer vectors satisfying these two sets of inequalities
  that uses the usual combinatorial trick for converting between
  statements about strictly increasing integer sequences and
  statements about weakly increasing ones.  From the weak inequalities
  we see that the $(k,n)$-dilated column $\mathcal{C}$ is a $(n-k)$
  dilation of an ordinary column and thus a union of $(n-k)^{k-1}$
  ordinary columns.
\end{defn}
\begin{exmp}[$(2,6)$-dilated column]\label{ex:26-dil-col}
  When $k=2$ and $n=6$, the defining inequalities are $x < y < x+6$
  and a portion of the $(2,6)$-dilated column $\mathcal{C}$ is shown
  in Figure~\ref{fig:2-pt-hex}. The meaning of the vertex labels used
  in the figure are explained in Example~\ref{ex:2-pt-hex-lab}. Note
  that $\mathcal{C}$ is metrically an ordinary column dilated by a
  factor of $4$, its cell structure is a union of $(6-2)^{2-1} = 4$
  ordinary columns, and it is defined by the weak inequalities $x+1
  \leq y$ and $y \leq x+5$ or, equivalently, $x-1 \leq y-2 \leq
  x-3+6$.
\end{exmp}
%

\part{Boundary Braids}\label{part:theorems}

We now come to our main topic of study: boundary braids. 
This part begins by introducing orthoscheme configuration spaces
and describing the specific case of an oriented $n$-cycle.
We then prove the fact that if several strands are individually boundary parallel
then they are simultaneously boundary parallel.
Finally, we study dual simple boundary braids in detail and use our findings
there to prove the main theorems.

\section{Configuration Spaces}\label{sec:conf-sp}

In this section we introduce a new combinatorial model for the 
configuration space of $k$ points in a directed graph and, more generally,
$k$ points in an orthoscheme complex. In contrast to the 
configuration spaces for graphs used by Abrams and Ghrist, 
which are cubical \cite{abrams00, ghrist01}, our models are simplicial.

\begin{defn}[Products of graphs]\label{def:orth-prod-sp}
  Let $\Gamma$ be a metric simplicial graph with oriented edges of
  unit length. Note that $\Gamma$ can be regarded either as an ordered
  simplicial complex or as a cubical complex and we can form direct
  products of several copies of $\Gamma$ in either context. The
  resulting spaces will be naturally isometric but their cell
  structures differ. We denote by $\Prod_k(\Gamma,\Orth)$ respectively
  $\Prod_k(\Gamma,\Cube)$
  the orthoscheme product respectively cubical product of $k$ copies
  of $\Gamma$.
\end{defn}
\begin{exmp}[Orthoscheme product spaces]\label{ex:orth-prod-sp}
  If $\Gamma$ is an oriented edge of unit length then $\Prod_k(\Gamma,\Cube)$
  is a unit $k$-cube while $\Prod_k(\Gamma,\Orth)$ is the simplicial
  subdivision of the $k$-cube described in Example~\ref{exmp:cubes}.
  If $\Gamma$ is $\mathbb{R}$ subdivided in edges of unit length then 
  $\Prod_k(\Gamma,\Cube)$ is the standard cubing of $\R^k$ while 
  $\Prod_k(\Gamma,\Orth)$ is the
  standard orthoscheme tiling of $\R^k$ described in Example~\ref{exmp:ortho-Rn}.
\end{exmp}
Recall from Definition~\ref{def:conf-sp} that the (topological) configuration space of $k$ points in $\Gamma$ is $\Gamma^k - \diag_k(\Gamma)$ where $\diag_k(\Gamma)$ is the thick diagonal. To obtain a combinatorial configuration space, we take the full subcomplex supported on this subset with respect to either of the above cell structures. For the cubical structure this was first done by Abrams \cite{abrams00}. With the simplicial cell structure in place, our definition is completely analogous.

\begin{defn}[Orthoscheme configuration spaces]\label{def:orth-conf-sp}
  Let $\Gamma$ be a metric simplicial graph with oriented edges of
  unit length. The
  \emph{orthoscheme configuration space of $k$ labeled points in an
    oriented graph $\Gamma$} is the full subcomplex $\conf_k(\Gamma,\Orth)$ of $\Prod_k(\Gamma,\Orth)$ supported on $\Prod_k(\Gamma,\Orth) -
  \diag_k(\Gamma)$. Thus, a closed orthoscheme of $\Prod_k(\Gamma,\Orth)$ lies in $\conf_k(\Gamma,\Orth)$ if and only if it is disjoint form $\diag_k(\Gamma)$.
    The \emph{orthoscheme configuration space of $k$ unlabeled points} is $\uconf_k(\Gamma,\Orth) = \conf_k(\Gamma,\Orth)/\sym_k$.
\end{defn}
  
\begin{rem}[Open questions]\label{rem:conf-sp-quest}
  Since the simplicial structure for the product $\Prod_k(\Gamma,\Orth)$ is 
  a refinement of the cubical structure of $\Prod_k(\Gamma,\Cube)$, the orthoscheme configuration space $\conf_k(\Gamma,\Orth)$ lies between the topological configuration space $\conf_k(\Gamma)$ and the cubical one $\conf_k(\Gamma,\Cube)$. It is therefore interesting to compare it to either, specifically, to determine under which conditions two of them are homotopy equivalent. In Example~\ref{ex:conf-spaces} we will see that the orthoscheme configuration and the cubical configuration space are generally not homotopy equivalent.
  Cubical configuration spaces are known to be non-positively curved \cite{abrams00}. We do not know whether the same is true of orthoscheme configuration spaces. However, in the next section we will see that they are in the most basic case where $\Gamma$ is a single oriented cycle.
\end{rem}
%
\section{Points on a Cycle}\label{sec:points_on_cycle}

For the purposes of this article, we are primarily interested in the
orthoscheme configuration spaces of a
single oriented cycle. In this section, we treat this case in detail.

\begin{defn}[Oriented cycles]\label{def:oriented-cycles}
  An \emph{oriented $n$-cycle} is a directed graph $\Gamma_n$ with
  vertices indexed by the elements of $\Z/n\Z$ and a directed unit-length edge
  from $i$ to $i+1$ for each $i \in \Z/n\Z$.  In illustrations we draw
  an oriented $n$-cycle so that it is the boundary cycle of a regular
  $n$-gon in the plane with its edges oriented counter-clockwise. The graph
  $\Gamma_n$ can be viewed as $\R/n\Z$. Similarly, the
  orthoscheme product space $\Prod_k(\Gamma_n,\Orth)$ is an $k$-torus
  $\R^k/(n\Z)^k$ where $\R^k$ carries the orthoscheme structure described in Example~\ref{exmp:ortho-Rn}.
\end{defn} 
For the rest of the section $\Gamma_n$ will denote an oriented $n$-cycle.

\begin{exmp}[Cubical vs.\ orthoscheme]\label{ex:conf-spaces}
  In both the cubical and orthoscheme cell structure of $(\Gamma_n)^n$ the only
  vertices not on the thick diagonal $\diag_n(\Gamma_n)$ are the $n$-tuples where each entry is a distinct
  vertex of $\Gamma_n$. These form a single $\sym_n$-orbit.
  In the cubical structure no edge avoids $\diag_n(\Gamma_n)$, so  
  $\conf_n(\Gamma_n,\Cube)$ is a discrete space consisting of
  $n!$ points and $\uconf_n(\Gamma_n,\Cube)$ is a single point.

  In the orthoscheme structure, there are edges that are disjoint from $\diag_n(\Gamma_n)$.
  These correspond to the motion where all $n$ points rotate
  around the $n$-cycle simultaneously in the same oriented direction
  and they are longest edges in the top-dimensional orthoschemes.  No
  other simplices avoid $\diag_n(\Gamma_n)$.  Thus
  $\conf_n(\Gamma_n,\Orth)$ has $(n-1)!$ connected
  components each of which is an oriented $n$-cycle. The unordered configuration space $\uconf_n(\Gamma_n,\Orth)$ is a circle consisting of a single vertex and a single edge.
  
  This illustrates that the cubical and the orthoscheme configuration spaces are generally not homotopy equivalent. Note that in this example, the topological configuration spaces are homotopy equivalent to the orthoscheme versions. However, reversing the orientation of a single edge makes the orthoscheme configuration spaces equal to the cubical ones and therefore not homotopy equivalent to the topological ones.
\end{exmp}
The purpose of the present section is the following result.

\begin{prop}[Points on a cycle and curvature]\label{prop:pt-curv}
  Each component of the universal cover of $\conf_k(\Gamma_n)$ is isomorphic, as an orthoscheme complex, to the $(k,n)$-dilated column and therefore $\cat(0)$.
  In particular, $\conf_k(\Gamma_n)$ and $\uconf_k(\Gamma_n)$ are non-positively curved.
\end{prop}
\begin{proof}
First recall that $\Prod_k(\widetilde{\Gamma}_n,\Orth) \cong \widetilde{\Prod}_k(\Gamma_n,\Orth)$ is $\R^k$ with the structure described in Example~\ref{exmp:ortho-Rn} (where the tilde denotes the universal cover on both sides). Let $C$ be the subcomplex obtained by removing the hyperplanes of the form $x_i - x_j = \ell$ with $i\neq j \in [k]$ and $\ell \in n\Z$ and taking the full subcomplex. Since these hyperplanes descend to the thick diagonal, we see that $\conf_k(\Gamma_n) \cong C/(n\Z)^k$.

Notice that these hyperplanes include the ones used to define the $(k,n)$-dilated column in $\R^k$ (Definition~\ref{def:nk-dil-cols}).  Thus one connected component of $C$ is a $(k,n)$-dilated column in $\R^k$. Since $\sym_k$ permutes the connected components, each component is a dilated column. Thus each component of $C$ is $\cat(0)$ and, in particular, is simply connected.

Since both $(n\Z)^k$ and $\sym_k$ act freely on $C$, both 
$C \to \conf_k(\Gamma_n)$ and $C \to \uconf_k(\Gamma_n)$ are covering maps and 
thus the configuration spaces $\conf_k(\Gamma_n)$ and $\uconf_k(\Gamma_n)$ 
are both non-positively curved.
\end{proof}
We now give two examples which illustrate Proposition~\ref{prop:pt-curv}.

\begin{figure}
  	\begin{tikzpicture}[scale=1.1]
	  \newcommand{\hexcondark}[2]{
	    \begin{scope}[shift={(#1,#2)}]
	      \node () [disk,fill=yellow!60,minimum size=0.65cm] {};
	      \foreach \n in {1,...,6} {
	        \coordinate (\n) at  (\n*60:0.2cm);
	      }
	      \draw[fill=blue!30!white] (1)--(2)--(3)--(4)--(5)--(6)--cycle;
	      \draw (#1*60:0.2cm) node [dot] {};
	      \draw (#2*60:0.2cm) node [opendot] {};
	    \end{scope}
	  }
	  \newcommand{\hexconlight}[2]{
	    \begin{scope}[shift={(#1,#2)}]
	      \node () [disk,color=gray,fill=white,minimum size=0.65cm] {};
	      \foreach \n in {1,...,6} {
	        \coordinate (\n) at  (\n*60:0.2cm);
	      }
	      \draw[color=gray,fill=blue!10!white] (1)--(2)--(3)--(4)--(5)--(6)--cycle;
	      \draw (#1*60:0.2cm) node [dot,fill=gray] {};
	      \draw (#2*60:0.2cm) node [opendot] {};
	    \end{scope}
	  }
	  \draw[thick,->] (-5,0)--(5.3,0);
	  \draw[thick,->] (0,-.3)--(0,6.3);
	  \begin{scope}[GreenPoly,very thick]
	    \filldraw[color=green!60!white] (0,1)--(0,5)--(1,6)--(5,6)--cycle;
	    \filldraw[thick] (0,1)--(0,5)--(-5,0)--(-1,0)--cycle;
	    \draw[color=white] (1,6)--(5,6);
	    \foreach \x in {1,...,4} {\draw[dashed,thick] (\x,6)--(\x+1,6);}
	    \foreach \x in {1,...,5} {\draw (\x,6)--(0,6-\x);}
	    \foreach \x in {1,...,5} {\draw[thick] (-\x,0)--(0,\x);}
	    \foreach \y in {1,...,4} {\draw[thick] (\y-5,\y)--(0,\y);}
	    \foreach \y in {2,...,5} {\draw[thick] (0,\y)--(\y-1,\y);}
	    \foreach \z in {1,...,4} {\draw (\z,\z+1)--(\z,6);}
	    \foreach \z in {-4,...,-1} {\draw[thick] (\z,0)--(\z,\z+5);}
	    \draw (0,1)--(0,5);
	  \end{scope}
	  \draw[dashed] (-.3,-.3)--(5.3,5.3);
	  \draw[dashed] (-5.3,.7)--(.3,6.3);
	  \foreach \x in {-5,...,-1} {\hexconlight{\x}{0}}
	  \foreach \x in {-4,...,-1} {\hexconlight{\x}{1}}
	  \foreach \x in {-3,...,-1} {\hexconlight{\x}{2}}
	  \foreach \x in {-2,...,-1} {\hexconlight{\x}{3}}
	  \foreach \x in {-1,...,-1} {\hexconlight{\x}{4}}
	  \foreach \x in {0,...,0} {\hexcondark{\x}{1}}
	  \foreach \x in {0,...,1} {\hexcondark{\x}{2}}
	  \foreach \x in {0,...,2} {\hexcondark{\x}{3}}
	  \foreach \x in {0,...,3} {\hexcondark{\x}{4}}
	  \foreach \x in {0,...,4} {\hexcondark{\x}{5}}
	  \foreach \x in {1,...,5} {\hexconlight{\x}{6}}
	\end{tikzpicture}
  \caption{A portion of the $(2,6)$-dilated column, i.e.\ the full
    subcomplex of the orthoscheme complex of $\R^2$ on the vertices
    satisfying the strict inequalities $x < y < x+6$. The vertex
    labels and the shaded regions are used to construct simplicial
    configuration spaces for $2$ labeled points in a $6$-cycle and for
    $2$ unlabeled points in a $6$-cycle.\label{fig:2-pt-hex}}
\end{figure}
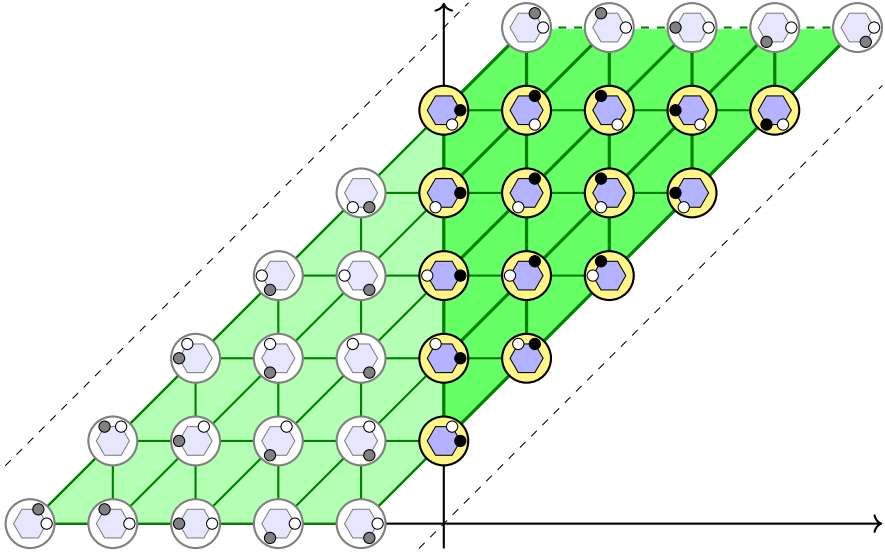
\begin{exmp}[$2$ labeled points in a $6$-cycle]\label{ex:2-pt-hex-lab}
  Figure~\ref{fig:2-pt-hex} shows a portion of the infinite strip that
  is the $(2,6)$-dilated column in $\R^2$.  When this strip is
  quotiented by the portion of the $(6\Z)^2$-action on $\R^2$ that
  stabilizes this strip, its vertices can be labeled by two labeled
  points in a hexagon.  The black dot indicates the value of its
  $x$-coordinate mod $6$ and the white dot indicates the value of its
  $y$-coordinate mod $6$.  The rightmost vertex of the hexagon corresponds
  to $0 \mod 6$ and the residue classes proceed in a counterclockwise
  fashion.  The five hexagons on the $y$-axis, for example, have
  $x$-coordinate equal to $0\mod 6$ and $y$-coordinate ranging from
  $1$ to $5\mod 6$.  One component of the labeled orthoscheme
  configuration space $\conf_2(X,\Orth)$ is an annulus formed by
  identifying the top and bottom edges of the region shown according
  to their vertex labels.  Actually, in this case there is only
  $(2-1)! = 1$ component, so the annulus is the full labeled
  orthoscheme configuration space.
\end{exmp}
\begin{exmp}[$2$ unlabeled points in a $6$-cycle]\label{ex:2-pt-hex-unlab}
  The unlabeled orthoscheme configuration space is formed by further
  quotienting the labeled orthoscheme configuration space to remove
  the distinction between black and white dots.  In particular the $5$
  vertices shown on the horizontal line $y=6$ are identified with the
  $5$ vertices on the vertical line $x=0$.  This identification can be
  realized by the glide reflection sending $(x,y)$ to $(y,x+2)$, a map
  which also generates the unlabeled stabilizer of the $(2,6)$-dilated
  column.  The heavily shaded region is a fundamental domain for this
  $\Z$-action and the unlabeled orthoscheme configuration space is the
  formed by identifying its horizontal and vertical edges with a
  half-twist forming a M\"obius strip.  The heavily shaded labels are
  the preferred representatives of the vertices in the quotient.
\end{exmp}
%
\section{Boundary Braids}\label{sec:boundary_braids}

We now come to our main object of study, boundary braids. The goal of this section is a key technical result saying that if certain strands of a braid can individually be realized as boundary-parallel strands then they can be realized as boundary parallel strands simultaneously.

\begin{lem}[Boundary parallel rotation Braids]\label{lem:rotation}
	Let $A\subseteq [n]$ not be a singleton and define
	\[
	B = \{b\mid b\not\in A \text{ or } b+1 \in A\}
	\]
	Then $\delta_A$ is a $(B,\cdot)$-boundary braid but not a
	$(b,\cdot)$-boundary braid for any $b \in [n]-B$.
\end{lem}

\begin{proof}
  For $A = \emptyset$ and $A = [n]$ the statement is clear so we assume $2 \le \card{A} < n$ from now on.

  The first statement is straightforward by considering the standard representative
  of $\delta_A$, given by constant-speed parametrization of each strand along the boundary
  of the subdisk $D_A$. 
  
  For the second statement, let $b\in [n]-B$. Then $b\in A$ but $b+1\not\in A$. Fix the standard representative $f$ of $\delta_A$.
  Let $c = (b-1)\cdot \perm(\delta_A)$, meaning that the strand $f^{b-1}$ ends in $c$. Thus $c \in \{b-1,b\}$.
  We consider a disc $U \subseteq I\times D_n$ that is bounded by the following four paths in $\partial (I \times D_n)$: the strand $f^{b-1}_c$, the
  strand $f^{b+1}_{b+1}$, the straight line in $\{0\}\times D_n$ connecting $(0,p_{b-1})$ and $(0,p_{b+1})$, and the straight line in $\{1\}\times D_n$ connecting $(1,p_c)$ and $(1,p_{b+1})$.
  Now note that since $b \in A$, the strand $f^b$ does not end in $p_b$ and therefore not in the set $\{p_c,\ldots,p_{b+1}\}$ which is either $\{p_b,p_{b+1}\}$ or $\{p_{b-1},p_b,p_{b+1}\}$, depending on whether or not $b-1 \in A$. As a consequence the strand $f^b$ starts on one side of the disk and 
  ends on the other side, and thus it transversely intersects the disk
  an odd number of times.  
  Since the parity of the number of transverse intersections is
  preserved under homotopy of strands and strands which remain
  in the boundary have no such intersections, we may conclude 
  that the $(b,\cdot)$-strand is not boundary parallel in any representative
  for $\delta_A$.
\end{proof}
\begin{defn}[Wrapping number]\label{def:wrapping}
	Let $\beta$ be a $(b,\cdot)$-boundary braid and let $f$ be a
	representative for which the image of $f^b$ lies in the boundary of $D_n$. If we view
	the boundary of $D_n$ as an $n$-fold cover $\pi \colon \partial D_n \to \mathbb{S}^1$
	of the standard cell structure for $\mathbb{S}^1$ with one vertex and one edge, then boundary paths
	in $\partial D_n$ that start and end at vertices of $D_n$ may be
	considered as lifts of loops in $\mathbb{S}^1$. More concretely, let
	$\varphi \colon \mathbb{R} \to \partial D_n$ be a covering map such 
	that $\varphi(i) = p_i$ for each $i \in \mathbb{Z}$. Let 
	$\tilde{f}^b$ be any lift of $f^b$ via this covering and define the
	\emph{wrapping number} of the $(b,c)$-strand of $f$ to be
	$w_f(b,c) = \tilde{f}^b(1) - \tilde{f}^b(0)$.
	
	A slightly different description is as follows. Consider the diagram
	\begin{center}
	\begin{tikzpicture}[xscale=2,yscale=-2]
	\node (i) at (0,0) {$[0,1]$};
	\node (bdry) at (1,0) {$\partial D_n$};
	\node (quot_sphere) at (0,1) {$\mathbb{S}^1$};
	\node (cov_sphere) at (1,1) {$\mathbb{S}^1$};
	\path (i) edge[->] node[anchor=south] {$f^b$} (bdry);
	\path (i) edge[->] node[anchor=east] {$/\scriptstyle 0 \sim 1$} (quot_sphere);
	\path (bdry) edge[->] node[anchor=west] {$\pi$} (cov_sphere);
	\path (quot_sphere) edge[->] node[anchor=south] {$f^b_*$} (cov_sphere);
	\end{tikzpicture}
	\end{center}
	where the left map is the quotient map that identifies $0$ and $1$. The map $f^b_*$ is defined by commutativity of the diagram. Then the wrapping number of $f^b$ is the winding number of $f^b_*$.
\end{defn}
Notice that the wrapping number is $n$ times what one would 
reasonably define as the winding number.

\begin{lem}\label{lem:wrapping_well-defined}
  The wrapping number is well-defined.
\end{lem}
\begin{proof}
  Let $f$ be a representative of the $(b,\cdot)$-boundary braid
  $\beta$ for which $f^b$ lies in the boundary; we temporarily
  denote the wrapping number by $w_f(b,\cdot)$ to indicate the 
  presumed dependence on our choice of representative. If $f$ and
  $f^\prime$ both represent $\beta$, then $f^\prime \cdot f^{-1}$ is a
  representative for the trivial braid with $w_{f^\prime \cdot f^{-1}}(b,b) =
  w_{f^\prime}(b,\cdot) - w_f(b,\cdot)$. It therefore suffices to show that each
  strand in every representative of the trivial braid has wrapping
  number zero.

  Now, let $f$ be a representative of the trivial braid $1$ for which
  $f^b$ lies in the boundary, and suppose that the wrapping number
  $w_f(b,b) \neq 0$. If $f^{b^\prime}$ is another strand in $f$,
  then we may obtain a map to the pure braid group $\purebraid_2$
  by forgetting all strands except $f^b$ and $f^{b^\prime}$.
  The image $\beta^\prime$ of the trivial braid under this map 
  may be written as an even power of $\delta_2$ since 
  ${\delta_2}^2$ generates $\purebraid_2$, and the wrapping numbers
  can then be related as 
  $w_1(b,b) = \frac{n}{2}w_{\beta^\prime}(b,b)$.
  However, it is clear from the procedure of forgetting strands
  that the resulting braid in $\purebraid_2$ is trivial,
  and since every braid in $\purebraid_2$ has both strands boundary parallel,
  we know that $w_{\beta^\prime}(b,b)$ is zero, and thus so is $w_1(b,b)$.
  Therefore, every representative for the trivial braid has trivial wrapping
  numbers, and we are done.
\end{proof}
\begin{lem}\label{lem:wrapping_add}
  If $\beta$ and $\gamma$ are braids in $\braid_n$ such that $\beta$
  is a $(b,b^\prime)$-boundary braid and $\gamma$ is a
  $(b^\prime,\cdot)$-boundary braid, then $\gamma\beta$ is a
  $(b,\cdot)$-boundary braid with wrapping number $w_{\beta\gamma}(b)
  = w_\beta(b) + w_\gamma(b^\prime)$.
\end{lem}
\begin{proof}
If $f$ and $g$ are representatives of $\beta$ and $\gamma$ respectively such that $f^b_{b'}$ and $g^{b'}$ are boundary strands then $fg$ is a representative of $\beta\gamma$ such that $(fg)^b$ is a boundary strand. For this representative it is clear that the wrapping numbers add up in the described way.
\end{proof}
\begin{lem}\label{lem:wrapping_0}
  Let $B\subseteq [n]$ and $\beta \in \braid_n$. Then $\beta \in
  \fix_n(B)$ if and only if $w_\beta(b) = 0$ for all $b\in B$.
\end{lem}
\begin{proof}
	If $\beta\in\fix_n(B)$, then there is a representative $f$ of $\beta$
	in which each $(b,\cdot)$-strand is fixed and thus $w_\beta(b,\cdot) = 0$.
	
	For the other direction, we begin with the case that $B = \{b\}$.
	Let $f$ be a representative of $\beta\in \braid_n(B,\cdot)$
	with the $(b,\cdot)$ strand in the boundary of $D_n$, and suppose
	that $w_\beta(b,\cdot) = 0$. Then the strand $f^b$ begins and ends at
	the vertex $p_b$, and there is a homotopy $f(t)$ of $f$ which 
	moves every other strand off the boundary without changing $f^b$.
	That is, $f^{b^\prime}(t) \not\in \partial D_n$ whenever $b^\prime \in [n] - \{b\}$
	and $0 < t < 1$. After performing this homotopy, we note
	that $f(1)$ is a representative of $\beta$ in which the $(b,b)$-strand
	has wrapping number $0$ and there are no other braids in the boundary.
	Thus, there is a homotopy of this strand to the constant path, and therefore
	$\beta\in \fix_n(\{b\})$.
	
	More generally, if $B \subseteq [n]$, then the set of braids 
	$\beta\in \braid_n(B,\cdot)$ with $w_\beta(b,\cdot) = 0$
	for all $b\in B$ are those which lie in the intersection of the
	fixed subgroups $\fix_n(\{b\})$. By Proposition~\ref{prop:para-induct},
	this is equal to $\fix_n(B)$ and we are done.
\end{proof}

\begin{lem}\label{lem:wrapping_inequalities}
  Let $b_1,\ldots,b_k$ be integers satisfying 
  $0 < b_1< \cdots< b_k \leq n$ and suppose that $\beta \in \braid_n$ 
  is a $(b_i,\cdot)$-boundary braid for every $i$. Then
  \[
  b_1 + w_\beta(b_1,\cdot) < b_2 + w_\beta(b_2,\cdot) < \cdots < 
  b_k + w_\beta(b_k,\cdot) < b_1 + w_\beta(b_1,\cdot) + n\text{.}
  \]
\end{lem}
\begin{proof}
  Note that it suffices to prove that 
  \[b_i + w_\beta(b_i,\cdot) < b_j + w_\beta(b_j,\cdot) < 
  b_i + w_\beta(b_i,\cdot) + n\] 
  whenever $i < j$ or, in other words, that 
  \[
  w_\beta(b_j,\cdot)-w_\beta(b_i,\cdot) \in 
  \{b_i - b_j+1, \ldots, b_i - b_j + n-1\}\text{.}
  \] 
  As a first case, suppose both $w_\beta(b_i,\cdot)$ and
  $w_\beta(b_j,\cdot)$ are divisible by $n$. Then forgetting all but the
  $(b_i,b_i)$- and $(b_j,b_j)$-strands of $\beta$ yields a pure braid
  $\beta^\prime \in \purebraid_2$ which can be expressed as
  $\beta^\prime = {\delta_{2}}^{2\ell}$ for some $\ell\in \mathbb{Z}$
  since $\purebraid_2 = \langle {\delta_{2}}^2 \rangle$. Then
  \[w_\beta(b_h,b_h) = \frac{n}{2} w_{\beta^\prime}(b_h,b_h) = n\ell\] 
  for each $h\in \{i,j\}$ and since every two-strand braid has simultaneously
  boundary parallel strands with equal wrapping numbers, we conclude
  that $w_\beta(b_i,b_i) = w_\beta(b_j,b_j)$. Therefore, $w_\beta(b_j,b_j) -
  w_\beta(b_i,b_i) = 0$, which satisfies the inequalities above.
	
  For the general case, define 
  \[\gamma = \beta{\delta_{n}}^{-w_\beta(b_i,\cdot)}\] 
  and observe that $w_\gamma(b_i,\cdot) = 0$. Note that $w_\gamma(b_j,\cdot)$ 
  is not congruent to
  $b_i-b_j$ mod $n$; if it were, then the $(b_i,\cdot)$- and
  $(b_j,\cdot)$-strands of $\gamma$ would terminate in the same
  vertex. Let $e$ then be the representative of $w_\gamma(b_j,\cdot)$ modulo
  $n$ that lies in the interval $\{b_i - b_j+1, \ldots, b_i - b_j + n-1\}$. Then 
    \[
    \alpha = \gamma{\delta_{[n]-\{b_i\}}}^{-e}
    \] 
  has both its $(b_i,\cdot)$- and its $(b_j,\cdot)$-strand
  boundary parallel with wrapping numbers
  \[
  w_\alpha(b_i,\cdot) = 0\quad\text{and}\quad
  w_\alpha(b_j,\cdot) \equiv 0 \mathrel{\text{mod}} n\text{.}
  \]
  It follows from the case initially considered that the congruence on the right is actually an equality.
  
  Tracing back we see that $w_\gamma(b_j,\cdot) = e$ and 
  \[
  w_\beta(b_j,\cdot) - w_\beta(b_i,\cdot) = e \in \{b_i - b_j+1, \ldots, b_i-b_j+n-1\}
  \] 
  as claimed.
\end{proof}
In what follows we will see that the inequalities given above are sharp in the sense that any tuple of numbers
satisfying the hypotheses for Lemma \ref{lem:wrapping_inequalities}
can be realized as the wrapping numbers for a braid.

\begin{lem}\label{lem:wrapping_example}
  Let $b_1,\ldots,b_k$ and $w_1,\ldots,w_k$ be integers satisfying 
  $0 < b_1 < \cdots < b_k \leq n$ and
  \[
  b_1 + w_1 < b_2 + w_2 < \ldots < b_k + w_k < b_1 + w_1 + n\text{.}
  \]
  There is a braid $\beta \in \braid_n$ such that $\beta$ is a $(B,\cdot)$-boundary braid
  for $B = \{b_1,\ldots,b_k\}$ with $w_\beta(b_i,\cdot) = w_i$ for each $i$.
\end{lem}
\begin{proof}
First let $w = \min \{w_1,\ldots,w_k\}$ and note that for any boundary braid $\beta' \in \braid_n(B,\cdot)$ the braid $\beta =\beta'\delta^w$ has wrapping numbers $w_{\beta}(b,\cdot) = w_{\beta'}(b,\cdot) + w$. It therefore suffices to show the claim in the case where some $w_i$ is $0$ and thus all $w_i$ are in $\{0,\ldots,n-1\}$. We assume this from now on.

Now the proof is by induction on $\max \{w_1,\ldots,w_k\}$, the case 
$v = 0$ being trivial. Let $B_0 = \{b_i \mid w_i = 0\}$ and 
$B_{\ge 1} = \{b_i \mid w_i \ge 1\}$. We claim that there is no 
$b \in B_0$ with $b-1 \in B_1$. 
If there were, then necessarily $b = b_i$ and $b-1 = b_{i-1}$ for some index $i$
(with the understanding that $b_0 = b_k$ and $b_{-1} = b_{k-1}$), 
but then
\[
b_{i-1} + w_{i-1} \ge b_{i-1} + 1 = b_i = b_i + w_i
\]
and this violates the assumption.

Let $\beta'$ be a braid satisfying the claim for
\[
b_i' = \left\{
\begin{array}{ll}
b_i & w_i = 0\\
b_i+1 & w_i \ge 1
\end{array}
\right.
\]
and
$w_i' = \min\{w_i-1,0\}$, where we note that
such a braid exists by the induction hypothesis. Let 
$A = B_{\ge 1} \cup \{b+1 \mid b \in B_{\ge 1}\}$.
We claim that $\beta = \delta_A\beta'$ is as needed. Indeed, $\beta$ is a $(B,\cdot)$-boundary braid by Lemma~\ref{lem:rotation}, and it has
the following wrapping numbers:
\[
w_{\beta}(b,\cdot) = \left\{
\begin{array}{ll}
w_{\beta'}(b+1,\cdot) + 1 & b \in B_{\ge 1}\\
w_{\beta'}(b,\cdot) & b \in B_{0}\text{.}
\end{array}
\right.
\]
Thus $w_{\beta}(b_i,\cdot) = w_i$ for every $i$.
\end{proof}
\begin{prop}\label{prop:simultaneously_boundary_parallel}
  Let $\beta \in \braid_n$. All boundary parallel strands of $\beta$
  are simultaneously boundary parallel. That is, if $\beta$ is a $(b,\cdot)$-boundary braid for every $b \in B$, then it is a $(B,\cdot)$-boundary braid.
\end{prop}
\begin{proof}
  Let $B\subseteq [n]$ and suppose $\beta \in \braid_n$ is a
  $(b,\cdot)$-boundary braid for each $b\in B$. 
  If we write $B = \{b_1,\ldots,b_k\}$ with $0 < b_1 < b_2 < \cdots < b_k \leq n$,
  then the wrapping numbers $w_\beta(b_i,\cdot)$ satisfy
  the inequalities given by Lemma~\ref{lem:wrapping_inequalities}. Therefore
  by Lemma~\ref{lem:wrapping_example} there is a $(B,\cdot)$-boundary braid
  $\gamma \in \braid_n$ with the same wrapping numbers as $\beta$.  By
  Lemma~\ref{lem:wrapping_add}, $\beta\gamma^{-1}$ is a
  $(b,b)$-boundary braid with $w_{\beta\gamma^{-1}}(b,b) = 0$ for
  each $b\in B$. Applying Lemma~\ref{lem:wrapping_0}, we see that
  $\beta\gamma^{-1} \in \fix_n(B)$ and, in particular, it is a $(B,\cdot)$-boundary braid.
  It follows that $\beta = (\beta\gamma^{-1})\gamma$ is a $(B,\cdot)$-boundary braid, as well.
\end{proof}
As a consequence of Proposition~\ref{prop:simultaneously_boundary_parallel} we 
obtain the following proposition which we state in analogy with Proposition
\ref{prop:para-induct}.

\begin{cor}\label{cor:intersections}
  Intersections of sets of boundary braids are sets of boundary
  braids.  Concretely, if $B\subseteq [n]$, then
  \[
  \braid_n(B,\cdot) = \bigcap_{b\in B} \braid_n(\{b\},\cdot)
  \]
  and equivalently,
  \[
  \braid_n(C\cup D,\cdot) = \braid_n(C,\cdot) \cap \braid_n(D,\cdot)
  \]
  for any $C,D \subseteq [n]$.
\end{cor}
%

\section{Dual Simple Boundary Braids}\label{sec:boundary_perms}

We start studying boundary braids in more detail by exploring the poset of boundary braids that are also dual simple.

\begin{defn}[Boundary braids]\label{def:boundary-dsb}
	Let $B\subseteq [n]$. We denote the subposet of $\DS_n$ consisting
	of $(B,\cdot)$-boundary braids by
	$\DS_n(B,\cdot)$. Notice that if $\beta\in \DS_n$, then the wrapping numbers
	satisfy $w_\beta(b,\cdot) \in \{0,1\}$ for each $b\in B$.
\end{defn}
\begin{defn}[Boundary partitions]\label{def:boundary-part}
	Let $B\subseteq [n]$. We say that a noncrossing partition $\pi\in\NC_n$ is a \emph{$(B,\cdot)$-boundary partition} if
	each $b\in B$ either shares a block with $b+1$ (modulo 
	$n$) or forms a singleton block $\{b\} \in \pi$.
	We denote by $\NC_n(B,\cdot)$ the poset of all 
	$(B,\cdot)$-boundary partitions.
\end{defn}
\begin{defn}[Boundary permutations]\label{def:boundary-perm}
	Let $B\subseteq [n]$. We say that $\pi\in\NC_n$ is a \emph{$(B,\cdot)$-boundary 
	permutation} if for all $b\in B$, $b\cdot \sigma_\pi \in \{b,b+1\}$
	(modulo $n$).
	We denote the sets of $(B,\cdot)$-boundary permutations by $\NP_n(B,\cdot)$.
\end{defn}
These definitions fit together in the expected way.

\begin{prop}\label{prop:boundary-iso}
	Let $B\subseteq [n]$. The natural identifications
	between $\NC_n$, $\DS_n$, and $\NP_n$ restrict to isomorphisms
	\[
	\DS_n(B,\cdot) \cong \NC_n(B,\cdot) \cong  \NP_n(B,\cdot)\text{.}
	\]
\end{prop}
\begin{proof}
	Fix $B\subseteq [n]$. Let $\pi \in \NC_n(B,\cdot)$ and consider 
	the corresponding dual simple braid $\delta_\pi \in \DS_n$.
	It is
	\[
	\delta_\pi = \prod_{\substack{A \in \pi\\\card{A} \ge 2}}\delta_A\text{.}
	\]
	It is clear that $\delta_\pi$ is a $(B,\cdot)$-boundary braid if and only if each
	$\delta_A$ is. By Lemma~\ref{lem:rotation} this is the case if $b \in A$ implies
	$b + 1 \in A$ for every $b \in B$ and every (non-singleton) $A$. This matches
	the definition for $\pi$ to be a $(B,\cdot)$-boundary partition.
	
	Now let $\sigma_\pi$ be the permutation corresponding to $\pi$. Note that
	$b \cdot \sigma_\pi = b$ if and only if $\{b\}$ is a block of $\pi$ and that
	$b \cdot \sigma_\pi = b+1$ if and only if $b$ and $b+1$ lie in the same block.
	From this it is clear that $\sigma_\pi$ is a $(B,\cdot)$-boundary permutation if and only
	if $\pi$ is a $(B,\cdot)$-boundary partition.
\end{proof}
\begin{figure}
	\begin{center}
	\begin{tikzpicture}
		\node[disk,minimum size=1.5cm] (top) at (0,0) {};
		\foreach \i in {1,...,3} \node[disk,minimum size=1.5cm] (a\i) at (-2+\i*2,-2.25) {};
		\foreach \i in {1,...,4} \node[disk,minimum size=1.5cm] (b\i) at (-2+\i*2,-4.5) {};
		\foreach \i in {1,...,3} \node[disk,minimum size=1.5cm] (c\i) at (\i*2,-6.75) {};
		\node[disk,minimum size=1.5cm] (bottom) at (6,-9) {};
		
		\begin{scope}[GrayPoly]
			\begin{scope}[shift={(0,0)}] 
				\makepent 
				\filldraw (1) -- (2) -- (3) -- (4) -- (5) -- cycle; 
				\drawpentb 
			\end{scope}
			
			\begin{scope}[shift={(0,-2.25)}] \makepent \filldraw 
				(1) -- (2) -- (3) -- (5) -- cycle; \drawpentb \end{scope}
			\begin{scope}[shift={(2,-2.25)}] \makepent \filldraw 
				(1) -- (3) -- (4) -- (5) -- cycle; \drawpentb \end{scope}
			\begin{scope}[shift={(4,-2.25)}] \makepent \filldraw 
				(1) -- (4) -- (5) -- cycle; \draw (2) -- (3); \drawpentb \end{scope}
			
			\begin{scope}[shift={(0,-4.5)}] \makepent \filldraw 
				(1) -- (2) -- (3) -- cycle; \drawpentb \end{scope}
			\begin{scope}[shift={(2,-4.5)}] \makepent \filldraw 
				(1) -- (3) -- (5) -- cycle; \drawpentb \end{scope}
			\begin{scope}[shift={(4,-4.5)}] \makepent \filldraw 
				(1) -- (5); \draw (2) -- (3); \drawpentb \end{scope}
			\begin{scope}[shift={(6,-4.5)}] \makepent \filldraw 
				(1) -- (4) -- (5) -- cycle; \drawpentb \end{scope}
			
			\begin{scope}[shift={(2,-6.75)}] \makepent \draw 
				(1) -- (3); \drawpentb \end{scope}
			\begin{scope}[shift={(4,-6.75)}] \makepent \draw 
				(2) -- (3); \drawpentb \end{scope}
			\begin{scope}[shift={(6,-6.75)}] \makepent \draw 
				(1) -- (5); \drawpentb \end{scope}
			
			\begin{scope}[shift={(6,-9)}] 
				\makepent
				\drawpentb 
			\end{scope}
		\end{scope}
		
		\begin{scope}[BlueLine]
			\draw (top) -- (a1) -- (b1) -- (c1) -- (b2) -- (a2) -- (top);
			\draw (a1) -- (b2);
			\draw (a3) -- (b3) -- (c2) -- (bottom) -- (c3) -- (b4) -- (a3);
			\draw (b3) -- (c3);	
		\end{scope}
		
		\begin{scope}[RedLine]
			\draw (top) -- (a3);
			\draw (a1) -- (b3);
			\draw (a2) -- (b4);
			\draw (b1) -- (c2);
			\draw (b2) -- (c3);
			\draw (c1) -- (bottom);
		\end{scope}
	\end{tikzpicture}	
	\end{center}
	\caption{The poset of boundary partitions $\NC_5(B,\cdot)$,
	where the upper-right vertex of each noncrossing partition is labeled 
	by $1$ and elements of $B=\{2,4,5\}$ are labeled by a white dot. 
	The blue and red edge colors serve to illustrate the direct product 
	structure described in Proposition~\ref{prop:nc-decomposition}.}
	\label{fig:nc-decomposition}
\end{figure}
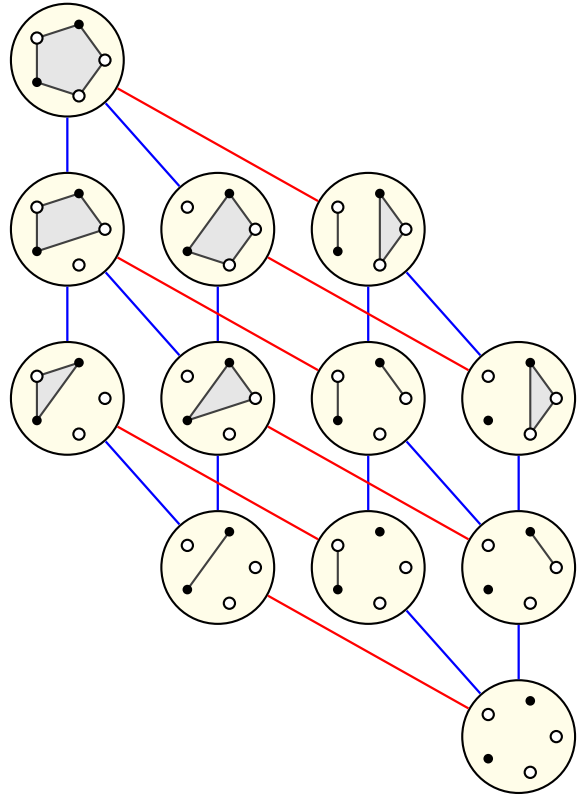
\begin{exmp}\label{ex:boundary-perm}
	Let $B = \{2,4,5\} \subseteq [5]$. Then $\NC_5(B,\cdot)$ is a subposet of 
	$\NC_5$ with $12$ elements, depicted in Figure \ref{fig:nc-decomposition}.
\end{exmp}

We now define two maps on the posets described above. The first map
takes a braid to one fixing $B$ while the second takes it to a canonical
braid with the same behavior
on the specified boundary strands.

\begin{defn}[$\fix^B$]\label{def:fix}
	For $B \subseteq [n]$ we define the map
	\begin{align*}
	\fix^B \colon \NC_n&\to\NC_n(B,\cdot)\\
	\pi &\mapsto \{A - B \mid A \in \pi\} \cup \{\{b\} \mid b \in B\}\text{.}
	\end{align*}
	Thus $\fix^B(\pi)$ is obtained from $\pi$ by making each $b \in B$
	a singleton block. We also denote by $\fix^B$ the corresponding
	maps $\NP_n \to \NP_n(B,\cdot)$ and $\DS_n \to \DS_n(B,\cdot)$.
	We call an element in the image of $\fix^B$ a \emph{$B$-fix}
	partition, braid, or permutation, and we refer to the entire image
	by the shorthand notation $\fix(\NC_n(B))$. We adopt similar notations
	for the analogous settings of $\NP_n(B,\cdot)$ and $\DS_n(B,\cdot)$.
\end{defn}

\begin{lem}\label{lem:fix_order}
Let $B \subseteq [n]$.
\begin{enumerate}
\item for $\pi \in \NC_n(B)$, $\fix^B(\pi)$ is the maximal $B$-fix element below $\pi$;\label{item:fix_char}
\item $\fix^B$ preserves order;
\item $\fix^B(\pi) \le \pi$ for all $\pi \in \NC_n$;
\item $\fix^B$ is idempotent, i.e. $(\fix^B)^2 = \fix^B$;\label{item:fix_idempotent}
\item if $\alpha \in \DS_n$ is $B$-fix and $\alpha\beta \in \DS_n$ then $\fix^B(\alpha\beta) = \alpha\fix^B(\beta)$;\label{item:fix_skips_fixed}
\end{enumerate}
\end{lem}
\begin{proof}
	The first statement is clear from the definition and the second,
	third and fourth statement follow from it.

	In the fifth statement $\alpha\fix^B(\beta) \le \alpha\beta$ is
	$B$-fix so it is $\le \fix^B(\alpha\beta)$ by \eqref{item:fix_char}.
	Conversely, $\alpha^{-1}\fix^B(\beta) \le \beta$ is $B$-fix so
	it is $\le \fix^B(\beta)$ by \eqref{item:fix_char}.
\end{proof}
\begin{lem}\label{lem:fix_image}	$\fix(\NC_n(B)) = \DS_n \cap \fix_n(B)$.
\end{lem}
\begin{proof}
	A braid $\delta_\pi \in \DS_n$ lies in $\fix(\NC_n(B))$ if and only 
	if every $b \in B$ is a singleton block of $\pi$ if and only if 
	$\delta_\pi \in \fix_n(B)$.
\end{proof}
\begin{defn}[$\move^B$]\label{def:move}
	For $B \subseteq [n]$ we define the map
	\[
	\move^B \colon \DS_n(B,\cdot) \to \DS_n(B,\cdot)
	\]
	by the equation
	$\delta_\pi = \fix^B(\delta_\pi)\move^B(\delta_\pi)\text{.}$
	We also denote by $\move^B$ the corresponding maps
	$\NC_n(B,\cdot) \to \NC_n(B,\cdot)$ and
	$\NP_n(B,\cdot) \to \NP_n(B,\cdot)$. We refer to the
	image of $\move^B$ by the shorthand $\move(\NC_n(B,\cdot))$, with
	analogous notations for $\NP_n(B,\cdot)$ and $\DS_n(B,\cdot)$.
\end{defn}
\begin{lem}\label{lem:move_well-def}
The map $\move^B$ is well-defined, i.e. $\fix^B(\delta_\pi)^{-1}\delta_\pi$ 
is a dual simple braid for each $\pi\in \NC_n(B,\cdot)$.
\end{lem}
\begin{proof}
We know from Lemma~\ref{lem:fix_order} that $\fix^B(\pi) \le \pi$. Thus by Proposition~\ref{prop:relations} there is a $\pi'$ such that $\fix^B(\delta_\pi)\delta_{\pi'} = \delta_\pi$. Then $\move^B(\delta_\pi) = \delta_{\pi'}$.
\end{proof}
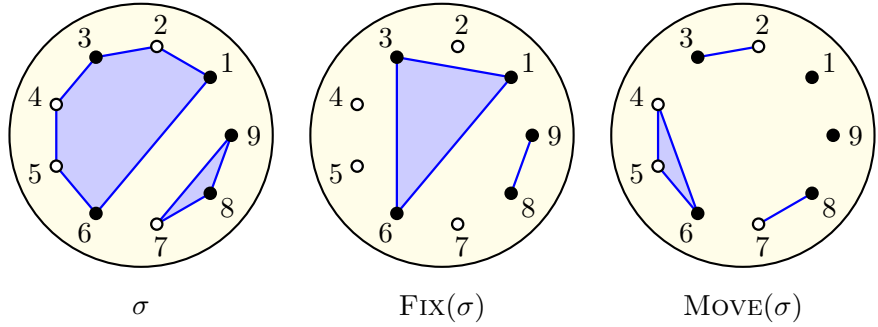
\begin{figure}\label{fig:fix-move}
  \begin{center}
    \begin{tikzpicture}
      \begin{scope}[shift={(-4,0)},BluePoly]
	\node () [disk,minimum size=3.5cm] {};
	\foreach \n in {1,...,9} {
	  \coordinate (\n) at (\n*40:1.2cm);
	  \node[black] () at (\n*40:1.5cm) {\n};
	}
	\filldraw (1)--(2)--(3)--(4)--
	(5)--(6)--cycle;
	\filldraw (7)--(8)--(9)--cycle;
	\foreach \n in {1,3,6,8,9} { \draw (\n) node [dot] {};}
	\foreach \n in {2,4,5,7} { \draw (\n) node [opendot] {};}
      \end{scope}
      
      \begin{scope}[shift={(0,0)},BluePoly]
	\node () [disk,minimum size=3.5cm] {};
	\foreach \n in {1,...,9} {
	  \coordinate (\n) at (\n*40:1.2cm);
	  \node[black] () at (\n*40:1.5cm) {\n};
	}
	\filldraw (1)--(3)--(6)--cycle;
	\draw (8)--(9);
	\foreach \n in {1,3,6,8,9} { \draw (\n) node [dot] {};}
	\foreach \n in {2,4,5,7} { \draw (\n) node [opendot] {};}
      \end{scope}
      
      \begin{scope}[shift={(4,0)},BluePoly]
	\node () [disk,minimum size=3.5cm] {};
	\foreach \n in {1,...,9} {
	  \coordinate (\n) at (\n*40:1.2cm);
	  \node[black] () at (\n*40:1.5cm) {\n};
	}
	\draw (2)--(3);
	\filldraw (4)--(5)--(6)--cycle;
	\draw (7)--(8);
	\foreach \n in {1,3,6,8,9} { \draw (\n) node [dot] {};}
	\foreach \n in {2,4,5,7} { \draw (\n) node [opendot] {};}
      \end{scope}
      
      \node () at (-4,-2.3) {$\sigma$};
      \node () at (0,-2.3) {$\fix(\sigma)$};
      \node () at (4,-2.3) {$\move(\sigma)$};
    \end{tikzpicture}
  \end{center}
  \caption{The noncrossing partitions corresponding to $\sigma$,
    $\fix^B(\sigma)$, and $\move^B(\sigma)$ as described in Example
    \ref{ex:fix-move}. The white dots form the set $B$.}
\end{figure}
\begin{exmp}\label{ex:fix-move}
	Let $\sigma = (1\ 2\ 3\ 4\ 5\ 6)(7\ 8\ 9)$ and 
	$B = \{2,4,5,7\}$. Then we have $\fix^B(\sigma) = (1\ 3\ 6)(8\ 9)$
	and $\move^B(\sigma) = (2\ 3)(4\ 5\ 6)(7\ 8)$. See Figure
	\ref{fig:fix-move}.
\end{exmp}
Although the output of $\move^B$ is less easily described than that of 
$\fix^B$, both maps satisfy many similar properties. Mirroring
Lemma~\ref{lem:fix_order}, we now describe several properties of the
$\move^B$ map. 

\begin{lem}\label{lem:move_order}
	Let $B\subseteq [n]$.
	\begin{enumerate}
	\item for each $\pi\in \NC_n(B)$, 
	$\move^B(\pi)$ is the minimal $\pi' \le \pi$ such that
	$w_{\delta_{\pi'}}(b,\cdot) = w_{\delta_\pi}(b,\cdot)$
	for each $b \in B$;\label{item:move_char}
	\item $\move^B$ preserves order;\label{item:move_order}
	\item $\move^B(\pi) \le \pi$ for all $\pi \in \NC_n$;
	\item $\move^B$ is idempotent, i.e.\ $(\move^B)^2 = \move^B$.
	\end{enumerate}
\end{lem}
\begin{proof}
	Suppose that $\delta_{\pi'}$ is a dual simple braid with 
	$\pi' \leq \pi$ and wrapping numbers which satisfy
	$w_{\delta_\pi}(b,\cdot) = w_{\delta_{\pi'}}(b,\cdot)$
	for all $b\in B$. Then $\delta_{\pi}\delta_{\pi'}^{-1}$ is a dual
	simple braid by Proposition~\ref{prop:dual-simple-order} and is 
	$B$-fix by Lemma~\ref{lem:fix_image}. Hence, by
	Lemma~\ref{lem:fix_order}, 
	$\delta_\pi\delta_{\pi'}^{-1} \leq \fix^B(\delta_\pi)$ and by 
	rearranging, we have that 
	$\fix^B(\delta_\pi)^{-1}\delta_\pi \leq \delta_{\pi'}$
	and thus
	$\move^B(\delta_\pi) \leq \delta_{\pi'}$.
	
	The second statement follows since for each $b\in B$, 
	the wrapping number $w_\beta(b,\cdot)$ 
	is monotone with respect to $\NC_n(B,\cdot)$. 
	The third and fourth statement are immediate from the first.
\end{proof}
The map $\move^B$ is multiplicative in the following sense.

\begin{lem}\label{lem:move_multiplicative}
Let $B \subseteq [n]$. Let $\beta \in \DS_n(B,B')$ and $\beta' \in \DS_n(B',\cdot)$ be such that $\beta \beta' \in \DS_n(B,\cdot)$. Then $\move^B(\beta\beta') = \move^B(\beta)\move^{B'}(\beta')$.
\end{lem}
\begin{proof}
	By Lemma~\ref{lem:move_order}\eqref{item:move_order}, we know that 
	$\move^B(\beta) \le \move^B(\beta\beta')$ and thus by 
	Proposition~\ref{prop:dual-simple-order}, $\move^B(\beta)^{-1}\move^B(\beta\beta')$
	is a dual simple braid. By Lemma~\ref{lem:wrapping_add}, this braid has
	$B'$-indexed wrapping numbers which are equal to those of $\move^{B'}(\beta')$.
	Hence, by Lemma~\ref{lem:move_order}, we know that 
	\[\move^{B'}(\beta') \le \move^B(\beta)^{-1}\move^B(\beta\beta').\]
	Equivalently, we have
	\[\move^B(\beta)\move^{B'}(\beta') \le \move^B(\beta\beta')\]
	in the partial order on $\braid_n$, and thus $\move^B(\beta)\move^{B'}(\beta')$
	is a dual simple braid. Since this braid has the same $B$-indexed wrapping numbers
	as $\move^B(\beta\beta')$, another application of Lemma~\ref{lem:move_order}\eqref{item:move_char}
	tells us that
	\[\move^B(\beta\beta') \le \move^B(\beta)\move^{B'}(\beta').\]
	Combining the above inequalities, we have 
	\[\move^B(\beta)\move^{B'}(\beta') \le \move^B(\beta\beta') \le 
	\move^B(\beta)\move^{B'}(\beta')\]
	and therefore $\move^B(\beta\beta') = \move^B(\beta)\move^{B'}(\beta')$.
\end{proof}
We now study at the structure of $\fix(\NC_n(B))$ and $\move(\NC_n(B,\cdot))$ inside $\NC_n(B,\cdot)$.

\begin{rem}[Minima and maxima]\label{rem:min-max}
	The identity braid is clearly the minimal element for
	both $\fix(\DS_n(B))$ and $\move(\DS_n(B,\cdot))$.
	Since $\fix^B$ and $\move^B$ are order-preserving maps, the maximal 
	elements are $\fix^B(\delta)$ and $\move^B(\delta)$, respectively.
\end{rem}
\begin{lem}\label{lem:fix-move-intersect}
	Let $B\subseteq [n]$. Then $\fix^B(\move^B(\delta_\pi))$ is the identity
	braid for all $\pi \in \NC_n(B,\cdot)$. In particular, the intersection of
	$\fix(\NC_n(B))$ and $\move(\NC_n(B,\cdot))$
	contains only the discrete partition.
\end{lem}
\begin{proof}
	Let $\pi \in \NC_n(B,\cdot)$ be arbitrary. Then
	\[
	\fix^B(\move^B(\delta_\pi)) = \fix^B(\fix^B(\delta_\pi)^{-1}\delta_\pi) = \fix^B(\delta_\pi)^{-1}\fix^B(\delta_\pi) = 1
	\]
	by Lemma~\ref{lem:fix_order}\eqref{item:fix_skips_fixed}.
	The second claim follows directly from the wrapping number 
	characterizations of $\fix(\NC_n(B))$ and $\move(\NC_n(B,\cdot))$
	given in Definition~\ref{def:fix} and Lemma~\ref{lem:move_unique}.
\end{proof}
We now prove the main result of this section.

\begin{prop} \label{prop:nc-decomposition}
	Let $B \subseteq [n]$. Then $\NC_n(B,\cdot)$ is isomorphic to the direct
	product of the subposets $\fix(\NC_n(B))$ and $\move(\NC_n(B,\cdot))$.
\end{prop}
\begin{proof}
	Since $\fix^B$ and $\move^B$ are order-preserving maps on $\NC_n(B,\cdot)$,
	the map which sends $\pi$ to the element
	\[
	(\fix^B(\pi),\move^B(\pi))\in \fix(\NC_n(B))\times\move(\NC_n(B,\cdot))
	\]
	is order-preserving as well. 
	
	Suppose that $\fix^B(\pi) = \fix^B(\pi')$ and 
	$\move^B(\pi) = \move^B(\pi')$. Then by definition of $\move^B$ we have
	\[
	\fix^B(\delta_\pi)^{-1}\delta_\pi =  \fix^B(\delta_{\pi'})^{-1}\delta_{\pi'}
	\]
	and thus $\delta_\pi = \delta_{\pi'}$, so the map is injective.
		
	To see surjectivity, let $\pi \in \fix(\NC_n(B))$ and
	$\pi' \in \move(\NC_n(B,\cdot))$ be arbitrary and let
	$\beta = \delta_\pi\delta_{\pi'}$.
	Note that in the partial order on $\braid_n$ obtained by extending that of
	$\NC_n$,
	\begin{align*}
		\beta &= \fix^B(\delta_\pi)\move^B(\delta_{\pi'}) \\ 
		&\le \fix^B(\delta_n)\move^B(\delta_n) \\
		&\le \delta_n
	\end{align*}
	and thus $\beta$ is a dual simple braid. Then
	\[
	\fix^B(\beta) = \fix^B(\delta_\pi)\fix^B(\delta_{\pi'}) = \fix^B(\delta_\pi) = \delta_\pi
	\]
	by Lemma~\ref{lem:fix_order}\eqref{item:fix_skips_fixed},
	Lemma~\ref{lem:fix-move-intersect} and Lemma~\ref{lem:fix_order}\eqref{item:fix_idempotent},
	showing also that $\move^B(\beta) = \delta_{\pi'}$.
	
	It remains to see that incomparable elements are mapped to incomparable
	elements. So suppose $\pi$ and $\pi'$ have the property that
	$\fix^B(\pi) \le \fix^B(\pi')$ and $\move^B(\pi) \le \move^B(\pi')$. Then
	\begin{align*}
	\delta_\pi &= \fix^B(\delta_\pi)\move^B(\delta_\pi) \\
	&\le \fix^B(\delta_{\pi'})\move^B(\delta_{\pi}) \\
	&\le \fix^B(\delta_{\pi'})\move^B(\delta_{\pi'}) \\ 
	&\le \delta_{\pi'}
	\end{align*}
	by Proposition~\ref{prop:relations}, so $\pi$ and $\pi'$ are
	comparable as well.
\end{proof}
We close by proving the following extension of Lemma~\ref{lem:fix-move-intersect}.

\begin{lem}\label{lem:move_unique}
	Let $B \subsetneq [n]$. 
	An element $\beta \in \move(\DS_n(B,\cdot))$ is uniquely determined by
	the tuple	$(w_b(\beta))_{b \in B}$.
	In particular, $\move(\DS_n(B,B'))$ contains at most one element.
\end{lem}
\begin{proof}
	Suppose $\delta_\pi$ and $\delta_{\pi'}$ are dual simple braids
	in $\DS_n(B,\cdot)$ with the property that 
	$w_{\delta_\pi}(b,\cdot) = w_{\delta_{\pi'}}(b,\cdot)$ for each 
	$b\in B$. Each wrapping number is either $0$ or $1$, and these
	can be characterized within $\pi$ and $\pi'$ by the fact that for each
	$b\in B$, either $\{b\}$ is a singleton or $b$ and $b+1$ (modulo $n$) 
	share a block. This property is preserved under common refinement, so 
	$\delta_{\pi \wedge \pi'}$ has the same $B$-indexed wrapping numbers 
	as $\delta_\pi$ and $\delta_{\pi'}$. By definition,
	$\delta_{\pi \wedge \pi'} \le \delta_\pi$ and
	$\delta_{\pi \wedge \pi'} \le \delta_{\pi'}$, 
	and we know by Lemma~\ref{lem:move_order}\eqref{item:move_order} that 
	$\move^B(\delta_{\pi \wedge \pi'}) \le \move^B(\delta_{\pi})$
	and $\move^B(\delta_{\pi \wedge \pi'}) \le \move^B(\delta_{\pi'})$.
	Finally, since all three of these braids have the same $B$-indexed wrapping 
	numbers, we may conclude by Lemma~\ref{lem:move_order}\eqref{item:move_char} that
	\[
	\move^B(\delta_{\pi})=\move^B(\delta_{\pi \wedge \pi'})=\move^B(\delta_{\pi'})
	\]
	and we are done.
\end{proof}

\section{The Complex of Boundary Braids}\label{sec:boundary_subcomplex}

In this section, we describe the subcomplex of the dual braid complex which
is determined by the set of boundary braids.

\begin{defn}[Complex of boundary braids]
Let $B \subseteq [n]$. The \emph{complex of $(B,\cdot)$-boundary braids}, denoted 
$\comp(\braid_n(B,\cdot))$, is the full subcomplex of $\comp(\braid_n)$ 
supported on $\braid_n(B,\cdot)$.
\end{defn}
The boundary strands of a $(B,\cdot)$-boundary braid 
define a path in the configuration space of 
$\card{B}$ points in $\partial P$. The following lemma is the combinatorial version of 
this statement. Recall from Definition~\ref{def:dual-braid-cplx} that we regard the 
edges of $\comp(\braid_n(B,\cdot))$ as labeled by elements of $\NC_n^*$.
Note also that the edge from $\beta \in \braid_n(B,B')$ to 
$\beta' \in \braid_n(B,B'')$ carries a label in $\NC_n(B',B'')$.

\begin{lem}\label{lem:bundle}
Let $B \subseteq [n]$. There is a surjective map
\[
\bdry^B \colon \comp(\braid_n(B,\cdot)) \to \uconf_{\card{B}}(\Gamma_n,\Orth)
\]
that takes $\beta \in \braid_n(B,B')$ to $B'$.
\end{lem}
\begin{proof}
Let $\beta \in \braid_n(B,B')$. An edge out of $\beta$ labeled $\pi \in \NC_n(B',\cdot)$ is taken to the edge of $\uconf_{\card{B}}(\Gamma_n,\Orth)$ out of $B'$ that keeps $b \in B'$ fixed if it forms a singleton block of $\pi$ and that moves it to $b+1$ if it does not. The (boundary partition) condition that $b$ and $b+1$ share a block of $\pi$ for every $b \in B'$ ensures that this is compatible with how vertices are mapped. It also ensures that if $b+1 \in B$ as well, then $b+1$ also moves, and so the edge actually exists in $\uconf_{\card{B}}(\Gamma_n,\Orth)$.

To verify surjectivity we show that for any edge $e$ from $B'$ to $B''$ in $\uconf_{\abs{B}}(\Gamma_n,\Orth)$ and any $\beta \in \braid_n(B,B')$ there does in fact exist a $\pi \in \NC_n(B',B'')$ such that the edge out of $\beta$ labeled $\pi$ is taken to $e$. If $B = B' = [n]$ this is achieved by the maximal element $\pi = \{[n]\}$. Otherwise the fact that the edge $e$ exists means that for every interval $\{i,\ldots,j\}$ (modulo $n$) of $B$ either that same interval or the interval $\{i+1,\ldots,j+1\}$ is in $B'$. The needed partition $\pi$ is the one whose non-singleton blocks are the intervals $\{i,\ldots,j+1\}$ where the second possibility happens. 
\end{proof}

Our goal is to show that $\bdry^B$ is in fact a trivial bundle. To do so, we use 
the local decomposition results from Section~\ref{sec:boundary_perms} to obtain a 
splitting. More precisely, we want to construct a map $\splt^B$ that makes the 
diagram

\begin{equation}
\label{eq:split}
\begin{tikzpicture}[xscale=3,yscale=-2,anchor=base,baseline,shift={(0,-.5cm)}]
\node (univ) at (0,0) {$\widetilde{\uconf}_{\card{B}}(\Gamma_n,\Orth)$};
\node (braid) at (2,0) {$\comp(\braid_n(B,\cdot))$};
\node (conf) at (1,1) {$\uconf_{\card{B}}(\Gamma_n,\Orth)$};
\path (braid) edge[->] node[anchor=north west] {$\bdry^B$} (conf);
\path (univ) edge[dashed,->] node[anchor=south] {$\splt^B$} (braid);
\path (univ) edge[->] node[anchor=north east] {$\cov$} (conf);
\end{tikzpicture}
\end{equation}
commute (where $\cov$ denotes the covering map).

\begin{lem}\label{lem:splitting}
The map $\splt^B$ in \eqref{eq:split} exists. It is characterized (modulo deck transformations) by the property that if an edge in its image is labeled by $\pi \in \NC_n(B',\cdot)$ then $\move^{B'}(\pi) = \pi$.
\end{lem}
\begin{proof}
We will use the shorthands $\widetilde{\uconf}$, $\uconf$ and $\comp$. Suppose there is an edge from $V'$ to $V''$ in $\widetilde{\uconf}$. Under $\cov$ it maps to an edge from $B'$ to $B''$ in $\uconf$. If $\beta'$ is a vertex above $B'$ (with respect to $\bdry^B$) then any vertex $\beta'' = \beta'\delta_\pi$ with $\pi \in \NC_n(B',B'')$ has the property that $\bdry^B(\beta'') = B''$. Our characterization states that if $\beta' = \splt^B(V')$ then $\splt^B(V'')$ should be the $\beta''$ with 
$\pi \in \move^{B'}(\NC_n(B',B''))$, which is unique by Lemma~\ref{lem:move_unique}. 
Similarly, if $\splt^B(V'')$ has already been defined, this uniquely characterizes $\splt^B(V')$.

If we choose a base vertex $V_0 \in \widetilde{\uconf}$ above $B$, declare that 
$\splt^B(V_0)$ is the vertex labeled by the identity braid, and extend the definition according to the above rule, we get a map that is defined everywhere since $\widetilde{\uconf}$ is connected (by edge paths).

It remains to see that this map is well-defined, i.e.\ that extensions along different edge paths agree. Since $\widetilde{\uconf}$ is simply connected, it suffices to check this along $2$-simplices. This amounts to the requirement that if 
$\beta \in \braid_n(B,B')$, $\delta_{\pi'} \in \move^{B'}(\DS_n(B',B''))$ and 
$\delta_{\pi''} \in \move^{B''}(\DS_n(B'',\cdot))$ are such that 
$\delta_{\pi'}\delta_{\pi''} \in \DS_n(B',\cdot)$ then 
$\move^{B'}(\delta_{\pi'}\delta_{\pi''}) = \delta_{\pi'}\delta_{\pi''}$. 
This is true by Lemma~\ref{lem:move_multiplicative}.
\end{proof}
\begin{defn}[Move complex]\label{def:move-subcomplex}
	Let $B \subseteq [n]$. We denote the image of $\splt^B$ by $\comp(\move_n(B,\cdot))$ and call it the \emph{move complex} associated to $B$. Its vertex set is denoted 
	$\move_n(B,\cdot)$.
\end{defn}
\begin{cor}\label{cor:move-cat}
	Let $B\subseteq [n]$.
	The corestriction of $\splt$ to the move complex 
	$\comp(\move_n(B,\cdot))$ is an isomorphism.
	In particular, the move complex is a $\cat(0)$ subcomplex 
	of the dual braid complex.
\end{cor}
\begin{proof}
	The corestriction of $\splt$ to $\comp(\move_n(B,\cdot))$ is a covering map by Lemma~\ref{lem:splitting}. We need to show that it is injective. To see this, recall from Proposition~\ref{prop:pt-curv} that $\widetilde{\uconf}_{\card{B}}(\Gamma_n,\Orth)$ is isomorphic to a dilated column. Let $(b_1,\ldots,b_k)$ with $0 < b_1 < \ldots < b_k <n$ be a basepoint above $B$ in the dilated column. Note that the edge from $(b_1,\ldots,b_k)$ to $(b_1 +\varepsilon_1,\ldots,b_k+\varepsilon_k)$, with $(\varepsilon_i)_i \in \{0,1\}^k$ is taken by $\splt$ to a 
	$(B,\cdot)$-boundary braid with wrapping numbers $(\varepsilon_1,\ldots,\varepsilon_k)$. It follows more generally that coordinates of the dilated column relative to the basepoint correspond to wrapping numbers. In particular, $\splt$ is injective.
\end{proof}
If we develop the image of $\fix^B$ in a similar way to how we just 
developed $\move^B$, we encounter a familiar structure.

\begin{defn}[Fix complex]\label{def:fix_complex}
	The \emph{fix complex} $\comp(\fix_n(B))$ is the full subcomplex of 
	$\comp(\braid_n)$ supported on $\fix_n(B)$. 
\end{defn}
Note that $\fix_n(B)$ is a parabolic subgroup and $\comp(\fix_n(B))$ is an isomorphic copy of 
$\comp(\braid_{n - \abs{B}})$. The fix complex does indeed relate to the decomposition of $\NC_n(B,\cdot)$ in a similar way as the move complex:

\begin{lem}\label{lem:fix_complex}
	The edges out of $\beta \in \fix_n(B)$ that lie in $\comp(\fix_n(B))$ are precisely those labeled by elements of $\fix(\NC_n(B))$. In particular, the fiber of $\bdry^B$ over $B'$ is a union over translates of $\comp(\fix_n(B'))$.
\end{lem}
\begin{proof}
	The first claim is just Lemma~\ref{lem:fix_image}. For the second claim note that the edges out of $\beta \in \braid_n(B,B')$ that are collapsed to a point are precisely those labeled by $\fix^{B'}(\NC_n(B',\cdot))$. Thus the fiber of $\bdry^B$ over $B'$ is the union over the $\beta \fix_n(B')$ for $\beta \in \braid_n(B,B')$.
\end{proof}
\begin{prop}\label{prop:fix-move-product}
Let $B, B' \subseteq [n]$ and let $\beta \in \braid_n(B,B')$ be arbitrary. There are unique braids $\fix^B(\beta) \in \fix_n(B)$ and $\move^B(\beta) \in \move_n(B,B')$ such that
\[
\beta = \fix^B(\beta)\move^B(\beta)\text{.}
\]
Moreover,
\begin{enumerate}
\item $\move^B(\move^B(\beta)) = \move^B(\beta)$\label{item:move_idempotent}
\item $\move^B(\beta)^{-1} = \move^{B'}(\beta^{-1})$\label{item:move_inverses}
\item if $\beta' \in \braid_n(B',\cdot)$ then
$\move_n(\beta\beta') = \move_n^B(\beta)\move_n^{B'}(\beta')$.\label{item:move_multiplicative}
\end{enumerate}
\end{prop}
\begin{rem}
Note that by Lemma~\ref{lem:splitting} and Lemma~\ref{lem:fix_complex} the braids $\fix^B(\beta)$ and $\move^B(\beta)$ coincide with the definitions in Section~\ref{sec:boundary_perms} if $\beta \in \DS_n(B,\cdot)$.
\end{rem}

\begin{proof}
Uniqueness amounts to the statement that $\fix_n(B) \cap \move_n(B,\cdot) = \{1\}$, which follows from Corollary~\ref{cor:move-cat}.

Let $\beta \in \braid_n(B,B')$. To define $\move^B(\beta)$, consider the diagram \eqref{eq:split} and let $V \in \widetilde{\uconf}_{\card{B}}(\Gamma_n,\Orth)$ be the base vertex with $\splt^B(V) = 1$. Let $p$ be an edge path from $1$ to $\beta$ in $\comp(\braid_n(B,\cdot))$, let $q = \bdry^B(p)$ and let $\tilde{q}$ be the path starting in $V$ and covering $q$. We take $\move^B(\beta)$ to be the endpoint of $q$. The properties \eqref{item:move_inverses} and \eqref{item:move_multiplicative} follow from the corresponding properties of paths. Commutativity of \eqref{eq:split} shows that if we did the same construction with $\beta$ replaced by $\move^B(\beta)$, we would again end up at $\move^B(\beta)$, thus proving \eqref{item:move_idempotent}.

Putting $\fix^B(\beta) = \beta \move^B(\beta)^{-1}$ it remains to verify that $\fix^B(\beta) \in \fix_n(B)$. We compute
\begin{align*}
\move(\fix^B(\beta)) &= \move^B(\beta)\move^{B'}(\move^B(\beta)^{-1})\\
&= \move^B(\beta)\move^B(\move^B(\beta))^{-1}\\
&= \move^B(\beta)\move^B(\beta)^{-1}\\
&= 1\text{.}
\end{align*}
This means that a path from $1$ to $\fix^B(\beta)$ is mapped to a null-homotopic path in the complex $\uconf_{\abs{B}}(\Gamma_n,\Orth)$. Thus $\fix^B(\beta)$ lies in the same component of the fiber of $\bdry^B$ over $B$ as $1$, which by Lemma~\ref{lem:fix_complex} is $\fix_n(B)$.
\end{proof}

\begin{lem}\label{lem:fix-product}
  If $\beta \in \braid_n(B,B')$ and $\beta'\in \braid_n(B',\cdot)$ then
  \[
  \fix^B(\beta\beta') = \fix^B(\beta)\fix^{B^\prime}(\beta')^{\move^B(\beta)^{-1}}\text{,}
  \]
  where we use the shorthand $x^y$ to mean the conjugation $y^{-1}xy$.
\end{lem}
\begin{proof}
  By Proposition~\ref{prop:fix-move-product} we have on one hand
  \[
  \beta\beta' = \fix^B(\beta\beta')\move^B(\beta\beta') = \fix^B(\beta\beta')\move^B(\beta)\move^{B'}(\beta')\text{,}
  \]
  and on the other hand
  \[
  \beta\beta' = \fix^B(\beta)\move^B(\beta)\fix^{B'}(\beta')\move^{B'}(\beta')\text{.}
  \]
  Solving for $\fix^B(\beta\beta')$ proves the claim.
\end{proof}
We are now ready to prove the main result of this article.

\begin{thm}\label{thm:boundary-product}
  Let $B\subseteq [n]$. The map 
  \begin{align*}
  \varphi:\braid_n(B,\cdot) &\to \fix_n(B)\times\move_n(B,\cdot)\\
  \beta & \mapsto (\fix^B(\beta),\move^B(\beta))
  \end{align*}
  induces an isomorphism of orthoscheme complexes
  \[
  \comp(\braid_n(B,\cdot)) \cong
  \comp(\fix_n(B)) \oprod \comp(\move_n(B,\cdot))\]
  which is, in particular, an isometry.
\end{thm}
\begin{proof}
	The map $\varphi$ is a bijection by Proposition~\ref{prop:fix-move-product}.
	
	For each $\beta \in \braid_n(B,B')$, we may restrict the domain of $\varphi$
	to the subcomplex of simplices in $\comp(\braid_n(B,\cdot))$ with minimum 
	vertex $\beta$  and the image of $\varphi$ to the subcomplex of simplices 
	in the orthoscheme product
	$\comp(\fix_n(B))\oprod \comp(\move_n(B,\cdot))$ with minimum vertex 
	labeled by
	$(\fix^B(\beta),\move^B(\beta))$. It suffices for us to show that
	this restriction of $\varphi$ is an isomorphism of orthoscheme complexes,
	and since both complexes are flag complexes, we need only check this on the
	$1$-skeleton.	
	
	The edges leaving $\beta$ are parametrized by $\NC_n(B',\cdot)$, which by
	Proposition~\ref{prop:nc-decomposition} is isomorphic to
	\begin{equation}\label{eq:left_link}
	\fix^{B'}(\NC_n(B',\cdot)) \times \move^{B'}(\NC_n(B',\cdot))\text{.}
	\end{equation}
	The edges out of $(\fix^B(\beta),\move^B(\beta))$ are parametrized by
	\begin{equation}\label{eq:right_link}
	\fix^{B}(\NC_n(B,\cdot)) \times \move^{B'}(\NC_n(B',\cdot))\text{.}
	\end{equation}
	Recall that for each $\pi \in \NC_n(B',\cdot)$, we have
	\[
	\move^B(\beta\delta_\pi) = \move^B(\beta)\move^{B'}(\delta_\pi)
	\]
	by Proposition~\ref{prop:fix-move-product} and
	\[
	\fix^B(\beta\delta_\pi) = \fix^B(\beta)\fix^{B'}(\delta_\pi)^{\move^B(\beta)^{-1}}\text{.}
	\]
	by Lemma~\ref{lem:fix-product}.	
	This shows that $\varphi$ indeed induces an isomorphism between
	the posets
	\eqref{eq:left_link} and \eqref{eq:right_link},
	namely it is the identity on the second factor and conjugation
	by $\move^B(\beta)$ on the first.
	Since these posets are isomorphic, the given restriction of $\varphi$ 
	is an isomorphism of orthoscheme complexes and by Proposition~\ref{prop:prod-oc}, 
	this isomorphism is an isometry.
\end{proof}
\begin{cor}
  Let $B\subseteq [n]$. If $\comp(\braid_{n-\card{B}})$ is $\cat(0)$ then $\comp(\braid_n(B,\cdot))$ is $\cat(0)$ as well.
\end{cor}
\begin{proof}
By Theorem~\ref{thm:boundary-product}, $\comp(\braid_n(B,\cdot))$ is isomorphic to the 
metric direct product of $\comp(\fix_n(B))$, which is isomorphic to the dual
braid complex $\comp(\braid_{n-\card{B}})$, and 
$\comp(\move_n(B,\cdot))$, which is $\cat(0)$ by Corollary~\ref{cor:move-cat}. 
The claim therefore follows from \cite[Exercise~II.1.16(2)]{brihae}.
\end{proof}
%
\section{The Groupoid of Boundary Braids}\label{sec:groupoid}

We close with a more algebraic view on the results of the last section. We refer the reader to \cite[Chapter~12]{higgins71} and \cite[Chapter~II]{dehornoy15} for basic background on groupoids.

\begin{defn}
The \emph{groupoid of boundary braids} has as objects the finite subsets
of $[n]$. The morphisms from $B$ to $B'$ are $\braid_n(B,B')$
if $\card{B} = \card{B'}$ and empty otherwise. Composition is
composition of braids.
\end{defn}
\begin{rem}
To be precise one should say that morphisms are represented
by boundary braids as a braid may at the same time be a
$(B,B')$-boundary braid and a $(C,C')$-boundary braid thus
represent two different morphisms. Since a morphism
is uniquely determined by the braid and either its source
or its target, we trust that no confusion will arise from this
imprecision.
\end{rem}
Parabolic subgroups form a subgroupoid in a trivial way.

\begin{defn}
The \emph{groupoid of fix braids} has as its objects the finite
subsets of $[n]$. The morphisms from $B$ to $B'$ are
$\fix_n(B)$ if $B = B'$ and are empty otherwise.
\end{defn}
The groupoid of fix braids is normal in the following sense.

\begin{lem}
If $\beta \in \braid_n(B,B')$ and $\beta' \in \fix_n(B)$ then $\beta^{-1}\beta'\beta \in \fix_n(B')$.
\end{lem}
\begin{proof}
Note that if $\beta$ is a $(b,b')$-boundary braid then $w_\beta(b) = -w_{\beta^{-1}}(b')$. The claim now follows from Lemma~\ref{lem:wrapping_add} and Lemma~\ref{lem:wrapping_0}.
\end{proof}
The corresponding quotient morphism is the map $\bdry^B$ from Lemma~\ref{lem:bundle} that takes a $(B,B')$-boundary braid to a $(B,B')$-path in the fundamental groupoid of $\uconf_{\card{B}}(\Gamma_n,\Orth)$. The upshot of the last section is that this map splits with image move braids.

\begin{defn}
The \emph{groupoid of move braids} has objects finite
subsets of $[n]$. The morphisms from $B$ to $B'$ are
the braids $\move_n(B,B')$, which are images under
$\splt^B$ of $(B,B')$-paths.
\end{defn}
It follows from Proposition~\ref{prop:fix-move-product}\eqref{item:move_inverses} and~\eqref{item:move_multiplicative} that this is indeed a subgroupoid. Now the algebraic conclusion can be formulated as follows, see \cite[Section~4]{witzel} for a discussion of semidirect products.

\begin{thm}
The groupoid $\braid_n(\cdot,\cdot)$ is a semidirect product
\[
\fix_n(\cdot) \rtimes \move_n(\cdot,\cdot)\text{.}
\]
Specifically,
\begin{enumerate}
\item every $\beta \in \braid_n(B,B')$ decomposes uniquely as $\beta = \varphi \mu$ with $\varphi \in \fix_n(B)$ and $\mu \in \move_n(B,B')$;
\item if $\varphi \mu = \mu' \varphi'$ with $\mu,\mu' \in \move_n(B,B')$, $\varphi \in \fix_n(B)$ and $\varphi' \in \fix_n(B')$ then $\mu = \mu'$.
\end{enumerate}
\end{thm}
\begin{proof}
The first statement is Proposition~\ref{prop:fix-move-product}. For the second note that $\bdry^B(\varphi \mu) = \bdry^B(\mu'\varphi')$ since elements of $\fix_n(\cdot)$ are mapped trivially under $\bdry^B$. It follows that $\move^B(\mu'\varphi') = \move^B(\varphi\mu) = \mu$.
\end{proof}

\newcommand{\etalchar}[1]{$^{#1}$}
\providecommand{\bysame}{\leavevmode\hbox to3em{\hrulefill}\thinspace}
\providecommand{\MR}{\relax\ifhmode\unskip\space\fi MR }
\providecommand{\MRhref}[2]{%
  \href{http://www.ams.org/mathscinet-getitem?mr=#1}{#2}
}
\providecommand{\href}[2]{#2}

\end{document}